%% file: interpolation.tex
\documentclass[a4paper]{amsart}
\pdfoutput=1

\usepackage[utf8x]{inputenc}
\usepackage{graphicx}
\usepackage{epstopdf}
\usepackage{float}
\usepackage[center]{subfigure}
\usepackage{color}
\usepackage{amsmath,amstext,amssymb,stmaryrd,amsthm,a4wide}
\numberwithin{figure}{section}
\numberwithin{table}{section}
\usepackage{algorithm}
\usepackage{algorithmic}
\usepackage{xspace}
\usepackage{esvect}
\usepackage{nicefrac}
\usepackage{url}
\usepackage{authblk}
\usepackage{amsaddr}

\allowdisplaybreaks
\def\RR{\mathbb{R}}

\def\P{\mathcal{P}}
\def\L{\mathcal{L}}
\def\Lt{\widetilde{\mathcal{L}}}

\def\T{\mathcal{T}}
\def\Tt{\mathcal{\widetilde{T}}}
\def\W{\mathcal{W}}

\DeclareMathOperator*{\argmin}{argmin}

\newtheorem{remark}{Remark}
\newtheorem{lemma}{Lemma}

\newcommand{\ip}{\textsc{IPOLY}\xspace}   
\newcommand{\ax}{\textsc{LSQP}\xspace} 
\newcommand{\fem}{\textsc{IFEM}\xspace}

\newcommand{\axop}{\textsc{LSQP OPT}\xspace}
\newcommand{\cons}{\textsc{CONS}\xspace}
\newcommand{\TP}{\ensuremath{^\top}}

\begin{document}

\title{A two-scale approach for efficient on-the-fly operator
assembly in massively parallel high performance multigrid codes} 
\author{Simon Bauer$^\MakeLowercase{a}$, Marcus Mohr$^\MakeLowercase{a}$, Ulrich R\"ude$^\MakeLowercase{b}$, 
Jens Weism\"uller$^\MakeLowercase{a,e}$,
Markus Wittmann$^\MakeLowercase{c}$, Barbara Wohlmuth$^\MakeLowercase{d}$}
\maketitle
\vspace*{-1.0em}
\begin{center}
\begin{scriptsize}
$^a$ Dept.~of Earth and Environmental Sciences, Ludwig-Maximilians-Universit{\"a}t M{\"u}nchen, 80333 M\"unchen, Germany\\
$^b$ Dept.~of Computer Science 10, FAU Erlangen-N{\"u}rnberg, 91058 Erlangen, Germany\\
$^c$ Erlangen Regional Computing Center (RRZE), FAU Erlangen-N{\"u}rnberg, 91058 Erlangen, Germany\\
$^d$ Institute for Numerical Mathematics (M2), Technische Universit{\"a}t M{\"u}nchen, 85748 Garching b.~M\"unchen, Germany\\[-1ex]
$^e$ Leibniz Supercomputing Center (LRZ), 85748 Garching b.~M\"unchen, Germany
\end{scriptsize}
\end{center}

\input{0_abstract}

\input{1_intro}

\input{2_interpolation_approach}

\input{3_twoscale}

\input{4_numaccuracy}

\input{5_kernel-optimization}

\input{6_scaling-results}

\input{7_outlook}

\input{acknowledgements}

%%%%%%%%%%%%%%%%%%%%%%%%%%%%%%%%%%%%%%%%%%%%%%%%%%%%%%%%%%%%%%%%%%%%%%%%%%%%%%%
\bibliographystyle{abbrv}
\bibliography{literature}
%%%%%%%%%%%%%%%%%%%%%%%%%%%%%%%%%%%%%%%%%%%%%%%%%%%%%%%%%%%%%%%%%%%%%%%%%%%%%%%

\end{document}

%% file: 0_abstract.tex
% \svnid{$Id: 0_abstract.tex 3198 2016-08-09 14:06:42Z simon.bauer $}
\begin{abstract}
Matrix-free finite element implementations of massively parallel geometric
multigrid 
save memory and are often significantly faster than
implementations using
classical sparse matrix techniques. 
They are especially well suited for 
hierarchical hybrid grids on polyhedral domains.
In the case of constant coefficients all fine grid
node stencils in the interior of a coarse macro element are equal.
However, for non-polyhedral domains the situation changes. 
Then even for the Laplace operator, the non-linear element mapping leads to fine
grid stencils that can vary from grid point to grid point.
This observation motivates a new two-scale approach that 
exploits a piecewise polynomial approximation of the fine grid operator with
respect to the coarse mesh size.
The low-cost evaluation of these surrogate polynomials
results in an efficient 
stencil assembly on-the-fly for non-polyhedral domains
that can be significantly more efficient than 
matrix-free techniques that are based on an element-wise assembly.
The performance analysis
and additional hardware-aware code optimizations
are based on the Execution-Cache-Memory model. 
Several aspects such as two-scale a priori error
bounds and double discretization techniques are presented.  
Weak and strong scaling results illustrate the benefits of the new technique
when used within large scale PDE solvers.
\end{abstract}

%% file: 1_intro.tex
% \svnid{$Id: 1_intro.tex 3209 2016-08-11 12:07:52Z simon.bauer $}
\section{Introduction}
\label{sec:intro}
In many real-life applications, e.g., from structural mechanics, fluid dynamics, or geophysics, a simulation method may have to accommodate curved boundaries of the computational domain, 
see e.g.\ \cite{CHB09,BTT13,Kennett:2008:CUP,Igel:2016:CUP}. 
In particular, one of the grand challenges in geophysics are Earth 
mantle simulations where the domain is a thick spherical shell. 
Here, patch-wise defined isogeometric elements can achieve a high
accuracy per degree of freedom. 
However, to date no high performance implementation exists
% for complex three dimensional geometries
on peta-scale supercomputer systems.

Classical low order finite elements
are well-established also as highly efficient parallel implementations.
Nevertheless, large scale applications often require a careful, memory aware realization.
For instance, the global resolution of the Earth
mantle with 1$\,$km already leads to a mesh with $10^{12}$ nodes.
Then the solution vector
% 100$\,$TB to store all stencil weights 
for a scalar PDE operator requires 8 TByte of memory.
Such a large vector still fits well in the aggregate memory
of current peta-scale supercomputers such as SuperMUC that we
will use for our scalability experiments in Sec.\ \ref{sec:scaling}.

This becomes different, when we consider storing the stiffness matrix.
For linear tetrahedral elements on average  15 non-zeros are generated per matrix row.
Then, even in classical compressed row sparse (CRS) matrix format \cite{Barrett:1994:SIAM},
such a matrix of dimension $10^{12}$ will require 240 TByte of storage
and this may drive even large computers beyond their limits.
While developing
more efficient sparse matrix techniques is an active field of research,
see e.g.\ \cite{Kreutzer:2014:SISC, Guo01022016, Dalton:2015:OSM:2835205.2699470},
pre-computing and storing the entries of the %even of a sparse 
stiffness matrix inevitably requires a vast amount of
memory and in turn produces massive memory traffic in the
solution phase of the algebraic system, \cite{2015_BiFaGrOlSc}.
We point out that transferring large amounts of data from memory
to the CPU is considered to be a critical cost factor, e.g.\ in terms
of energy consumption.
We also note that memory caching as well as spatial and temporal blocking
techniques can be used to  alleviate this problem \cite{douglas2000cache,kowarschik2002data}, 
but that their benefit is also limited.
% This is not always feasible and will possibly  exceed the memory 
% capabilities
% even of modern supercomputers.

To avoid the bottleneck of storing huge matrices, so-called matrix-free approaches 
\cite{Rietbergen1996computational,Bergen2004,
arbenz2008scalable,Kronbichler:2012:CAF,May:2015:CMAME}
must be employed when possible.
Here all matrix entries are re-computed
on-the-fly when needed.
We refer to \cite{baker2012scaling,baker-yang-2016,bastian2008generic,Bastian2014,falgout2002hypre,notay2015massively} 
for state of the art massively
parallel solvers for PDEs and to 
\cite{Bielak:2005:CMES,burstedde-stadler-alisic-wilcox-tan-gurnis-ghattas_2013,Rudi2015}
for large scale applications in geophysics.

For polyhedral domains, matrix-free methods
can exploit the structure of uniformly refined
mesh sequences.
Here simple on-the-fly assembly routines
can be developed
that  make use of the similarity of the finite elements in the refinement hierarchy.
In this case, the performance depends essentially on
 exploiting the redundancy in the mesh structure.
However, the situation is more complex for non-polyhedral domains.
% Then the resolution must be adapted to the mesh-size and the finite element order. 
% A standard approach is to use isoparametric or even superparametric element 
% mappings, i.e., the geometry
% resolution is at least of the same order as the finite element space.
A na\"{i}ve approach which ignores the physically correct domain
and resolves the geometry only with respect to a coarser mesh 
will result in a reduced accuracy.
The cheapest version, using
a fixed non-exact resolution of the geometry will, in general, 
result asymptotically in a loss of convergence.
% Unfortuntaley,
% a mathematically sound approach 
% that maps each refinement level to the physical node coordinates 
% with respect to the physical domain is 
% will lead to a quite expensive
% on-the-fly assembly computation.

Resolving the geometry by nodal low order interpolation 
and using affine element mappings, gives optimal order results for linear finite elements.
However, resolving the geometry more and more accurately with
increasing refinement level 
% by moving the newly generated boundary nodes to the physically correct boundary 
is computationally expensive,
since now  the refinement process generates new non-similar elements
and thus the local assembly process cannot profit to the full extent from 
the uniform refinement structure. 
These fundamental observations
are the reason why current matrix-free codes reach significantly better performance 
on polyhedral domains than in more general geometries.

The main contribution of this article is a novel technique that improves
the performance of matrix-free techniques in the case of 
complex geometries and curved boundaries without sacrificing accuracy.
To achieve this, we will propose and analyze
a new two-scale combination of classical interpolation 
with finite element approximations.

The outline of the paper is as follows: We start with the formulation
of the model problem and a brief discussion of hybrid hierarchical
grids in Sec.\ 2.
In Sec.\ 3, we introduce the new two-scale method based
on a higher order polynomial approximation of the low order
finite element stencil presentation.
We discuss two different variants using standard interpolation
and least square fitting techniques and comment on a priori estimates.
In Sec.\ 4, we focus on a study of the numerical accuracy and cost considerations.
The influence of the  polynomial degree in the stencil approximation
on the accuracy as well as the effect of a double discretization
\cite{Brandt2011} used for residual computation and smoother application is illustrated.
Sec.\ 5 is devoted to the kernel optimization as it is essential for large scale 
high performance computations.
The analysis is based on an enhanced Execution-Cache-Memory (ECM) model.
Finally in Sec.\ 6, 
we present weak and strong scaling results on a current peta-scale architecture.

%% file: 2_interpolation_approach.tex
% \svnid{$Id: 2_interpolation_approach.tex 3211 2016-08-12 09:55:39Z simon.bauer $}
\section{Model problem and hybrid hierarchical meshes}
\label{sec:model}

Matrix-free implementations of finite element schemes for partial differential equations are attractive on modern heterogeneous architectures. However, special care is required in the case of PDEs on domains with 
curvilinear boundary surfaces or interfaces.
Here we propose a novel two-scale scheme that is perfectly suited to
large scale matrix-free geometric multigrid solvers for low order
conforming finite element discretizations.

\subsection{Model problem}
We consider the Laplace equation with homogeneous Dirichlet boundary conditions on a spherical shell 
as our model problem 
\begin{equation}
\label{eq:model_eq}
-\Delta u = f\quad\text{in }\Omega\quad\text{and}\quad u = 0\quad \text{on } \partial\Omega
\end{equation}
where $\Omega := \{(x,y,z) \in \RR^3 : r_1^2 < x^2 + y^2 + z^2 < r_2^2\}$ with $0 < r_1 < r_2$ and
$f \in L^2(\Omega)$.
But we point out, that our approach
can be also applied  for more complex systems of PDEs,  variable coefficients PDEs and general geometries. 

The standard Galerkin formulation of \eqref{eq:model_eq} reads as: find $u \in H_0^1(\Omega)$ which satisfies
\begin{equation}\label{eq:model_galerkin}
a(u,v) = (f,v) \qquad \forall v \in H_0^1(\Omega) 
\end{equation}
where the bilinear and linear forms are given by
\begin{equation*}
a(u,v) := \langle \nabla u, \nabla v \rangle_{\Omega}
\enspace,\quad
(f,v) := \langle f,v \rangle_{\Omega},
\end{equation*}
for all $u,v \in H^1_0(\Omega)$.

\subsection{Domain and finite element approximations}
If a  polyhedral domain is considered, the input mesh can be selected
such that the domain is equal to the union of all elements of the
mesh. Thus, the domain can be represented exactly.
However, for our model domain of a spherical shell this is
not possible when using tetrahedra that are affinely mapped from the
reference tetrahedron $\hat T$ with nodes $(0,0,0)$, $(1,0,0)$, $(0,1,0)$
and $(0,0,1)$. We use two different sequences of triangulations.
Both of them are based on a possibly unstructured simplicial initial mesh
$\T_{-2}$ of meshsize $H$ defining $\overline \Omega_H := \cup_{T \in \T_{-2}} \overline T$.
We start by constructing a hierarchy of uniformly refined triangulations
$\T := \{\T_\ell,\, \ell=0,1,\dots,L\}$ based on successive regular refinement
\cite{Bey95}.

\begin{figure}\centering
\begin{tabular}{cc}
\includegraphics[width=0.2\textwidth]%
{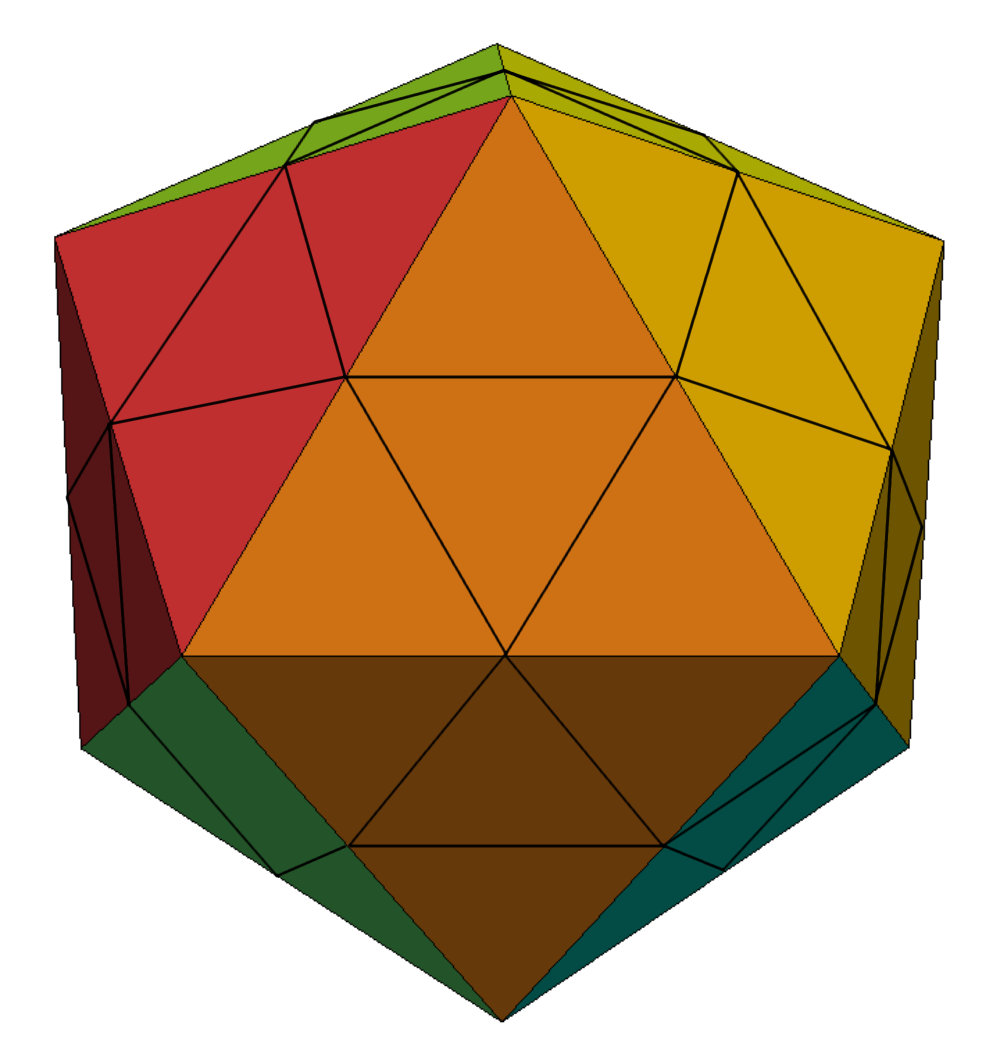} &
\includegraphics[width=0.2\textwidth]%
{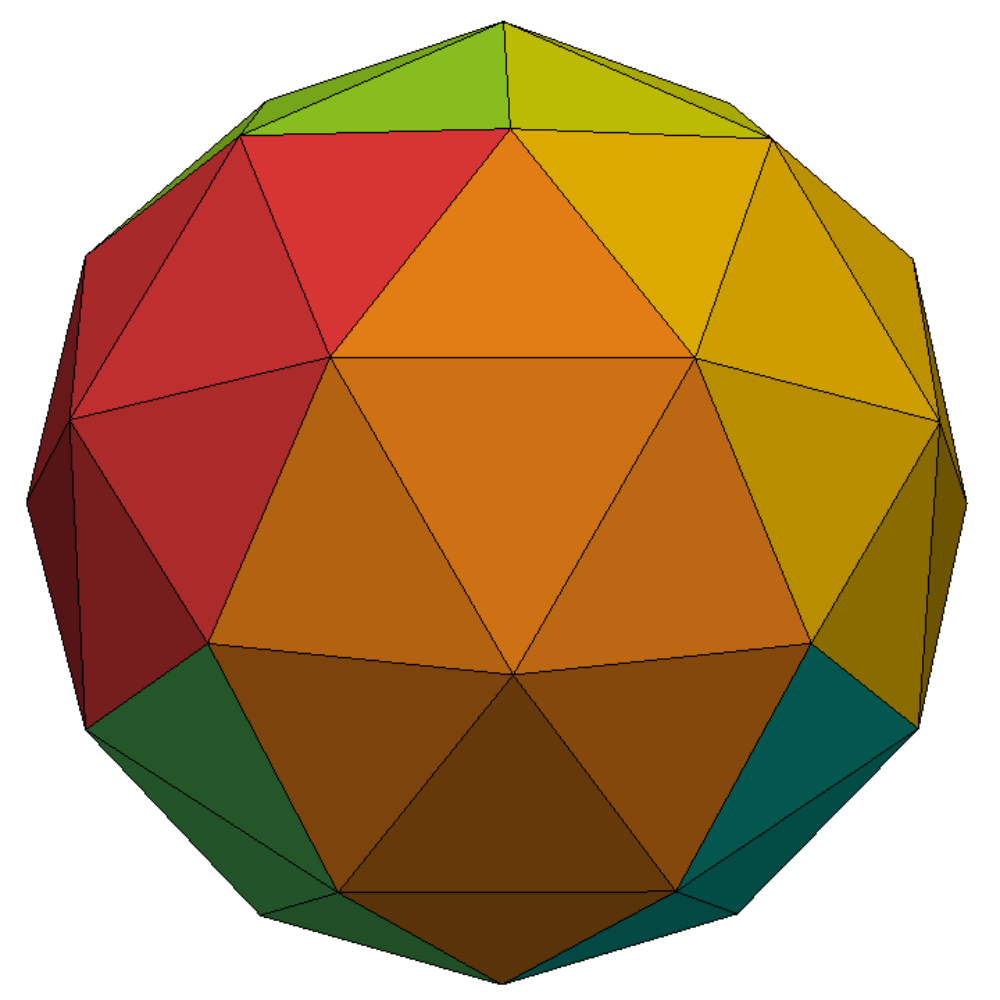} \\
\end{tabular}
\caption{Approximation of spherical shell with $\T_{-1}$ (left) and $\tilde{\T}_{-1}$ (right).
\label{fig:domainApproximation}}
\end{figure}

We refer to the elements of $\T_{-2}$ as \emph{macro elements} and point out
that our notation guarantees that each macro element has at least one inner
node for every mesh in the sequence $\T $. The level mesh-size is given by
$h_\ell = 2^{-(\ell+2)} H$.
For each mesh $\T_\ell \in \T$, we find
$\overline\Omega_H =\cup_{T \in \T_{\ell}} \overline T$.
Consequently, the quality of the geometry approximation is fixed and does not
improve with increasing refinement level. Therefore in a subsequent step,
we project all refined nodes onto the spherical surface.

\begin{figure*}\centering
\includegraphics[width=0.9\textwidth]%
{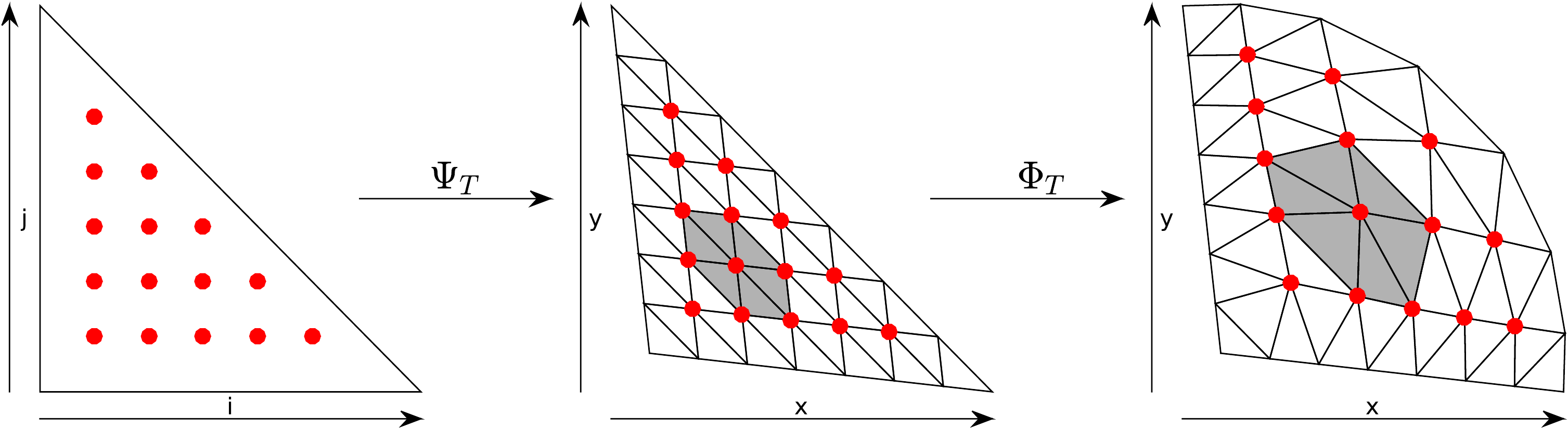}
\caption{Sketch of  the different macro element mappings. 
In the reference domain (left) nodes are associated with the discrete indices
$(i,j,k)$. Here $k$ is fixed to get a 2D clipping and we call it an index plane. Firstly the reference domain is affinely mapped by  $\Psi_T$ 
to the physical coordinates. This is the standard %classical
HHG domain (center). Secondly,
the macro element $T$ is mapped by $\Phi_T$ to the projected  element (right).
\label{fig:ref_hhg_projected_domain}}
\end{figure*}

This is schematically illustrated for one ma\-cro element by a  2D clipping in
Fig.~\ref{fig:ref_hhg_projected_domain}. 
Firstly, the reference element $\hat T$ is affinely mapped by $\Psi_T$
onto a macro element $T$ in $\T_{-2}$. Secondly, $T$ is mapped by
a blending function $\Phi_T$ reflecting the domain of interest. In case of 
a polyhedral domain, $\Phi_T$ can be selected as identity and in more general
but piecewise smooth settings, we have $\|(\Phi_T - {\rm Id}) \circ \Psi_T\|_\infty = {\mathcal O} (H^2)$.
We refer to  \cite{FMW05} for an explicit proof in 2D if $\Phi_T$ is defined by a smooth surface representation. Note that in our case the mapping $\Phi_T$ is also done 
for the interior elements of $\Tt_{-2}$, such that they form spherical layers.

Now, we define the global mapping $\Phi$ elementwise by $\Phi|_T := \Phi_T$ and assume that $\Phi$ is globally continuous and $\Phi(\Omega_H) = \Omega$.
In terms of $\Phi $, we obtain a second family $\Tt := \{\Tt_\ell\}_{_\ell}$
of simplicial triangulations as follows:  The vertices   of $\Tt_\ell$ 
are obtained by an application of $\Phi $ on the vertices of $\T_\ell$,
 and the edge graphs are topologically equivalent to the ones of
$\T$, i.e., the node connectivity remains unchanged. Under the assumption that  $\Phi$  is a piecewise smooth bijection, $\Tt$ satisfies a uniform shape regularity and standard a priori finite element estimates holds. Moreover the  vertices of
$\Tt_\ell $ are also vertices of $\Tt_{\ell+1}$.
However the midpoints of the edges on level $\ell$ are, in general, not the new vertices on level $\ell+1$. Consequently the  arising sub-elements are not similar to one of the three sub-element classes which occur, if one tetrahedron is uniformly refined, see also Fig. \ref{fig:uniform}.

Typically, low order finite element approximations are then based on the mesh sequence $\Tt$ and not on $\T$. 
The  standard finite element space is defined as
\begin{equation*}
V_\ell = \{ v \in H_0^1(\Omega_\ell) :  v|_T \in P_1(T),~\forall \, T \in \Tt_\ell \},
\end{equation*}
where $\overline \Omega_\ell := \cup_{T \in \Tt_{\ell}} \overline T$. We point out  that $\Omega_\ell$ forms a sequence of domains approximating $\Omega$ and that due to the domain approximation $V_\ell \not\subset V_{\ell+1} $. The
 discrete Galerkin formulation of \eqref{eq:model_galerkin} then  reads as: find $u_\ell \in V_\ell$ that satisfies
\begin{equation}
a_\ell(u_\ell,v_\ell) = \langle f,v_\ell \rangle_{\Omega_\ell} \qquad \forall v_\ell \in{V}_\ell 
\end{equation}
where  $a_{\ell}(\cdot,\cdot) := \langle \nabla \cdot ,
\nabla \cdot  \rangle_{\Omega_\ell}$.
Let $\{\phi_I\}_I$ be the nodal basis of $V_\ell$; for simplicity  we ignore in the notation of $\phi_I$ the level index $\ell$, if there is no ambiguity.   
Then this transforms to the algebraic system 
\begin{equation}
\label{eq:algebraic}
\L_\ell \mathbf{u}_\ell = \mathbf{f}_\ell
\end{equation}
where $\mathbf{u}_\ell$ denotes the vector of nodal values of the finite 
element approximations, i.e., $u_\ell = \sum_I (\mathbf{u}_\ell)_I \phi_I$,
and $ \mathbf{f}_\ell$ stands for the algebraic representation of the 
linear form
$(\mathbf{f}_\ell)_I = \langle f, \phi_I\rangle_{\Omega_\ell}$.
As it is standard, the entries of $\left(\L_\ell\right)_{I,J}$ are
given by $a_\ell(\phi_I,\phi_J)$. The standard
matrix formulation \eqref{eq:algebraic} can be related to a representation
of $\L_\ell$ in stencil form, i.e., each line of the matrix is associated with a node $I$, and its non-zero entries form the weights of the associated node stencil $s^I$.

\subsection{Hybrid hierarchical mesh stencils}   
Massively parallel and highly scalable PDE software frameworks depend on a sophisticated data structure based on ghost layers. To reduce communication and memory traffic, the abstract concept of different hierarchical primitives can be applied. Here we use the hybrid hierarchical grids (HHG) framework which provides an excellent
compromise between flexibility and performance \cite{Bergen2004,Bergen:2005:PhD,Gmeiner2015}. While the coarse mesh
can be completely unstructured, the fine mesh
is obtained by uniform refinement following the rules given in
\cite{Bey95}. Each element is refined into eight sub-elements each of which
belongs to one of three different sub-classes, see the left of
Fig. \ref{fig:uniform}. All nodes $I$ of the triangulation can then  be grouped into vertex, edge, face or volume primitives
depending on their position within one macro element. The vertex, edge and face primitives form the lower primitive classes.
 Such a structured block
refinement then guarantees in
the case of affine element mappings that the node stencils for partial
differential equations with constant coefficients do not depend on the
node location within one primitive. More precisely for each macro
element, all volume primitive stencils are exactly the same and are
given by a symmetric 15-point
stencil, 
see the right of Fig. \ref{fig:uniform}.

\begin{figure}[ht]
\centerline{
\includegraphics[width=0.24\textwidth]{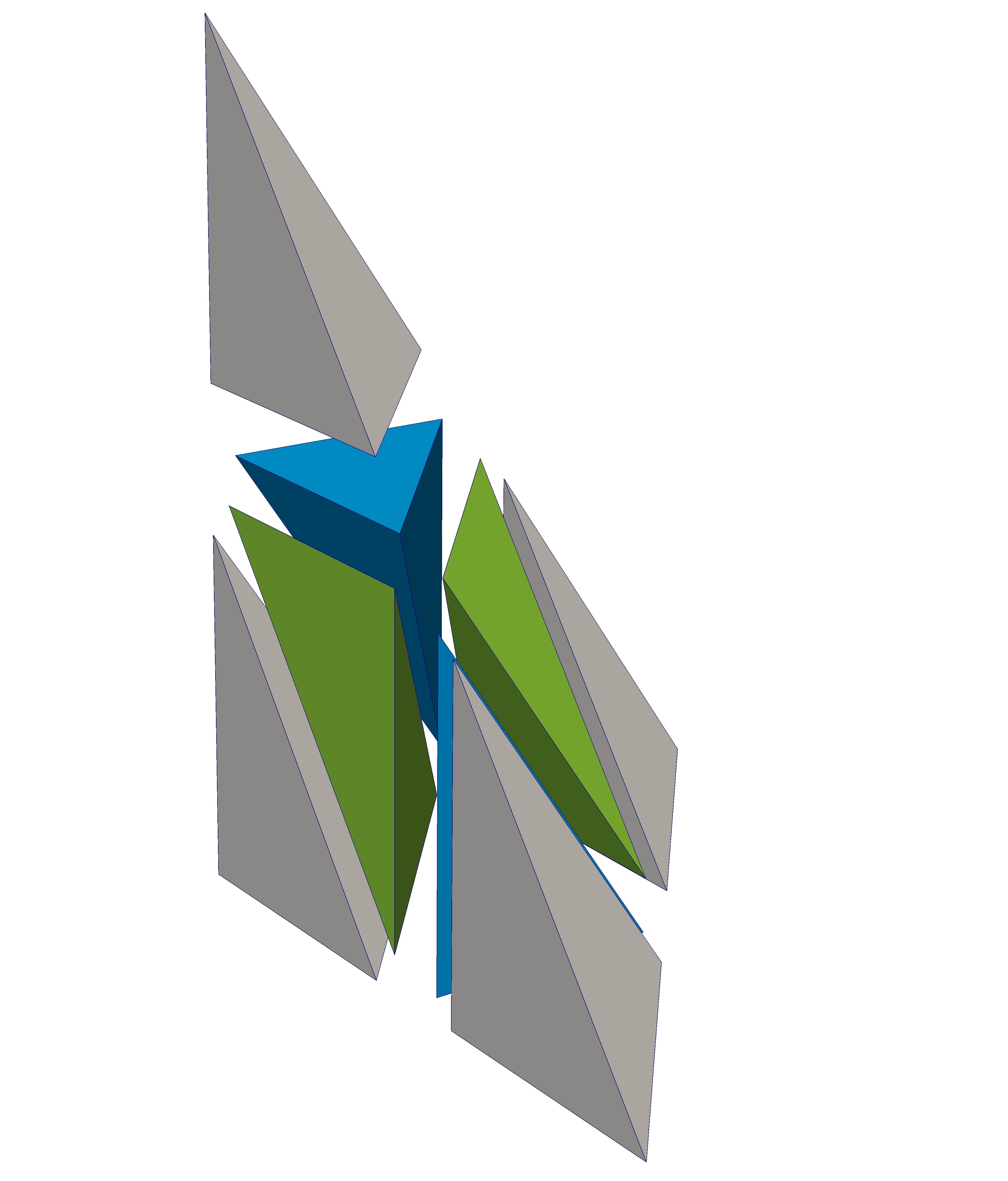} 
\includegraphics[width=0.24\textwidth]{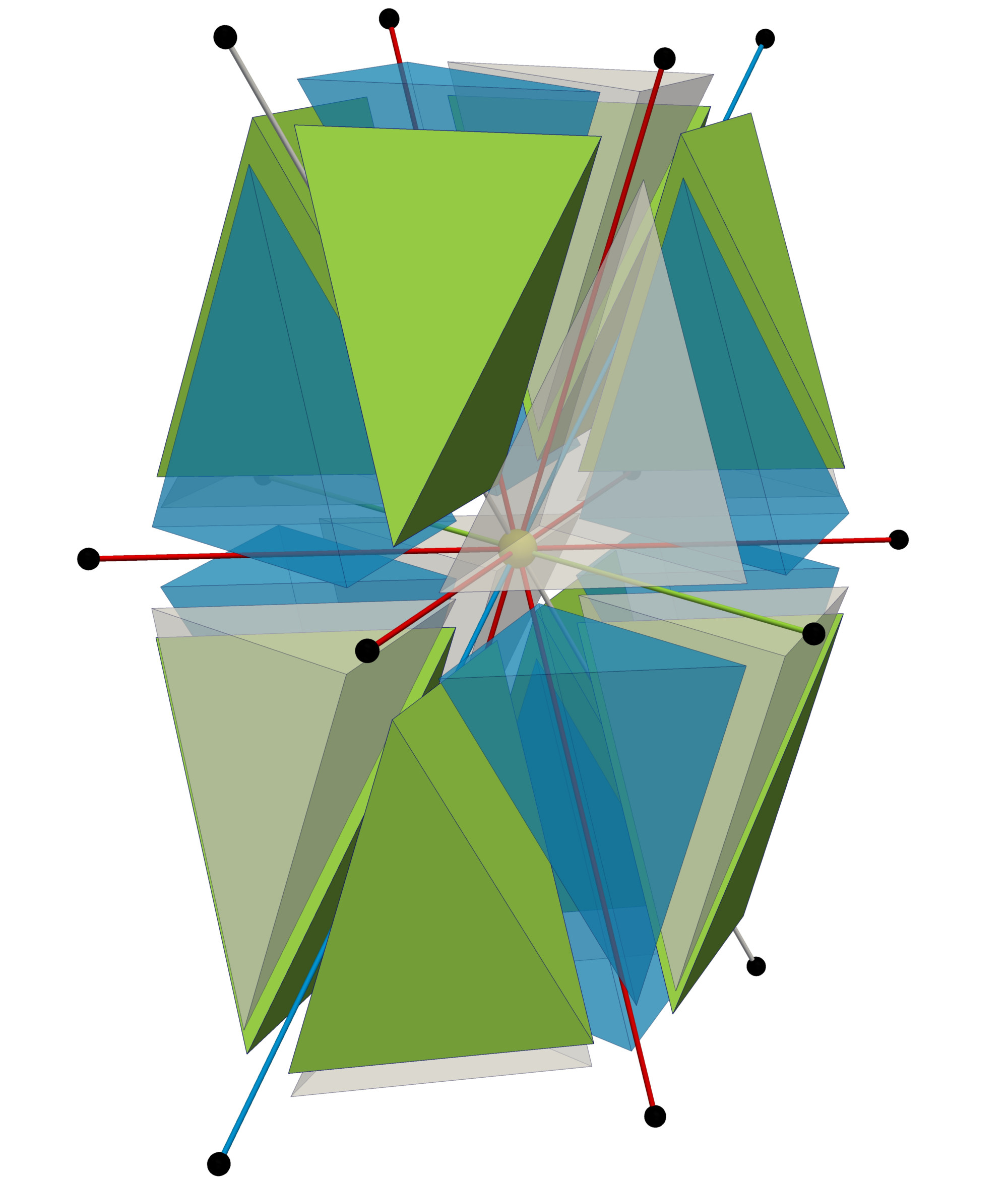} }
\caption{15-point stencil (right) and the three classes of sub-tetrahedra (left). \label{fig:uniform}}
\end{figure}

 In large scale applications, the number of nodes in the volume primitive containers is significantly higher than the number of nodes in the lower primitive classes. It grows with
${\mathcal O} (2^{3\ell})$ whereas the  number of remaining nodes increases only
with ${\mathcal O} (2^{2\ell})$. Since we are interested in large scale
simulations, the computational cost associated with the node stencils on the lower primitives is rather small compared to the cost associated with the nodes in the volume primitives. Thus we 
 compute on-the-fly the nodes stencils for nodes in a lower primitive class
 in a straightforward way
by reassembling the contributions of local stiffness matrices. Only the on-the-fly computation of the node stencils for nodes in a volume primitive class will be replaced by a surrogate stencil operator.
We recall that the volume primitives are associated with the macro elements.
 Let $T \in \T_{-2}$ be fixed, then we denote by
 $\mathcal{M}_\ell$ the set of all interior points of $T$. These nodes are within the volume primitive associated with $T$, at refinement level
$0 \leqslant \ell \leqslant L$. For $\ell = 0$ this consists of exactly one
single point. The number $n_\ell$
of elements in $\mathcal{M}_\ell$ is given by
\begin{equation}
n_\ell = \frac {\left(2^{\ell+2} - 3\right) \left(2^{\ell+2} - 2\right)
 \left(2^{\ell+2} - 1\right)}{6} .
\label{eqn:numInteriorPoints}
\end{equation}

\begin{remark}
We point out that all our techniques can also be applied to the construction of
stencil entries associated with face and edge nodes. This is of interest in case
 of    moderate level hierarchies. 
\end{remark}

For a fixed node $I$, the stencil weight for accessing another node $J$ is
given by $a_\ell(\phi_I,\phi_J)$ and will only be non-zero, if both nodes are
connected by an edge in $\T_\ell$.
In case of our %the described
uniform mesh refinement, the stencil weights can be
identified by their cardinal directions $w \in \mathcal{W} = \{be,bc,...,mc\}$,
where, as noted before, we have $|\W|=15$.
The first character in $w$ denotes the bottom, middle and top plane whereas the
second specifies the orientation within the plane, as illustrated in
Fig.~\ref{fig:stencilHHG}.  
Then the stencil for node $I$ can be represented by a vector % $s^I\in\RR^{15}$
%with
$s^I=(a_\ell(\phi_I,\phi_{J_{be}}),$ \ldots ,$a_\ell(\phi_I,\phi_{J_{mc}}))$ 
in $\RR^{15}$.
We note that $J_{mc} =I$.

\begin{figure}[t]
\centering
\includegraphics[trim=178pt 145pt 759pt 180pt, clip=true, width=0.4\textwidth]{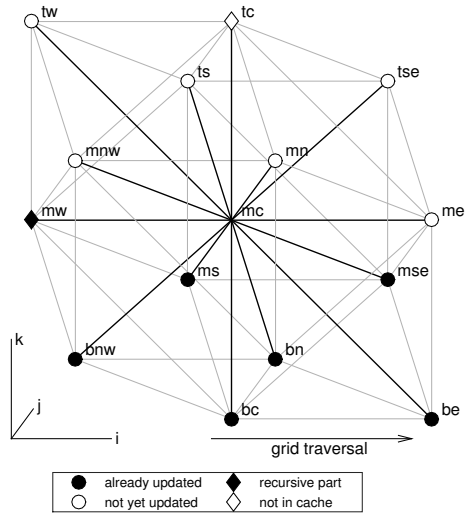} 
\caption{Layout of a typical low order 15pt stencil; see also
\cite{Bergen:2005:PhD}; node symbols will be relevant in Sec.~\ref{sec:kernel_opt}. \label{fig:stencilHHG}}
\end{figure}

Moreover there exists a unique mapping between the  nodes in ${\mathcal M}_\ell$ and the
set of index triples ${\mathcal I}_\ell $ given by
\begin{equation*}
\{(i, j, k); 0 < i,j,k < 2^{\ell+2} - 2, i + j + k < 2^{\ell+2} \}.
\end{equation*}
Hence the node stencils on level $\ell$ can be associated with a stencil
function $s: {\mathcal I}_\ell \rightarrow \RR^{15}$, i.e.~$s^I = s(i_I,j_I,k_I)$ where
$(i_I, j_I,k_I)$ is the index triple associated with node $I$.
If $\Phi = \text{Id} $ then all $s^I $, $I \in {\mathcal M}_\ell $ are the same, 
i.e., $s$ is a constant function. Moreover it is obtained by simple
scaling from a reference stencil $\hat s$, i.e., $s^I = 2^{-\ell} \hat s
$, $I \in {\mathcal M}_\ell$,
with $\hat s $ being the stencil associated with the single node in ${\mathcal M}_0$.
This is not the case for our model domain. Here
$s$ is a non-constant function
and consequently all stencil weights have to be  computed
on-the-fly. To do so for each node $I \in {\mathcal M}_\ell$, we have to
consider the 24 tetrahedra adjacent to it and their associated local stiffness
matrices which can be different from element to element due to a general $\Phi$.
Thus this process, although linear with respect to $n_\ell$, can be rather cost
intense.
For $L=6$ we have to compute on-the-fly $2,731,135$ different stencils per
macro element.

\begin{figure}[ht]
\centering
\includegraphics[width=0.25\textwidth]{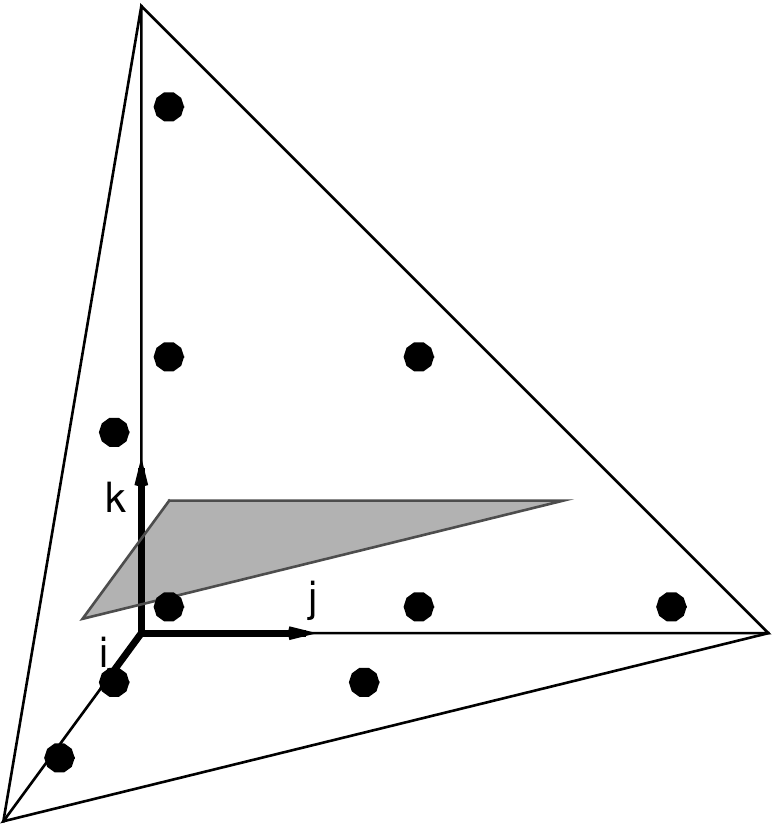} 
\caption{Uniformly refined reference tetrahedron. Highlighted plane shows the
location of nodal centers for the stencils visualized in
Fig.~\ref{fig:stencilCCvsVC} and Fig.~\ref{fig:stencilVCvsINTvsAPPROX}.
Black dots mark the ten sampling points for \ip (Sec.~\ref{subsec:ipoly}).%
\label{fig:tetplane}}
\end{figure}

In Fig.~\ref{fig:stencilCCvsVC}, we illustrate the influence of $\Phi$ on the
node stencil restricted to an index plane, see Fig. \ref{fig:tetplane}.
Note that for $\Phi = \text{Id}$ not only the stiffness matrix but
also the stencil is symmetric, i.e., opposite
stencil entries, e.g., \textit{tw} and \textit{be}, are identical.
While this kind of symmetry is not completely preserved in the more general case, it
is still reflected. Hence,
we show only seven stencil entries plus the central one in Fig.~\ref{fig:stencilCCvsVC}.
One can clearly see that the stencil entries are smooth functions that can be easily approximated by polynomials. This observation motivates
our new approach.

\begin{figure*}\centering
\begin{tabular}{@{}cc@{}}
\includegraphics[trim=0pt 0pt 0pt 0pt, clip=true, width=0.475\textwidth]%
{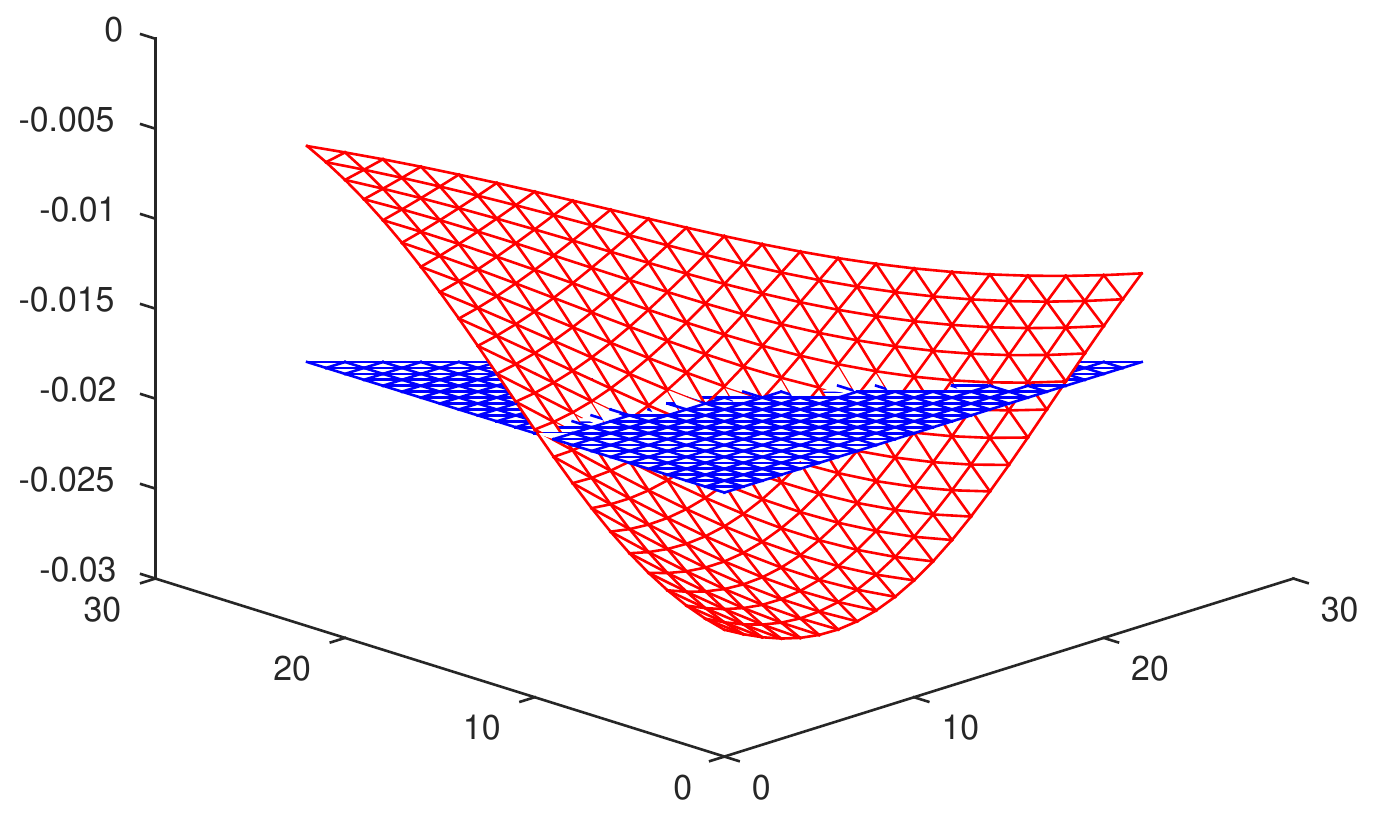} &
\includegraphics[trim=0pt 0pt 0pt 0pt, clip=true, width=0.475\textwidth]%
{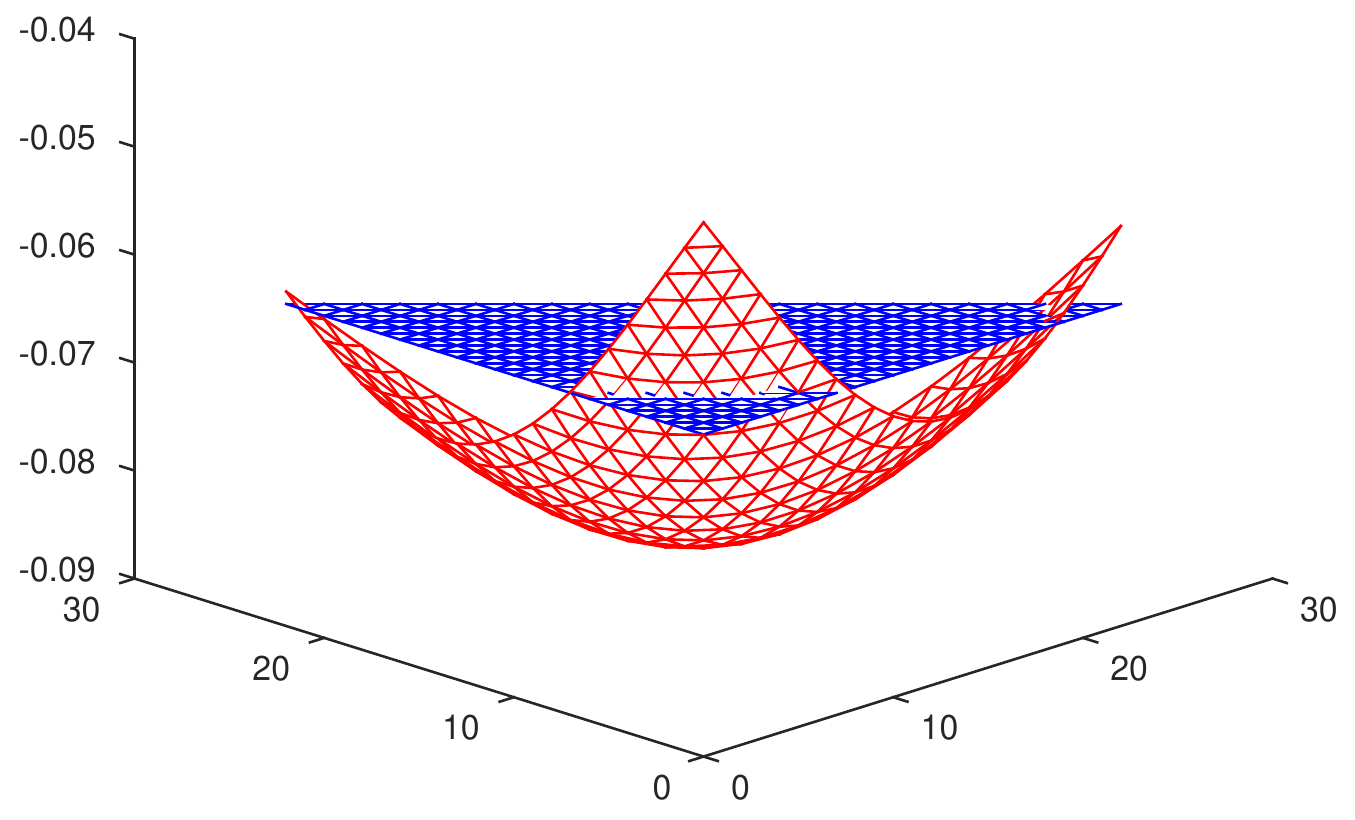} \\
 (a) $w=tc$ & (b) $w=ts$  \\

\includegraphics[trim=0pt 0pt 0pt 0pt, clip=true, width=0.475\textwidth]%
{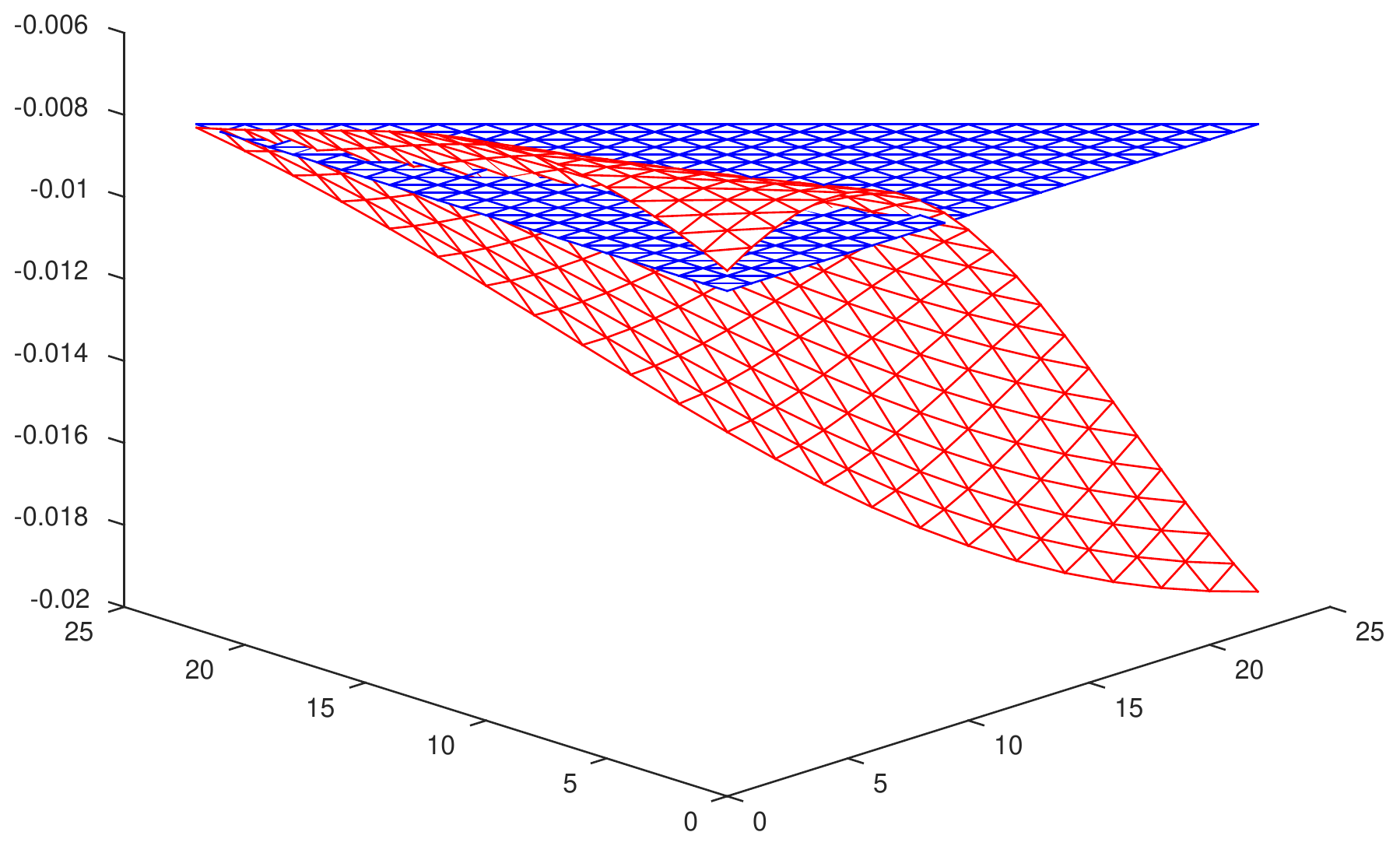} &
\includegraphics[trim=0pt 0pt 0pt 0pt, clip=true, width=0.475\textwidth]%
{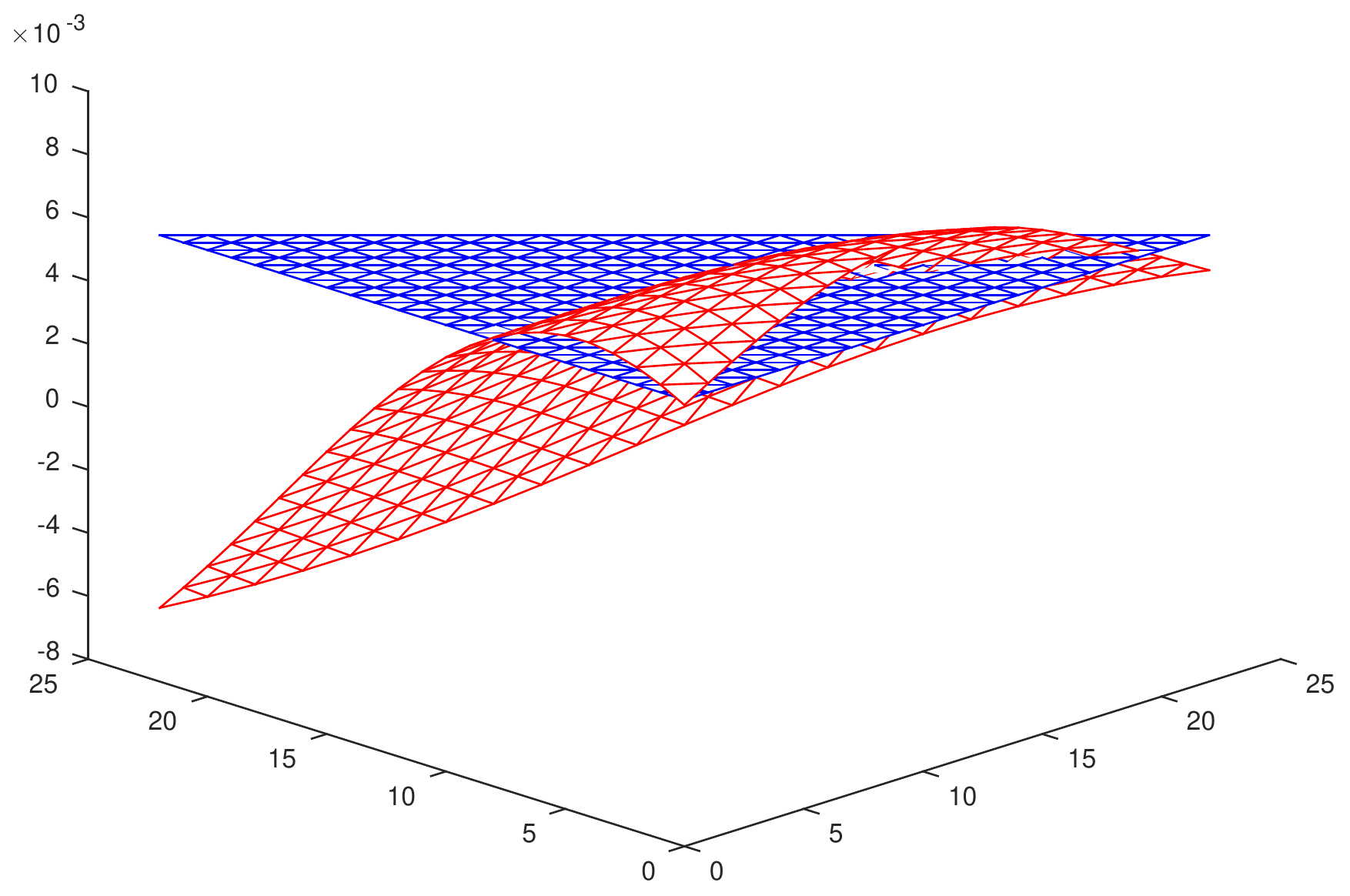} \\
(c) $w=mw$ & (d) $w=mn$  \\

\includegraphics[trim=0pt 0pt 0pt 0pt, clip=true, width=0.475\textwidth]%
{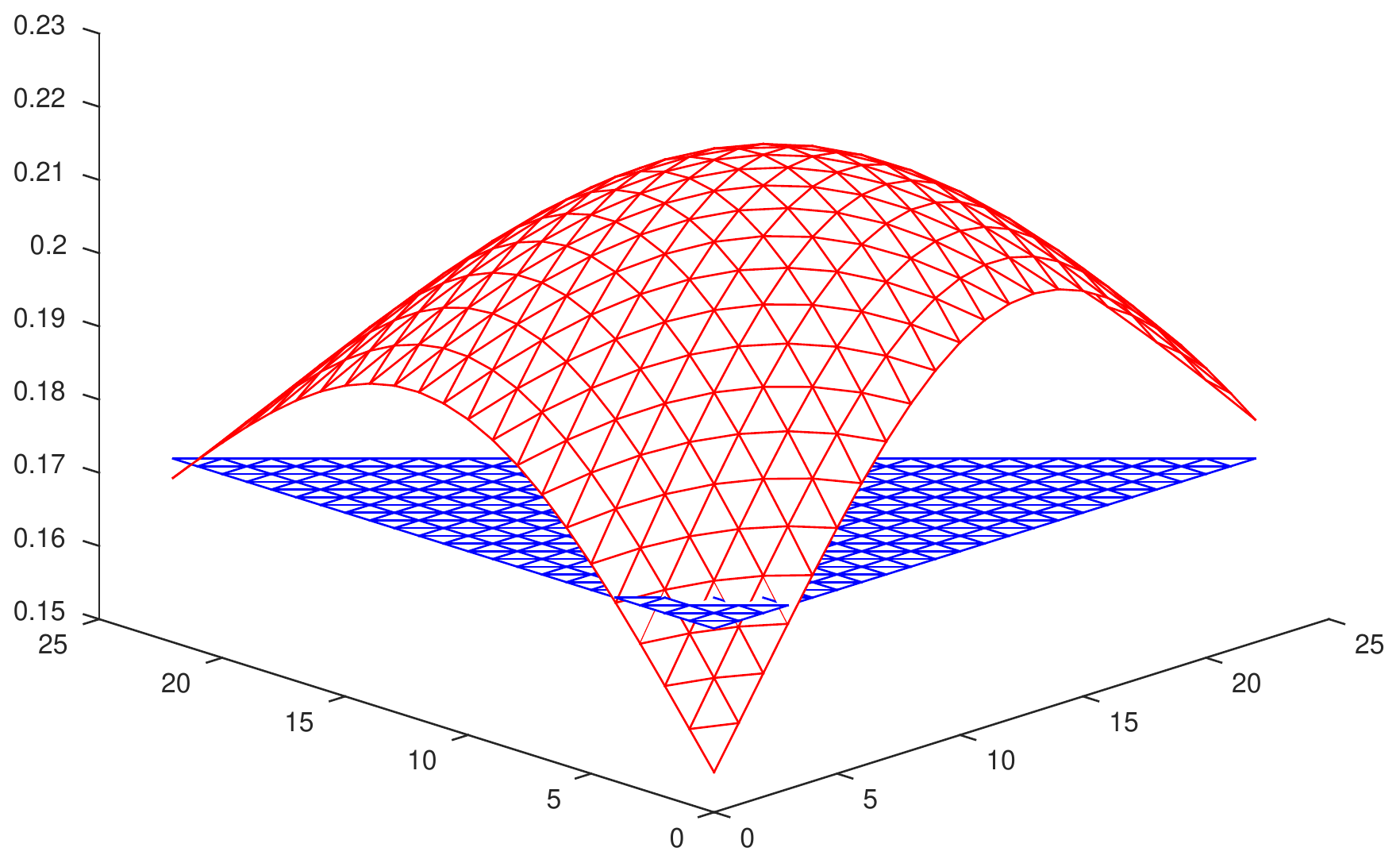} &
\includegraphics[trim=0pt 0pt 0pt 0pt, clip=true, width=0.475\textwidth]%
{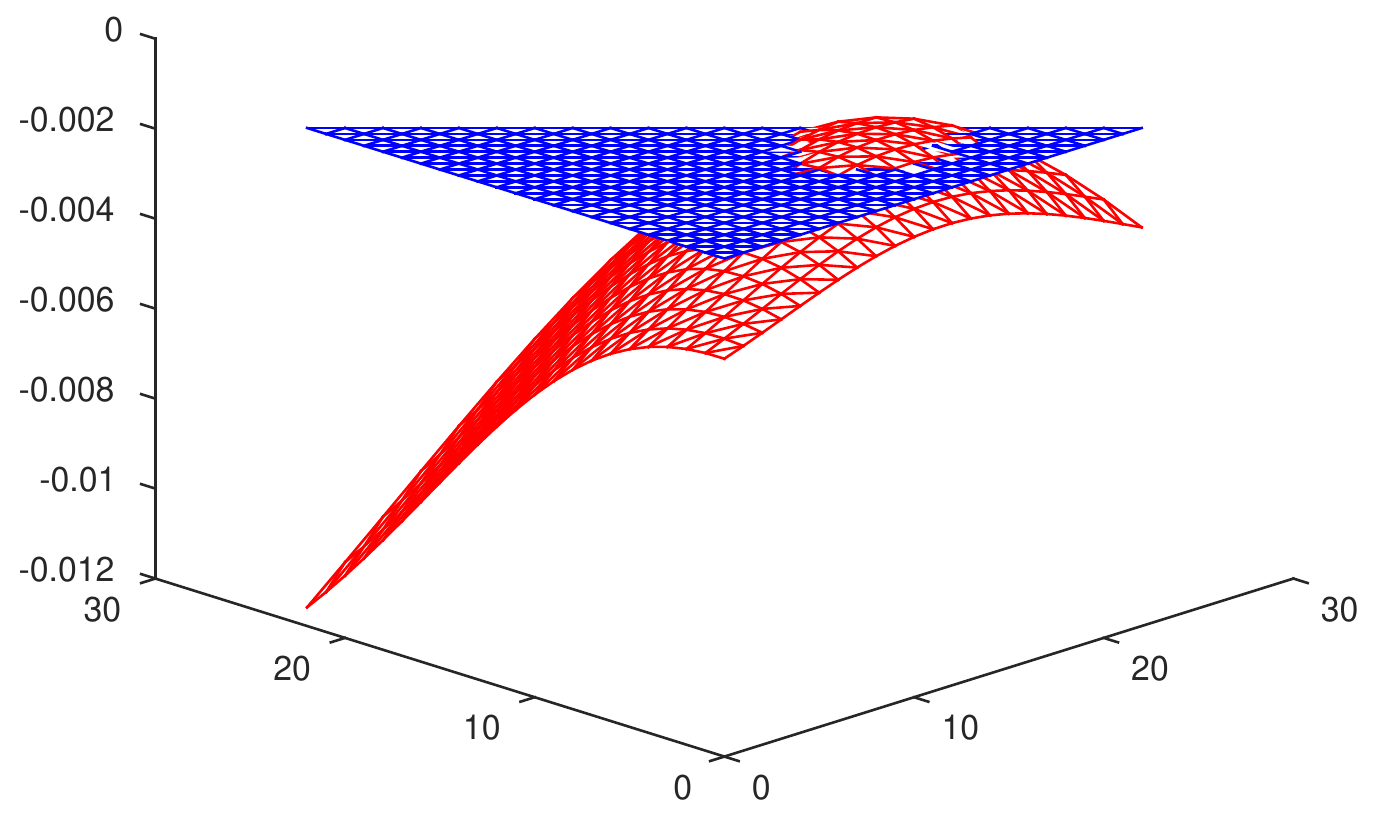} \\
 (e) $w=mc$  & (f) $w=mse$  \\
 
\includegraphics[trim=0pt 0pt 0pt 0pt, clip=true, width=0.475\textwidth]%
{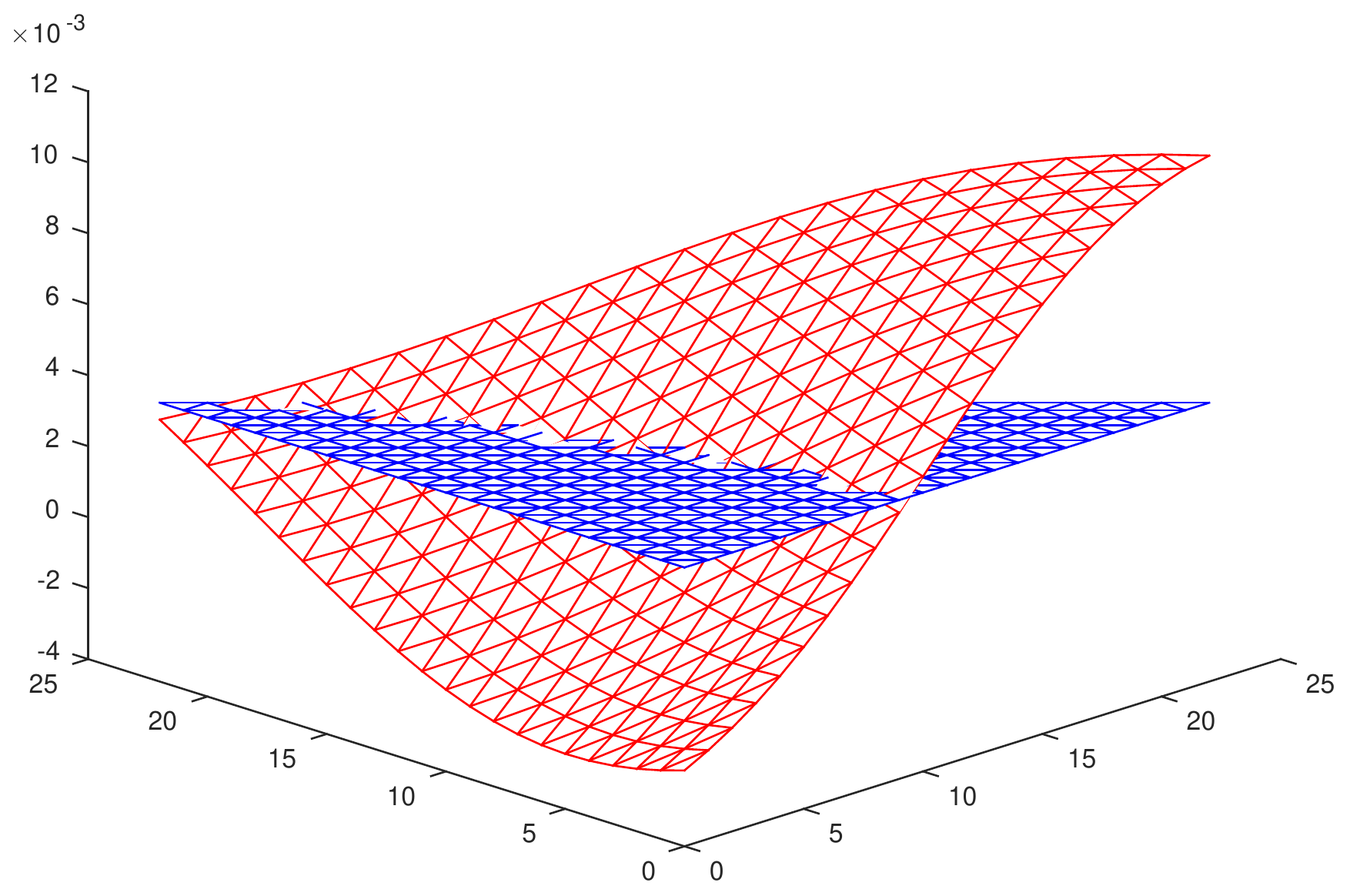} &
\includegraphics[trim=0pt 0pt 0pt 0pt, clip=true, width=0.475\textwidth]%
{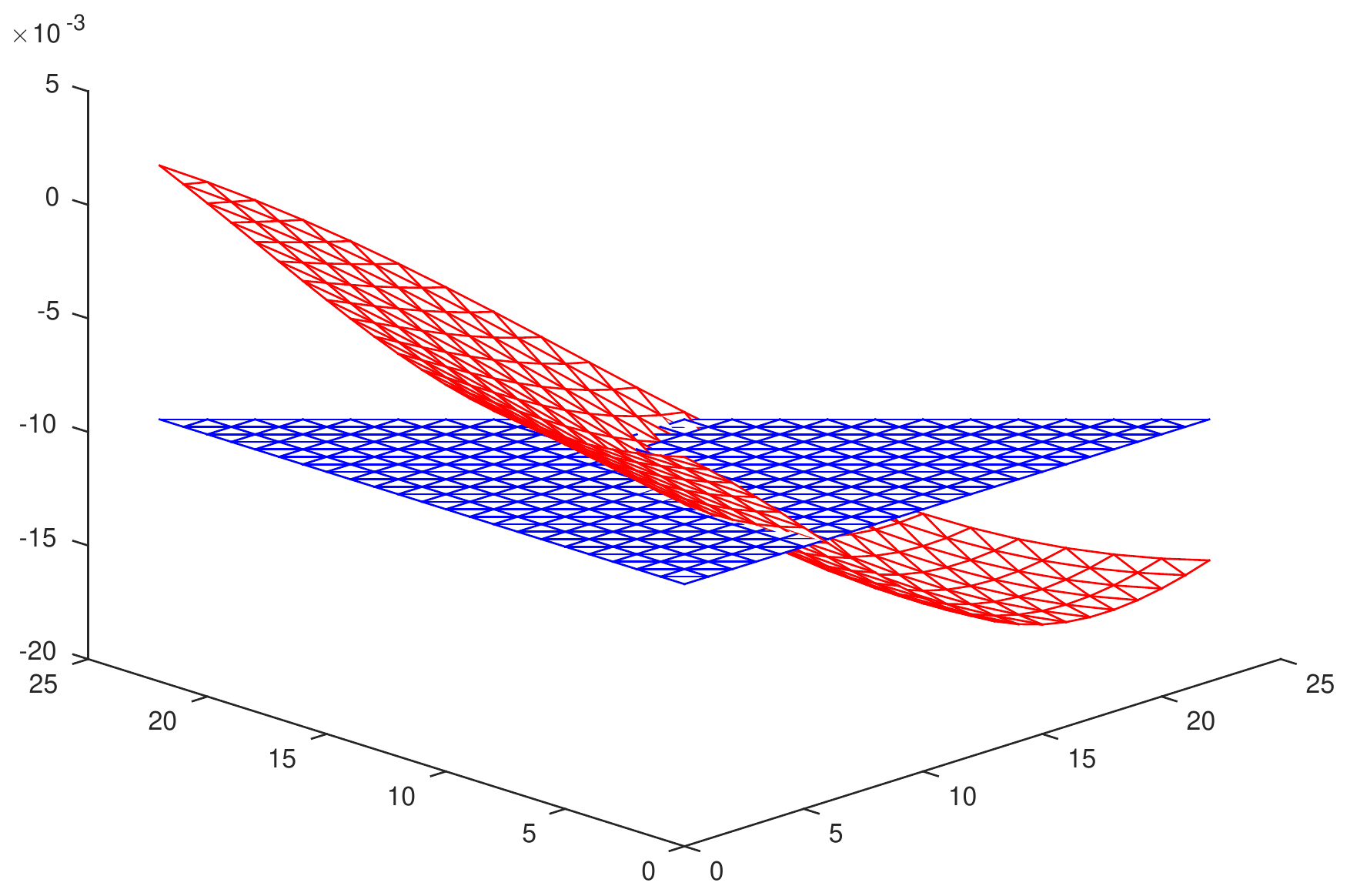} \\
 (g) $w=bnw$  & (h) $w=be$
\end{tabular}
\caption{Eight out of the 15 stencil values associated with the uniformly refined mesh series $\T$ (blue) and the projected series $\Tt$ (red). Each subplot shows the stencil values of one $w \in \W$
for all fine grid nodes in the plane $k=7$ of a macro element
at refinement level $\ell=3$.
Here the x,y location corresponds to the i,j index of the nodes, and the z-direction 
shows the stencil value. The fixed  $k$ index is chosen such that it contains the median node.
\label{fig:stencilCCvsVC}}
\end{figure*}

%% file: 3_twoscale.tex
\section{Two-scale approach}
\label{sec:interpolation_approach}
Our idea now is to replace the components of the exact stencil function $s$
associated with a volume primitive node by surrogate polynomials of
moderate order.
More precisely, we consider for each cardinal direction $w \in {\mathcal W}$
a polynomial defined either by interpolation or
alternatively by a discrete $L^2$-best fit.
 We will only perform
the replacement of the stencil function on levels $\ell\geqslant 1$ as on
the coarsest level $\ell = 0$ a macro element has only one interior node with
a single associated stencil.

\subsection{Interpolation of stencils (IPOLY)}
\label{subsec:ipoly}
A tri-variate polynomial of degree less or equal to $q$ can be re\-presented as
\begin{equation}
\label{eq:quadraticPolynomial}
p(i,j,k) = \sum_{\substack{l,m,n=0\\l+m+n\leqslant q}} a_{l,m,n}
\lambda_{l,m,n} (i,j,k)
\enspace,
\end{equation}
where the $\lambda_{l,m,n}$ form a basis.
Solving for the coefficients $\mathbf{a}= (a_{q00},..., $ $ a_{000}) \in
\RR^{m_q}$ where
\begin{equation}\label{eqn:numPolyCoeffs}
m_q = \frac 16 (q+3)(q+2)(q+1)
\end{equation}
of an interpolation polynomial in general leads to a linear system
$A \mathbf{a} = \mathbf{b}$ with a Vandermonde-type matrix
$A \in \RR^{m_q\times m_q}$ and sampling values $\mathbf{b}\in\RR^{m_q}$.
The special choice of a Lagrange interpolation basis yields $A = \text{Id}$.
As example, we consider the case $q=2$. Here we have $m_2=10$ and for sampling we use
the quadratic nodal interpolation points of a shrunk macro element.
 More precisely we use the nodes with indices 
\begin{gather*}
(1,1,1), \enspace (1,1,a),\enspace (1,a,1),\enspace (a,1,1), \\
(1,1,b),\enspace (1,b,1),\enspace (b,1,1), \\
(1,b,b),\enspace (b,1,b),\enspace (b,b,1)
\end{gather*}
where $a=2^{\ell+2} -3$, $b=2^{\ell+1} -1$
as sampling points, see also Fig.~\ref{fig:tetplane}.

\subsection{Least-squares approximation (LSQP)}
\label{subsec:approximation}

Applying an interpolation polynomial implies that the polynomial is exact at
the sampling points, which elevates these points. 
Alternatively, we can use a discrete $L^2$-best approximation polynomial.
This leads to a system with an $n_\ell\times m_q$ matrix, as we now use
all points of $\mathcal{M}_\ell$ as our sampling set $\mathcal{S}_\ell$.
More generally, we define the coefficients $\mathbf{a}$ of the approximating
polynomial as solution of the least-squares problem
\begin{gather*}
\mathbf{a} := \argmin_{\mathbf{z}\in\RR^{m_q}}\|A \mathbf{z} - \mathbf{b}\|_2
\enspace,\\
A\in\RR^{|\mathcal{S}_\ell|\times m_q}\:,\, \mathbf{b}\in\RR^{|\mathcal{S}_\ell|}
\enspace.
\end{gather*}
Each row of $A$ is defined by the values of the
basis functions $\lambda_{l,m.n} $ evaluated at the corresponding  sampling point. The choice $\mathcal{S}_\ell = \mathcal{M}_\ell$
quickly becomes quite expensive, due to the fast growth of $n_\ell$ with
$\ell$. Thus we also consider reduced sampling sets
 $S_\ell := \mathcal{M}_{m(\ell,j)}$ with
\begin{equation}\label{eqn:samplingSetFunctions}
m(\ell,j) := \min\{\ell,\max\{1,j\}\}
\end{equation}
where $j$ may depend on $\ell$. The min-max definition guarantees that we neither 
exceed the largest possible set, i.e., $ \mathcal{M}_\ell$ on level $\ell$, nor
go below $\mathcal{M}_1$. For simplicity of notation, we also denote $
\mathcal{M}_{m(\ell,j)}$ by $\mathcal{M}_j$ if there is no ambiguity.

We note that for $q\leqslant 4$, the matrix $A$ has full rank for any choice of
$S_\ell=\mathcal{M}_k$ with $k\in\{1,\ldots,\ell\}$. This is easily seen by
comparing \eqref{eqn:numInteriorPoints} and \eqref{eqn:numPolyCoeffs} and considering
the spatial distribution of the nodes in the mesh. Consequently the least-squares
problem will then have a unique solution.

\subsection{IPOLY versus LSQP for quadratic surrogate polynomials}
In a first test, we compare the difference between the least-squares technique for ${\mathcal S}_\ell = {\mathcal M}_{\ell-1}$  and the interpolation based approach, see Fig.~\ref{fig:stencilVCvsINTvsAPPROX}.
Note that, as in Fig.~\ref{fig:stencilCCvsVC}, the points visualized are chosen
to be in the plane with $k=7$. Hence the displayed corners are
no sampling points for \ip. While the values of \ip are more accurate close to
the boundary,  the values of \ax have an
uniformly better fit.

\begin{figure*}\centering
\begin{tabular}{@{}cc@{}}
\includegraphics[width=0.475\textwidth]%
{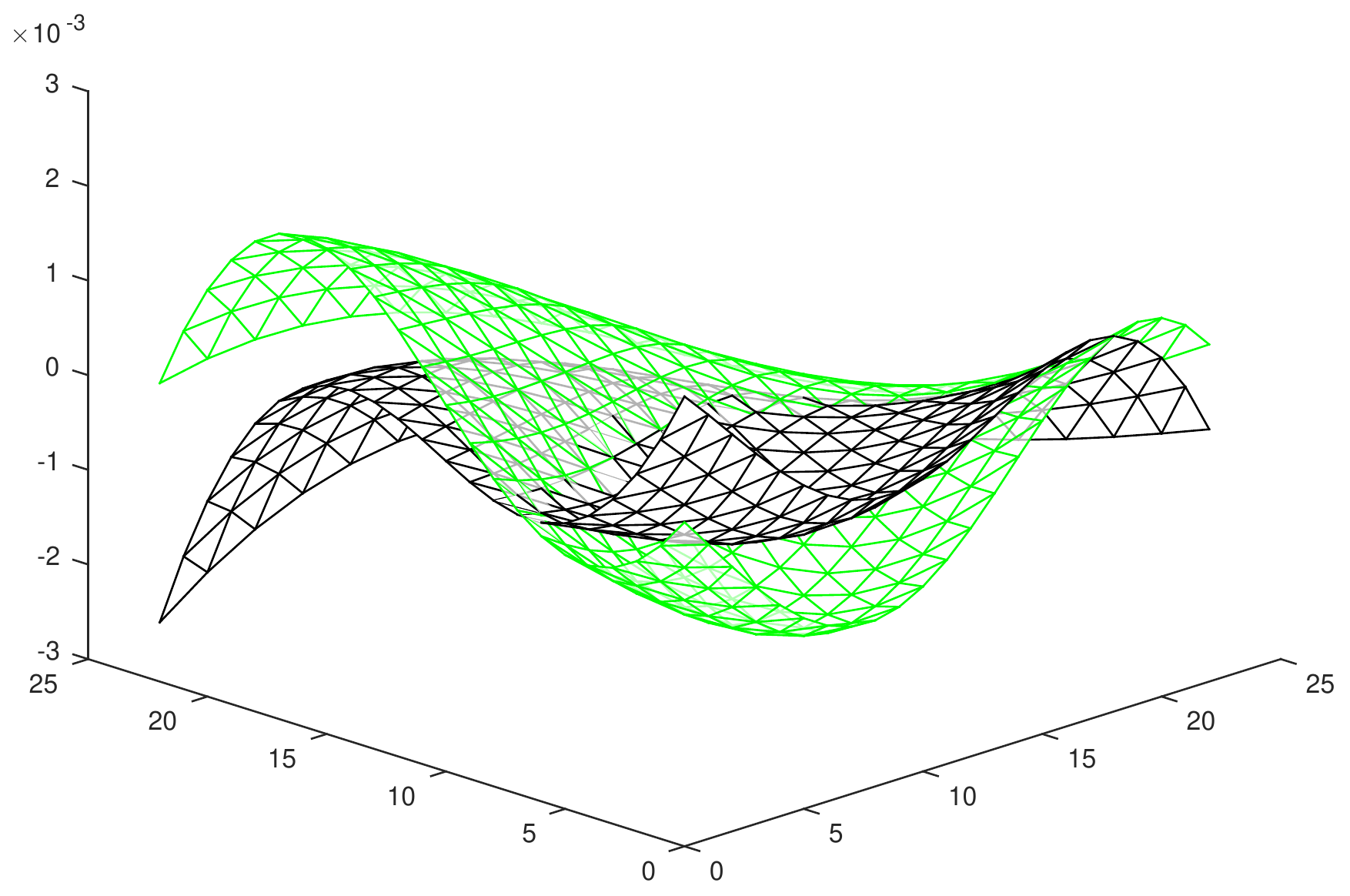} &
\includegraphics[width=0.475\textwidth]%
{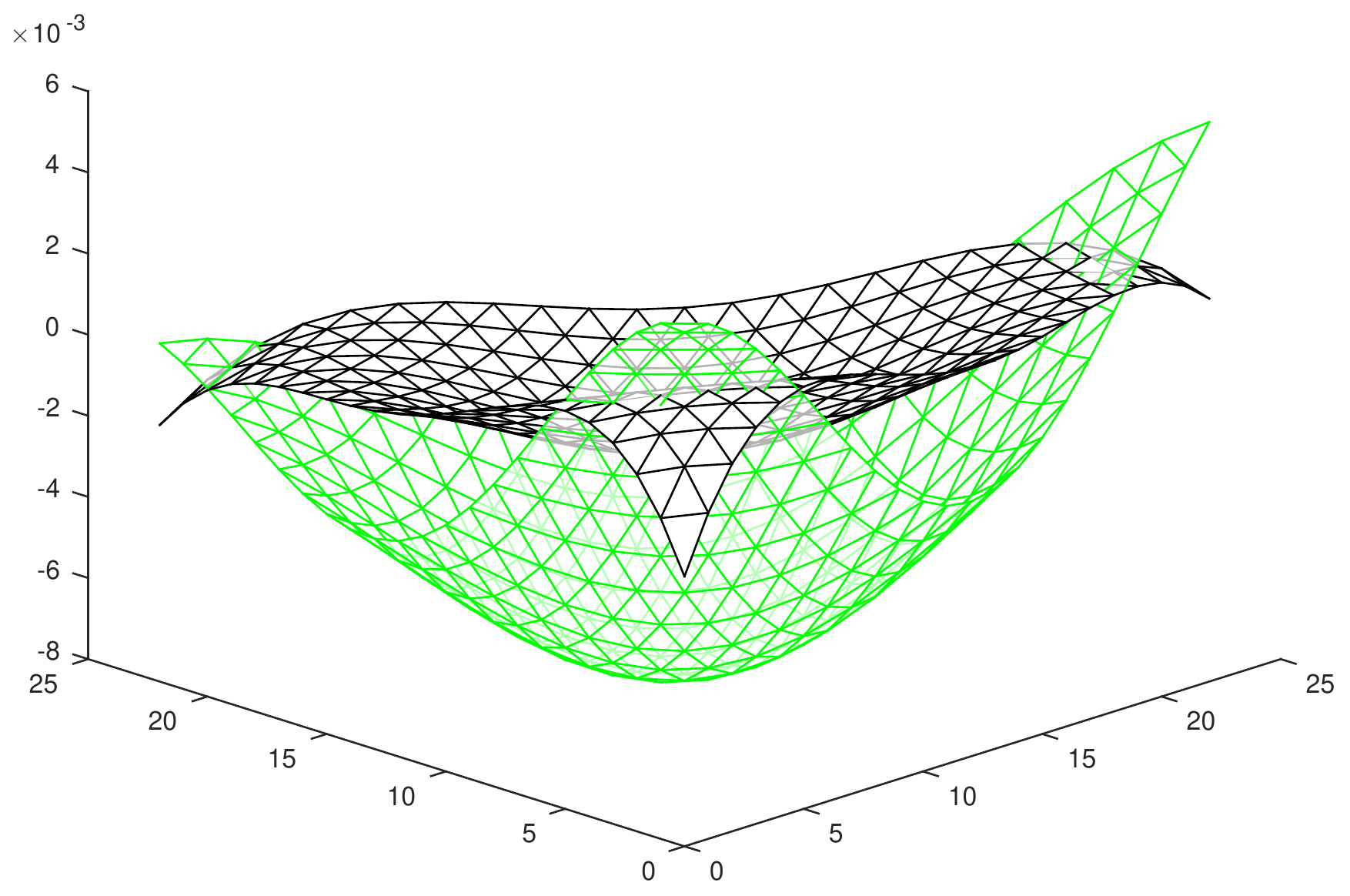} \\
(a) $w=tc$ & (b) $w=ts$  \\

\includegraphics[width=0.475\textwidth]%
{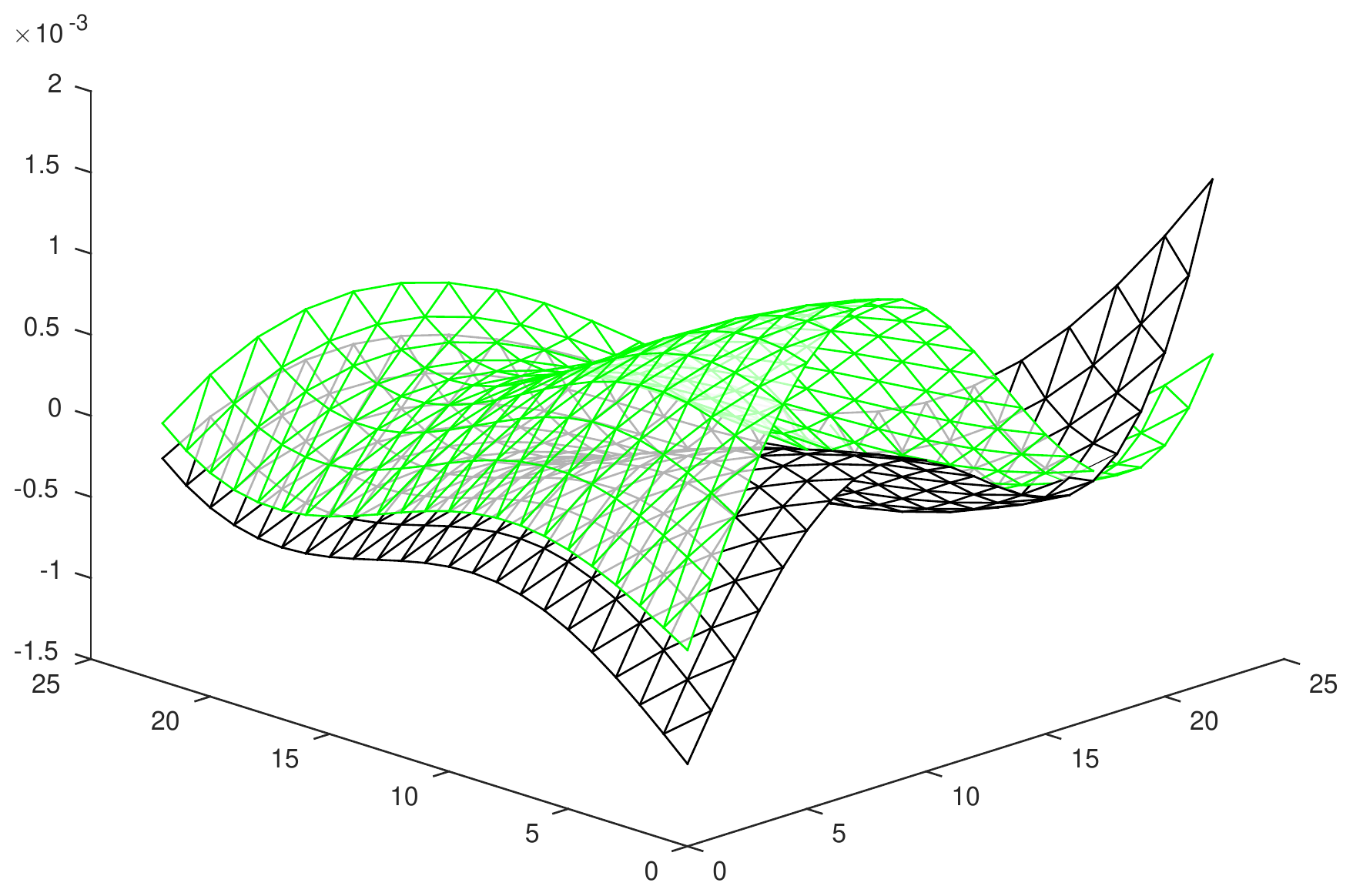} &
\includegraphics[width=0.475\textwidth]%
{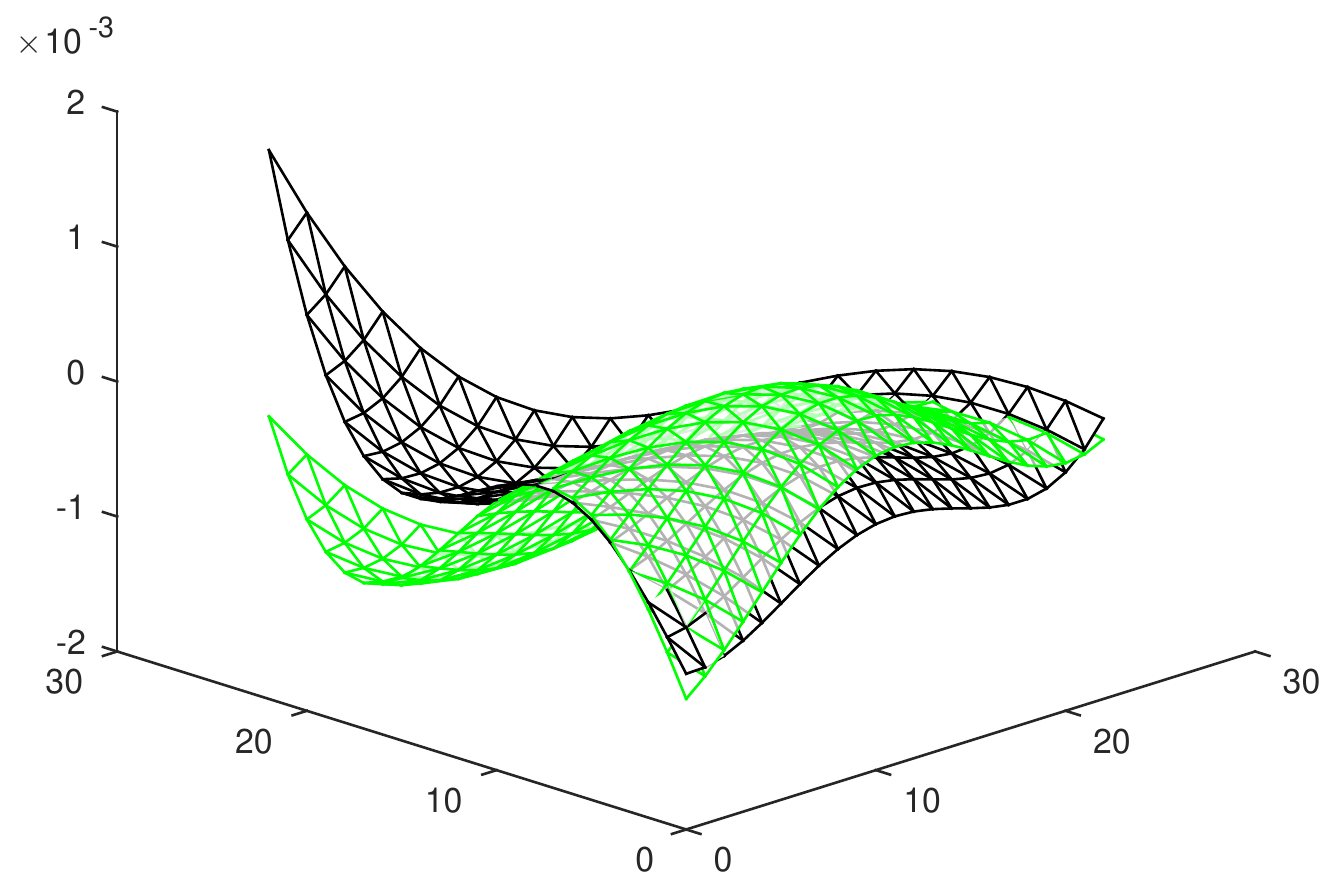} \\
(c) $w=mw$ & (d) $w=mn$ \\

\includegraphics[width=0.475\textwidth]%
{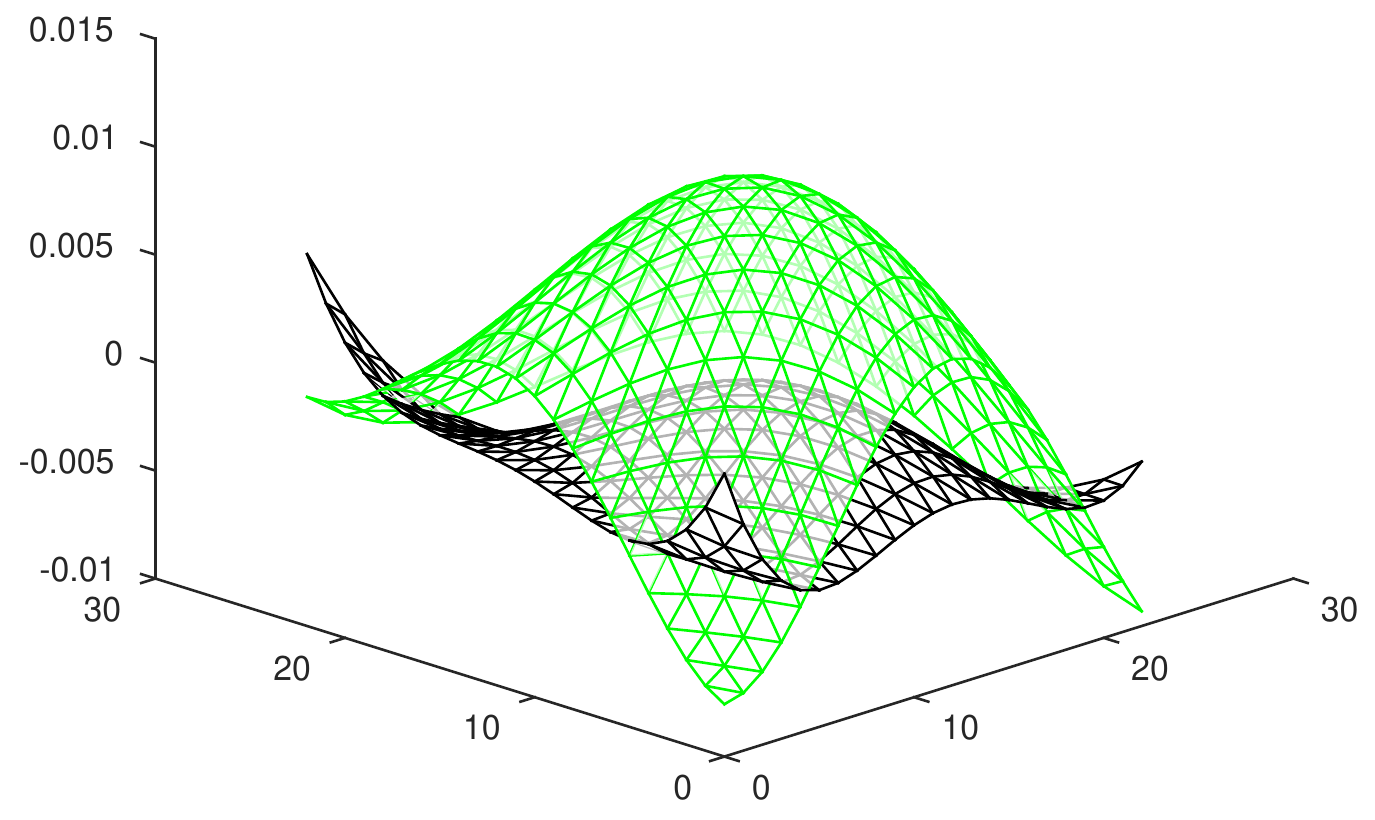} &
\includegraphics[width=0.475\textwidth]%
{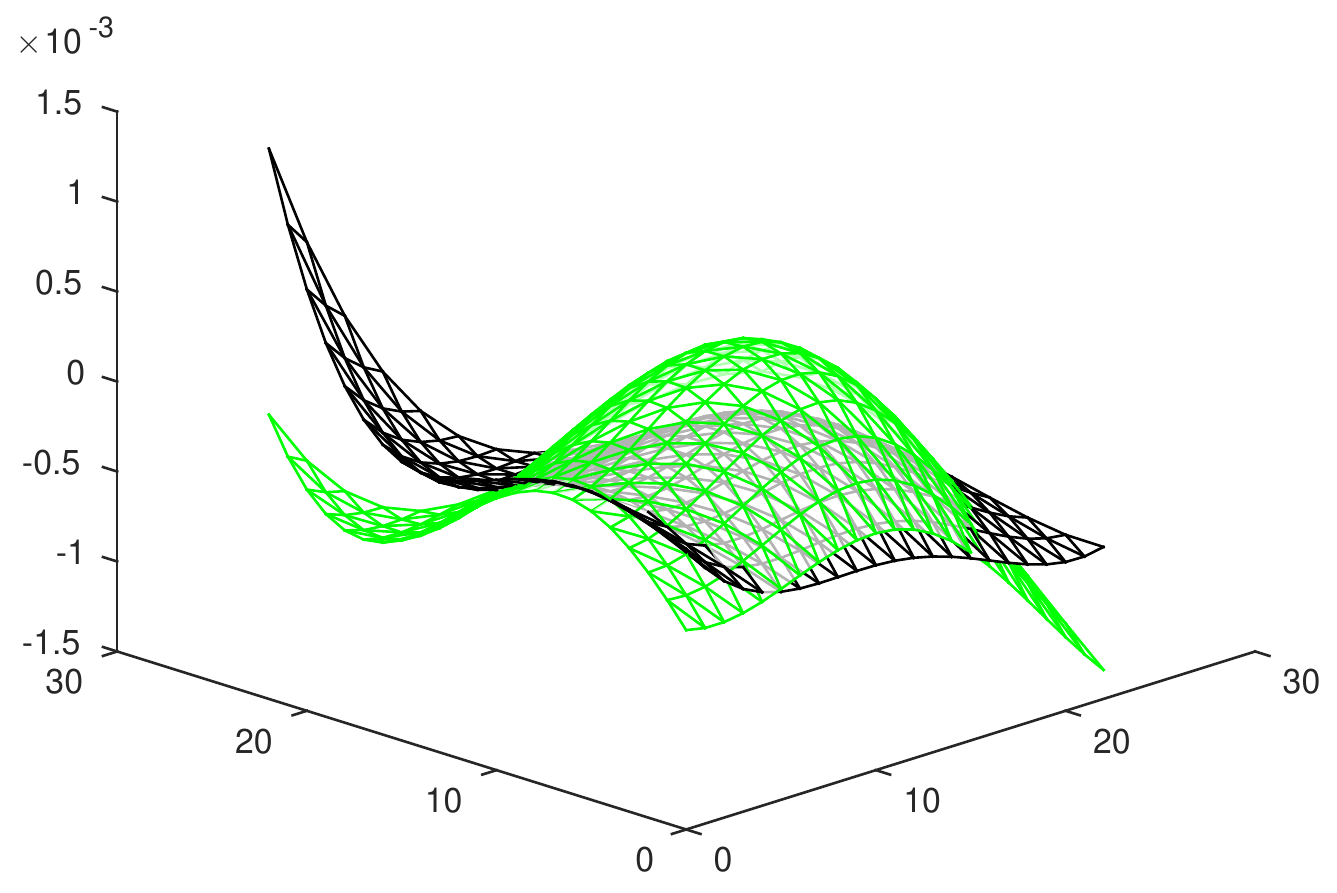} \\
 (e) $w=mc$  & (f) $w=mse$  \\
 
\includegraphics[width=0.475\textwidth]%
{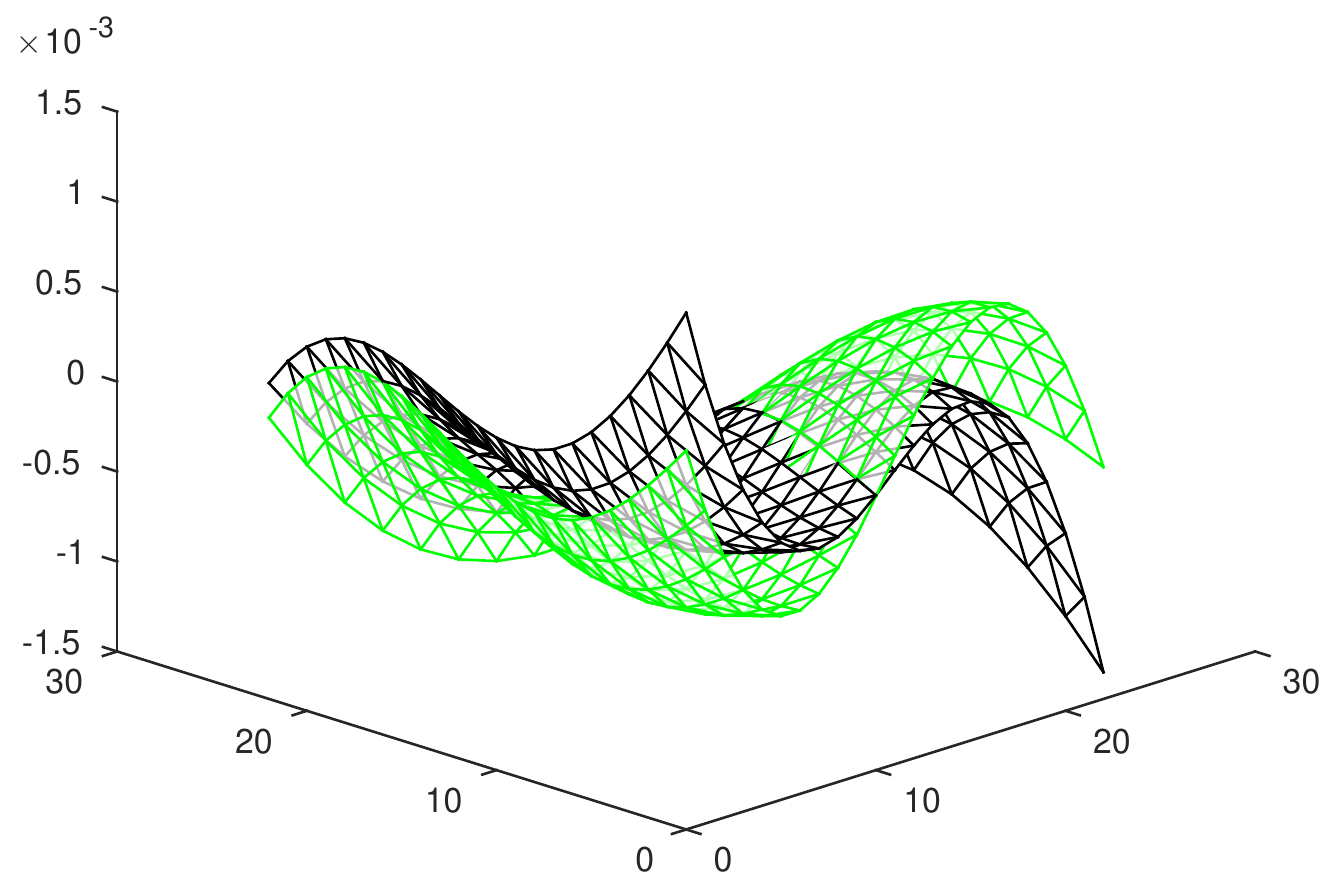} &
\includegraphics[width=0.475\textwidth]%
{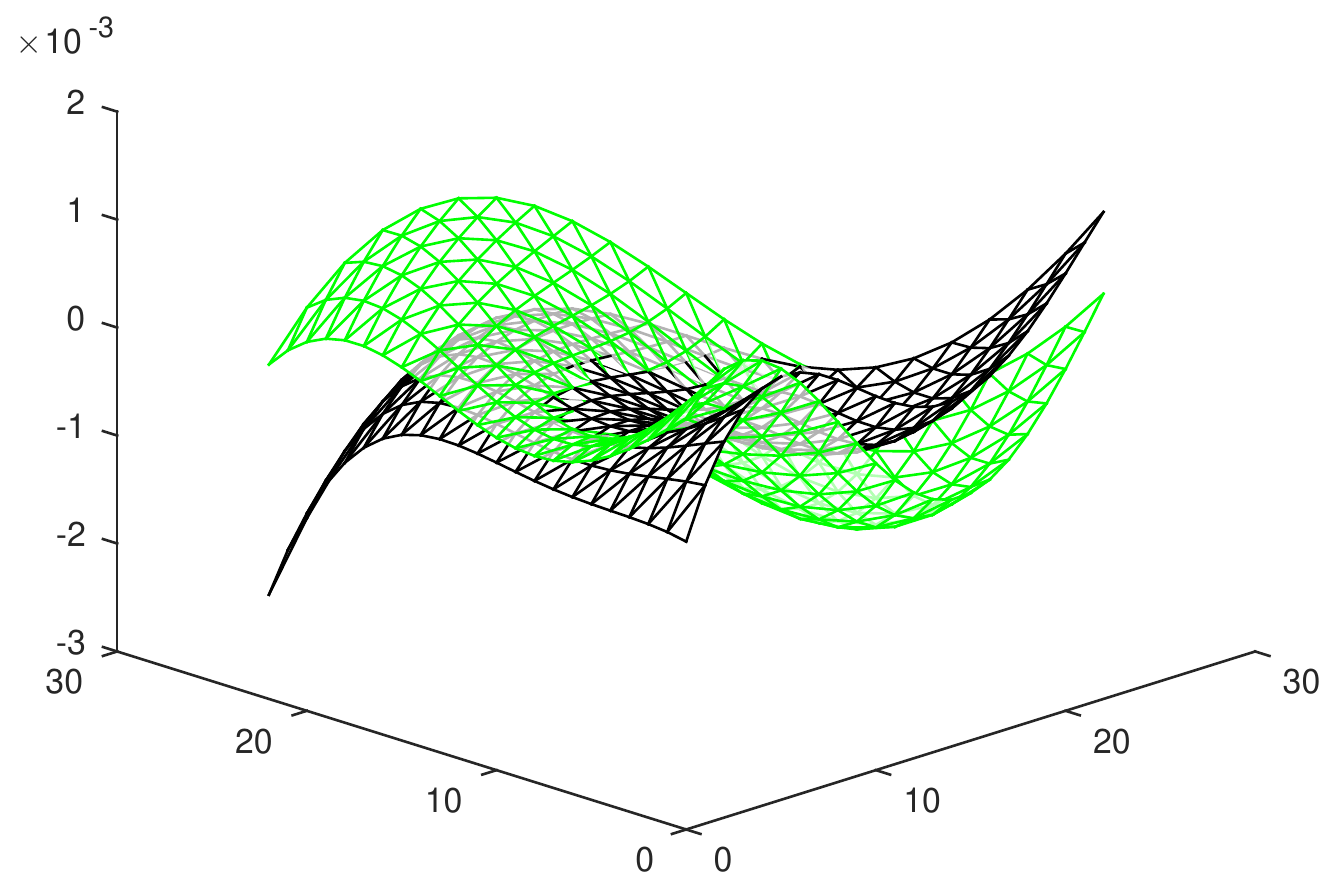} \\
 (g) $w=bnw$  & (h) $w=be$
\end{tabular}
\caption{Difference between \fem and \ip (green) or \ax (black)
   for the same setting than in Fig.~\ref{fig:stencilCCvsVC}.
\label{fig:stencilVCvsINTvsAPPROX}}
\end{figure*} 

In order to provide a more quantitative measure for the approximation accuracy,
we introduce two metrics, one for the overall fit and one for the maximal
deviation. For $w \in {\mathcal W}$ we set
\begin{align*}
e_w^2 &:= \sqrt{\sum_{(i,j,k) \in {\mathcal I}_\ell} \frac{\left[ s_w(i,j,k) - \P^\alpha {s}_w(i,j,k)\right]^2}{n_\ell}} \\
e_w^\infty &:= \max_{(i,j,k) \in {\mathcal I}_\ell}
\left|{s}_w(i,j,k) - \P^\alpha {s}_w(i,j,k)\right|
\,,
\end{align*}
where $\P^\alpha s_w$, $\alpha \in \{\ax, \ip\}$, denotes one of the proposed
surrogate polynomials associated with the cardinal direction $w \in
{\mathcal W}$. Obviously these two metrics correspond  to the discrete
$L^2$ and $L^\infty$ norms of the error in approximating the $w$-th component of
the stencil function $s$.

\begin{figure*}\centering
\begin{tabular}{@{}cc@{}}
\includegraphics[trim=110pt 270pt 110pt 265pt, clip=true, width=0.48\textwidth]%
 {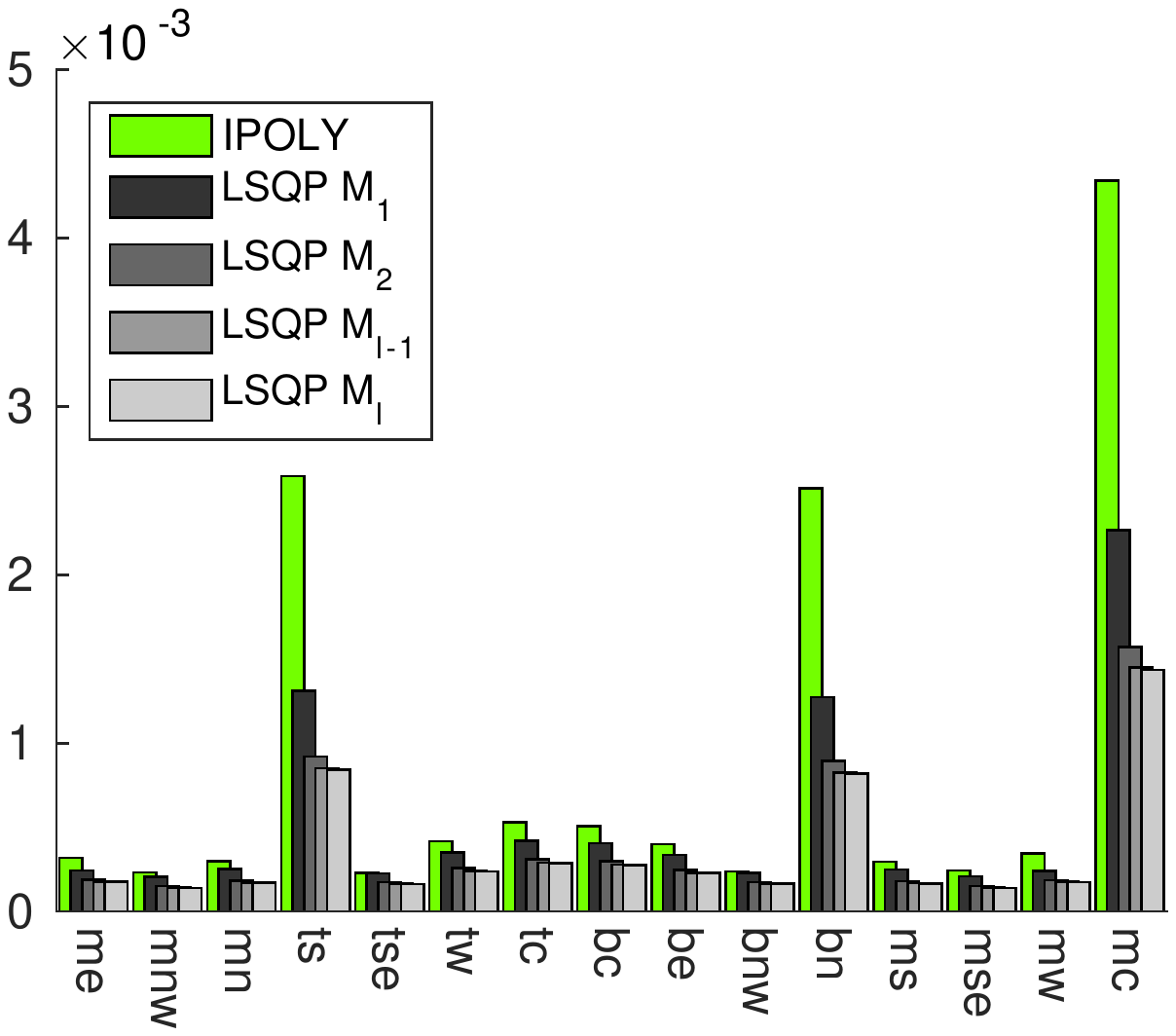} &
\includegraphics[trim=110pt 270pt 110pt 265pt, clip=true, width=0.48\textwidth]%
 {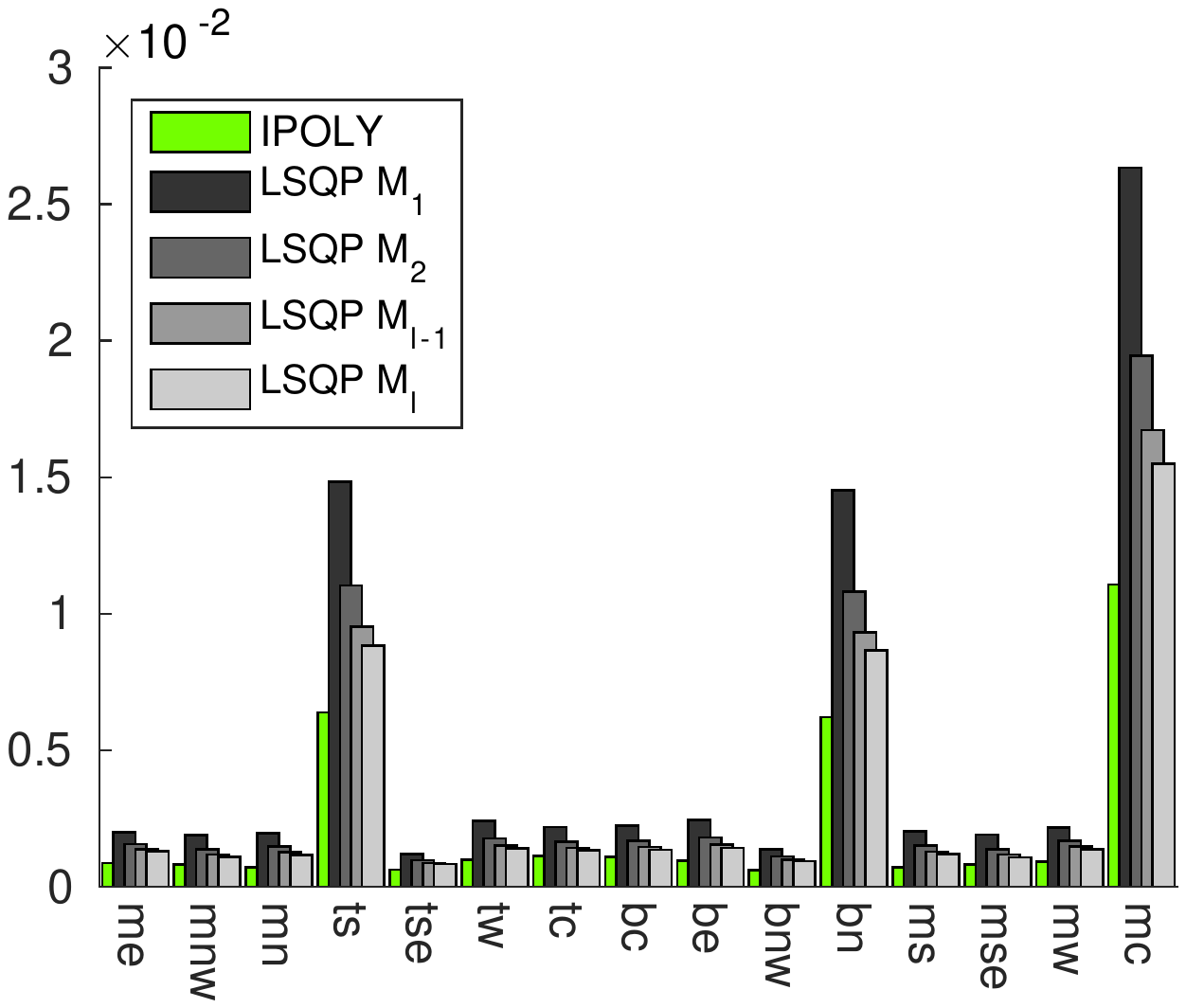}  \\
\end{tabular}
\caption{$e_w^2$ (left) and $e_w^\infty$ (right) for \ip and \ax with different
choices for the sampling set. Macro mesh consists of 60 elements and $L=4$.
\label{fig:bar_l2_lmax}}
\end{figure*} 

In Fig.~\ref{fig:bar_l2_lmax}, we illustrate the influence of the
choice of the
sampling set on \ax and compare it to the \ip approach.
We observe two peaks for the off-center values
and consequently also one in the center stencil entry.
The mapping function $\Phi$ depends on the location of the
macro element within the spherical shell and determines which
$w \in {\mathcal W}$ is more strongly affected by the approximation.
Comparing \ax and \ip we find that the maximal difference for \ax is higher
while the $L^2$ fit is better by construction. Already for the smallest sampling
set ${\mathcal M}_1$, i.e., $n_1 =35$, it is significantly better than for \ip. 
We note that ${\mathcal M}_2$, i.e., $n_2 = 455$, is, in general, a good compromise, in particular if the number of macro elements is already large.

\subsection{Definition of the surrogate operator}
We define our new two-scale finite element approximation as the solution of
\begin{equation} \label{eq:ts}
\L_\ell^{H,\alpha,q} \mathbf{u}_\ell^{H,\alpha,q} = \mathbf{f}_\ell.
\end{equation}
Here, the superscripts mark the dependencies on the macro element size $H$, 
the degree $q$ of the surrogate polynomial, and the selected approach
$\alpha\in\{\ip,\ax\}$.
For ease of notation, we replace the superscript triple by a tilde, i.e., we
set $\Lt_\ell:=\L_\ell^{H,\alpha,q}$ and similarly for $\mathbf{u}_\ell^{H,\alpha,q}$.

As is the case with $\L_\ell$, each row of $ \Lt_\ell $ is associated with a
node. In our approach, we replace for all nodes in a volume primitive
the exact finite element stencil on level $\ell$ by its surrogate polynomial
approximation, i.e., instead of computing  $s_w(i,j,k)$ on-the-fly, we evaluate
the surrogate ${\mathcal P}^\alpha s_w (i,j,k)$. As already mentioned,
the stencil entries associated
with nodes on lower dimensional primitives are assembled on-the-fly in standard
fashion.

\begin{remark}
We point out that, if $\Phi_T = \text{Id}$, $T \in {\mathcal T}_{-2}$,
then both our approaches reduces to the standard one. Also, they
can be applied selectively, and switched on or off macro element-wise
depending on the location of $T$ within the domain or its shape. Additionally
the choice of the polynomial degree can depend on $T$, and we can even use
anisotropic orders reflecting the properties of the mapping $\Phi_T$. 
\end{remark}

Since we have one stencil weight for each of the 15 couplings $w \in \W$, we
use 15 different second order polynomials, described by 15 different sets
of coefficients $\mathbf{a}_w$. The latter are computed in a setup phase.
This, firstly, requires for each macro element and each level $\ell\geqslant 1$
the evaluation of $m_q$ (\ip) or $|\mathcal{S}_{\ell}|$ (\ax) stencils and then
the solution of a linear or least-squares system.

Once the setup is completed, whenever the stencil for node $I$ is to be
applied, e.g.~as part of a smoothing step in a multigrid method, or a residual
computation, 15 polynomials of type \eqref{eq:quadraticPolynomial} must be
evaluated at this node to provide the surrogate stencil weights.
At first glance, this does not seem to be any faster than a standard
implementation. Recall, however, that within a typical matrix-free framework,
the on-the-fly evaluation of the stencil involves costly operations such as the
(re-)computation of local element matrices or the direct evaluation of weak
forms via quadrature rules. The advantage of our approach is
that the setup needs to be performed only once. The operator assembly itself
reduces to polynomial evaluations.
This also implies that the cost for the actual stencil application is the
same for both, \ip and \ax. Only in the setup phase is \ax more expensive
as more sampling points need to be evaluated and the minimization problem
includes a larger system matrix.

Compared to a full stencil pre-calculation and storage scheme also the amount
of required memory can be  neglected. We only need $15 m_q L$
coefficients per macro element to fully describe the approximated stencils.
Such a small number of coefficients can easily be stored even when $L$ is large.

\begin{remark}
Alternatively, one can think of using the same fine level surrogate polynomial to evaluate the stencils on the coarser mesh\-es. A rigorous study of this effect on the multigrid convergence rate is beyond the scope of the paper.
\end{remark}

We want to point out that in the application of the surrogate operator
we loop over the indices $(i,j,k)$ in a line-wise fashion. This implies
that we can always fix two of the three indices and, consequently, only
need to evaluate a 1D and not a 3D polynomial. This reduces the cost
significantly as the 1D polynomial has fewer terms.
Evaluation of the 1D polynomial itself can be accomplished very
efficiently by using a hierarchical recursion formula. More precisely 
starting from a single direct evaluation of $\hat p (i) := {\mathcal P}^\alpha s_w(i,j_0,k_0) $, we can evaluate the polynomial incrementally in terms of $q$ basic operations. 
In computer graphics this technique 
is known as forward differencing \cite{Rockwood1987,Rappoport1991}.

Apart from its cost, the accuracy of a numerical scheme is, of course,
of major importance. In this context we provide the following lemma related to
the consistency of the discretization expressed by $\Lt_\ell$.

\begin{lemma}
The row sum property of the stencil
 $$ \sum_{w\in\W}s_w^I =0
$$ is preserved by 
\ip and \ax.
\end{lemma}
\begin{proof}
We recall that the coefficients of both our approaches can formally be
obtained by solving a system of normal equations.
 As $A$ has full rank its solution
is unique and 
{\newcommand{\ncc}{\substack{w\in\mathcal{W}\\w\neq mc}}
\begin{align*}
\mathbf{a}_{mc} &= \left(A\TP A\right)^{-1} A\TP \mathbf{b}_{mc} \\
&= - \left(A\TP A\right)^{-1} A\TP \sum_{\ncc}\mathbf{b}_w 
 = -\sum_{\ncc} \mathbf{a}_w
\,.
\end{align*}}%
Thus, the row sum condition holds for the coefficients of the surrogate
polynomials and, consequently, also for the stencil entries approximated by
an evaluation of the polynomials at each node position.
\end{proof}

Finally, we recall that the matrix  $\L_\ell$ is symmetric, as
$a_\ell({\phi}_I,{\phi}_J)$ = $a_\ell({\phi}_J,{\phi}_I)$. This property is not
preserved by our approach, due to the fact that the polynomials are evaluated
at the nodes.
Let $I$ be a node, and $J$ be the neighboring node which is reached, if we move
from $I$ in cardinal direction $w$. By $w^\text{o}$, we denote the opposite
cardinal direction, i.e., if we move from $J$ into the direction of
$w^\text{o} $ then we recover $I$. In our approach, we find
\begin{align*}
(\Lt_\ell )_{IJ} &= {\mathcal P}^\alpha s_w (i_I,j_I,k_I) \\
(\Lt_\ell )_{JI} &= {\mathcal P}^\alpha s_{w^\text{o}} (i_J,j_J,k_J) \, .
\end{align*}

To quantify this loss of symmetry, we measure it in the relative Frobenius norm.
We consider 
\begin{equation*}
%\left(\text{nnz}\right)^{\nicefrac{1}{2}}
\frac{\| (\Lt_\ell|_{T} )\TP - \Lt_\ell|_{T}\|_{_F}}{\| \L_\ell|_{T} \|_{_F}}
\,.
\end{equation*}

Here $\Lt_\ell|_{T}$ is the restriction of the global stiffness matrix to one
macro element. 
As expected, for both \ip and \ax we find 
that the relative non-symmetry is in $\mathcal{O}(h_\ell)$. 
\begin{remark}
A symmetric matrix can also be recovered. One possibility is to only define
seven cardinal directions, associated with the seven edge directions of a
macro element. For this we identify the cardinal directions $w$ and $w^\text{o}$
and evaluate the surrogate polynomial at the center of the edge connecting two
nodes. Note that in that case, we need to define the central stencil weight
via the row sum condition. As we will see in  Sec.~\ref{sec:kernel_opt}, a direct evaluation of the row sum is more expensive than the evaluation of a 
 surrogate polynomial of moderate order. Therefore we will not further consider this option. We also point out that the loss of symmetry can be considered as a small perturbation and does not prohibit using a conjugate gradient method as solver.
\end{remark}

\subsection{A priori estimate}
\label{subsec:apriori}

In this section, we will present an a priori estimate for the \ip and \ax discretization errors
in the $L^2$-norm.
Let $\mathbf{u}_\ell$ and $\mathbf{\tilde{u}}_\ell$ be the solution of 
\eqref{eq:algebraic} and \eqref{eq:ts}, respectively, 
and $u$ the exact solution of \eqref{eq:model_eq}. 

Since $\Lt_\ell$ can be interpreted as polynomial approximation of a
smooth function, we assume that it can be written as
\begin{equation*}
\Lt_\ell  = \L_\ell \left(\text{Id} + H^{q+1} \widetilde{B}_\ell \right)
\end{equation*}
with a suitable  $\widetilde{B}_\ell$ satisfying 
\begin{equation}
\label{eq:Bbounded}
\|\widetilde{B}_\ell\| \leqslant C_{q,\alpha} < \infty .
\end{equation}
Here $\|\cdot\|$ denotes the Euclidean  norm.
We point out that our assumption  is motivated by the  assumed smoothness of
$\Phi$ in combination with best  approximation and interpolation results.
Moreover by construction we preserve the row sum property, and thus the kernel
of $\L_\ell $ and $\Lt_\ell$ is locally the same. 
The notation $\lesssim$ is used for $\leqslant C$ with an $\ell,H$-independent
constant $C < \infty$.

Furthermore, we assume that $H$ is chosen small enough such that $H^{q+1}\|\widetilde{B}_\ell\| < 1$. 
By employing the properties of the Neumann series, we then obtain
\begin{align*}
\begin{split}
\Lt_\ell^{-1} &= \left( \text{Id} + H^{q+1}\widetilde{B}_\ell\right)^{-1}\L_\ell^{-1} \\ 
              &= \left[\sum_{k=0}^{\infty}\left(-H^{q+1}\widetilde{B}_\ell\right)^k\right] \L_\ell^{-1} \\               
              &= \left(\text{Id} - H^{q+1}\widetilde{B}_\ell\right)
              \L_\ell^{-1} + \mathcal{O}\left(H^{2(q+1)}\right) .
\end{split}
\end{align*}
Neglecting the higher order terms, we arrive at   
\begin{align*}
\begin{split}
\|\tilde{{u}}_\ell - {u}_\ell\|_0^2 &\lesssim h^3
\|\tilde{\mathbf{u}}_\ell - \mathbf{u}_\ell\|^2 \\
                                &= h^3 \| \L_\ell^{-1} \mathbf{f}_\ell - \Lt_\ell^{-1} \mathbf{f}_\ell\|^2 \\
                                &\lesssim h^3 \| H^{q+1} \widetilde{B}_\ell \L_\ell^{-1} \mathbf{f}_\ell \|^2 \\
                                &\lesssim h^3 H^{2(q+1)} \| \widetilde{B}_\ell \| ^2 \|  \mathbf{u}_\ell \|^2 \\
                                &\lesssim H^{2(q+1)} \| \widetilde{B}_\ell \|^2 \| u_\ell \|_0^2\:.
\end{split}
\end{align*}
It is well known that under suitable regularity assumptions $\|u -
u_\ell\|_0 \lesssim h_\ell^2 \| f \|_0$. Then the triangle inequality yields
\begin{align}
\begin{split}
\|u - \tilde{u}_\ell \|_0 &\leqslant \|u - u_\ell \|_0 + \|u_\ell - \tilde{u}_\ell \|_0  \\
&\lesssim \left(h_\ell^2 + H^{q+1}\right) \|f \|_0
\label{eq:apriori}
\\
                          &\lesssim h_\ell^2 \left( 1 + 4^{l+2} H^{q-1}\right) \|f \|_0\:. 
\end{split}
\end{align}
We note that the generic constants do possibly depend on $\alpha$ and
$q$. Keeping $H$ fixed and increasing $\ell$ 
will result in an asymptotically constant error not equal to
zero. However for  $\ell \leqslant L$ and $L$ fixed and  a small enough
macro meshsize $H$, we obtain  optimal order results in $h_\ell$, $q \geqslant 2$.

%% file: 4_numaccuracy.tex
% \svnid{$Id: 4_numaccuracy.tex 3209 2016-08-11 12:07:52Z simon.bauer $}
\section{Evaluation of accuracy and cost}
\label{sec:numstudy}

For our numerical studies, we employ model problem \eqref{eq:model_eq} with an
exact solution given by
\begin{equation*}
\bar{u} = (r-r_{1})(r-r_{2})\sin(10x)\sin(4y)\sin(7z)
\end{equation*}
where $r = r(x,y,z) = \left(x^2 + y^2 + z^2\right)^{1/2}$.
The right-hand side is set accordingly, and the homogeneous Dirichlet boundary
conditions are automatically satisfied by $\bar{u}$. An illustration is given
in the left plot of Fig.~\ref{fig:u_ex_err_proj_vs_cons}. The right plot shows 
the finite element error associated with the mesh sequence $\T$. It is
 obviously  dominated by the inaccurate boundary approximation of $\Omega_H$ and is of second order in $H$.

\begin{figure*}\centering
\begin{tabular}{ccc}
\includegraphics[trim=160pt 50pt 20pt 50pt, clip=true, width=0.3\textwidth]{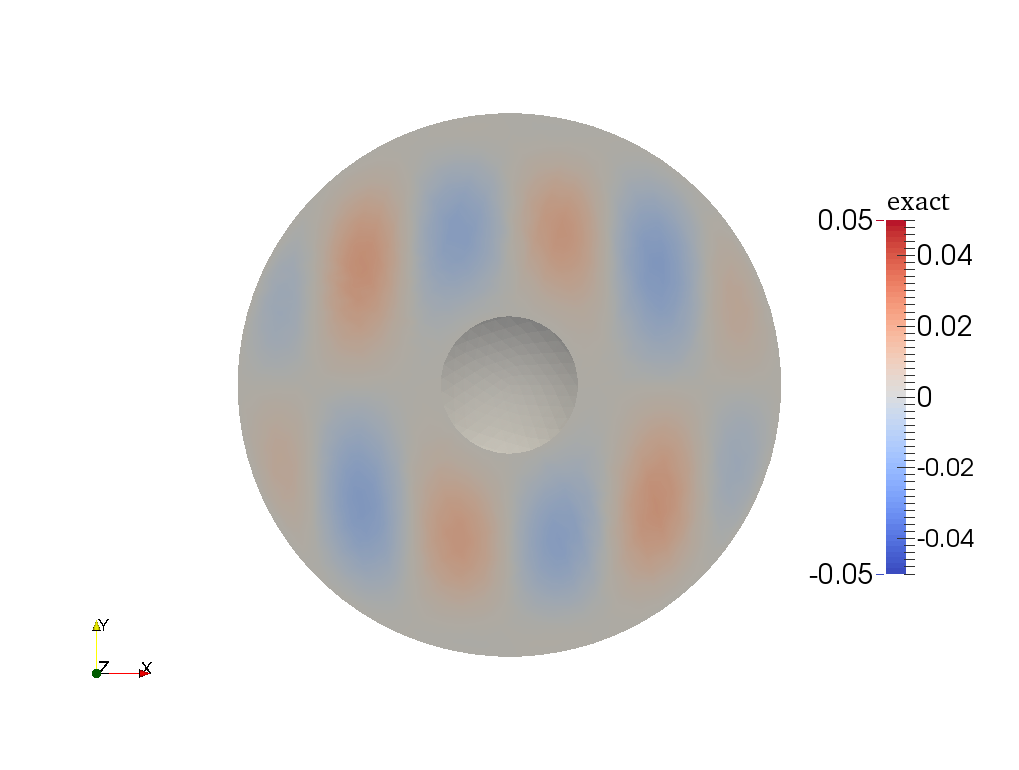} &
\includegraphics[trim=160pt 50pt 20pt 50pt, clip=true, width=0.3\textwidth]{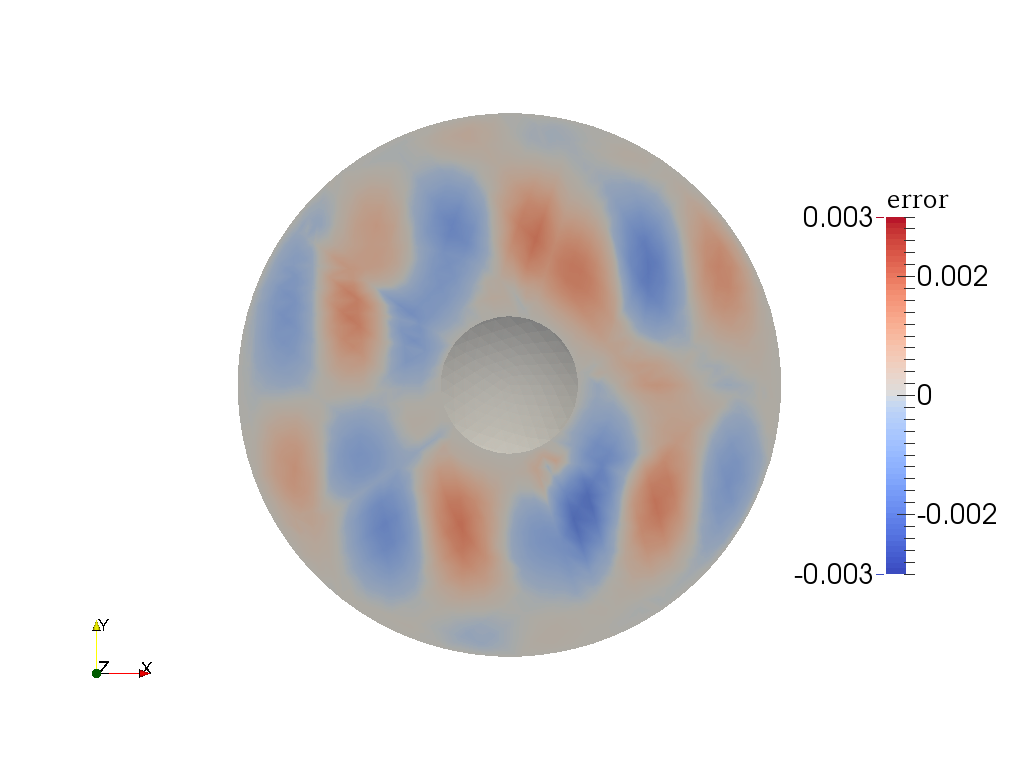} &
\includegraphics[trim=160pt 50pt 20pt 50pt, clip=true, width=0.3\textwidth]{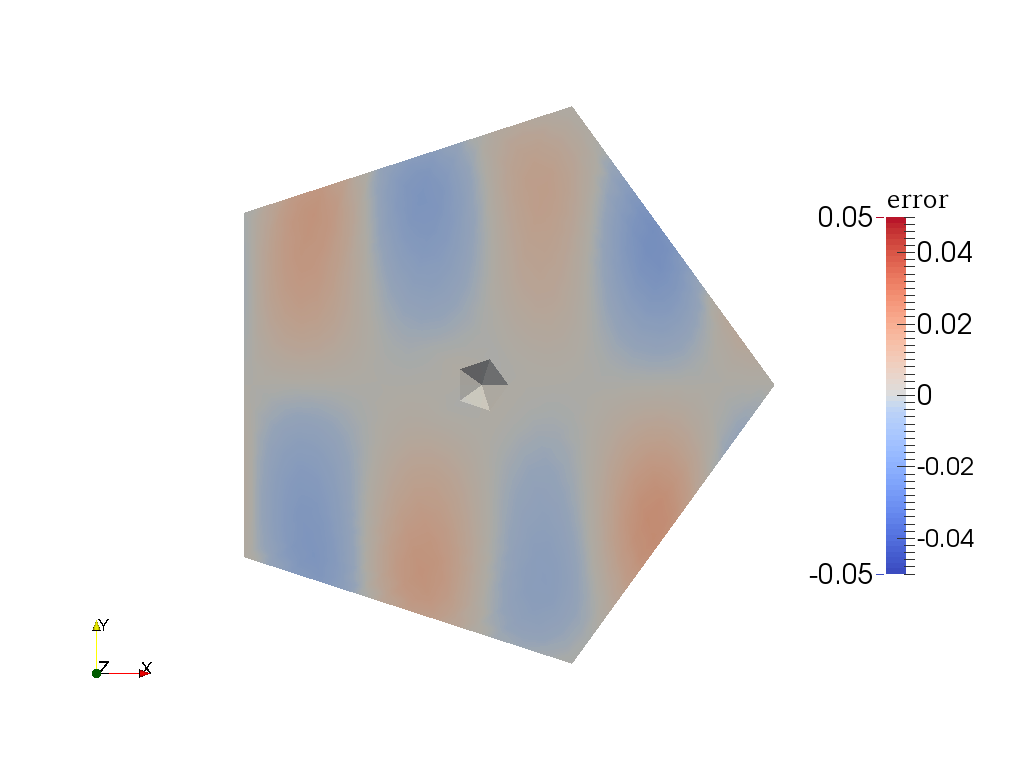} \\
\end{tabular}
\caption{Cut parallel to $x,y$-plane at $z=-0.5$. 
Exact solution $\bar{u}$ (left), error $\bar{u} - u_\fem$ (middle) and $\bar{u} - u_{\text{cons}}$ (right).
On the boundary the error of the constant solution is equal to the exact $\bar{u}$ as we impose zero Dirichlet 
boundary conditions. For $u_\fem$ the error is one order of magnitude lower.
The slight asymmetry visible results from the unstructured macro grid.
\label{fig:u_ex_err_proj_vs_cons}}
\end{figure*}

To solve the discrete problem, we employ a geometric multigrid solver. It
uses a V(3,3)-cycle with a lexicographic hybrid Gauss-Seidel (GS) smoother,
which also serves as coarse grid solver. The smoother is hybrid in the sense
that, due to the HHG communication patterns between primitives, some
dependencies are neglected and low-order primitives can have some points that
are only updated in a Jacobi-like fashion, see \cite{Bergen2004} for details.
However, the majority of points is updated as in a standard GS scheme.
Experiments with different numbers of smoothing steps have shown qualitatively the same results.
The iteration is started with a random initial guess. 

A standard geometric multigrid solver contains two operations which involve the 
application of the stencils. These are \emph{smoothing} and
\emph{residual computation}. Instead of applying either the stencil for
$\L_\ell$ or $\Lt_\ell$ in both operations, it is also possible to apply a mixed
formulation. This is known as \emph{double discretization} (DD), see
\cite{Brandt2011}. We include this technique in our tests. For this we apply
the less expensive surrogate operator $\Lt_\ell$ for the relaxation sweeps,
as those require less accuracy, and the exact one $\L_\ell$ for residual
computation.
Using two different operators leads to two competing components in the multigrid
method, as the smoother and the coarse grid correction aim to drive the
iterate towards the algebraic solution of their respective operators. Hence in
the double discretization approach, the algebraic convergence will typically suffer,
and  residual-based stopping criteria are inappropriate. Instead one should,
e.g., check the sequence of updates for convergence to zero \cite{Brandt2011}.

\subsection{Accuracy of IPOLY and LSQP}
\label{subsec:accuracy}

For our numerical study, we will compare the standard Galerkin formulation with
projected coordinates, termed \fem, against the \ip and \ax approaches.
For the latter, we will also test the aforementioned double discretization
technique. The resulting iterative schemes are denoted as \ip DD and \ax DD.
Besides this we will show finite element results for an unprojected approach
(\cons), i.e.~with a fixed $\Omega_H$. In order to further stress out the quality of the double
discretization approach, we test it using this unprojected operator for
smoothing, giving us CONS DD.
If not explicit\-ly mentioned otherwise, all \ip and \ax based results use
quadratic polynomials.

We start our tests with an initial mesh $\mathcal{T}_{-2}$ consisting of
60 macro elements and set $L\in\{1,\ldots,6\}$. To measure the
discretization error, we introduce  the discrete $L^2$-norm
\begin{equation*}
e:=h_\ell^{\nicefrac{3}{2}} \|I_\ell(\bar{u}) - \mathbf{v}_\ell\|_2 \,,
\end{equation*}
where $I_\ell$ denotes the nodal interpolation operator, and $\mathbf{v}_\ell$
represents either the actual multigrid approximation or
the final iterate of the employed scheme. This error is equivalent
to the $L^2$-error of the associated finite element functions. In order to be
sure to have reached the asymptotic regime, we perform  10 multigrid
cycles. Results for $L=5$ are given in Fig.~\ref{fig:res_discerr_level72}.

\begin{figure}[!h]\centering
\includegraphics[width=0.48\textwidth]{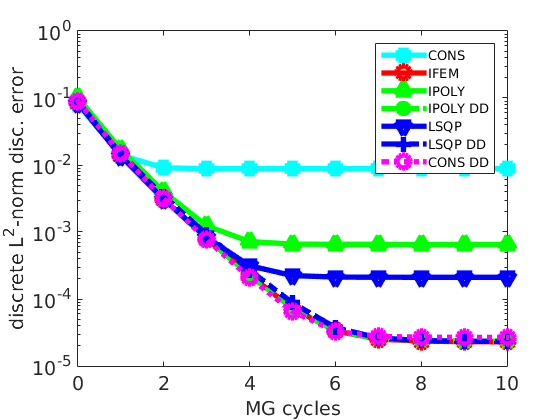} 
\caption{Development of the error in the iterates for an input mesh of 60
macro tetrahedrons and refinement level $L=5$.\label{fig:res_discerr_level72}}
\end{figure} 

\begin{figure}\centering
\includegraphics[width=0.48\textwidth]{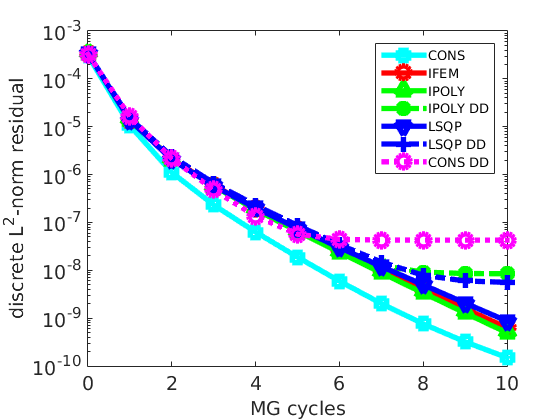} 
\caption{Discrete $L^2$ norm of the residual at $L=5$ for a
input mesh with 60 macro tetrahedrons. 
\label{fig:res_discerr_level71}}
\end{figure} 

One can see that the discretization error for CONS clearly suffers from the
insufficient resolution of the spherical geometry. The errors for \ip and \ax
are roughly a factor of 10 and 50 smaller, respectively. We observe that
\ax performs better than \ip which can be attributed to a smaller constant in
\eqref{eq:Bbounded}. 
This supports the observation from Sec.~\ref{subsec:approximation} that the
stencils are uniformly better fitted.
Note that, as was to be expected for a rather
coarse input mesh (large $H$) and multiple refinements (large $L$), the error
is dominated by the approximation error in the stencil $s$ introduced by
replacing its values by those of the surrogate polynomials.
Using the double discretization scheme, however, allows to (almost) recover the
discretization error of the \fem approach. Here, both \ip DD and \ax DD perform
equally well and even CONS DD achieves very good results.

In Fig.~\ref{fig:res_discerr_level71}, we provide the discrete $L^2$-norms of
the residuals. We observe that the convergence for the DD variants 
breaks down after a few iterations. As already indicated, this was to be
expected due to the inconsistency in smoothing and residual computation.
Whereas the residuals of all non DD-scenarios show similar
asymptotic convergence rates.

In Fig.~\ref{fig:discerr_alllevels}, we show the discrete $L^2$-error for increasing refinement levels. As  expected, \fem shows quadratic order convergence.

This is also satisfied by the DD approaches. But, as already mentioned before, this does not hold for \ip and \ax.
Nevertheless, we observe in the pre-asymptotic regime a second order error decay
 and only asymptotically, the error is
dominated by the term in $H^{q+1}$, see also \eqref{eq:apriori}. For \ax we observe second order for 
$L \leqslant 3$, while for \ip this is only the case for $L \leqslant 2$.

\begin{figure}\centering
\includegraphics[width=0.48\textwidth]{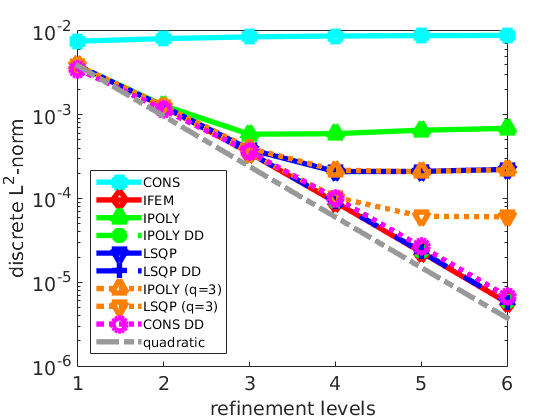}
\caption{Discretization error in discrete $L^2$ norm for different levels of refinement. 
\label{fig:discerr_alllevels}}
\end{figure}

While DD shows excellent convergence,  it also requires the
application of the expensive operator $\L_\ell$. 
Although it is used only every seventh operator call, in case of a V(3,3) cycle, 
it still dominates the time to solution. This can be seen in the scaling results in Sec.~\ref{sec:scaling}.

Motivated by the a priori estimates from Sec.~\ref{subsec:apriori},
we further investigate the dependence of \ax and \ip on the refinement
level $L$, the macro mesh-size $H$ and the polynomial degree $q$. 
The above test setting describes the coarsest  macro mesh with 60 elements. 
Further refinement in radial and tangential direction results in input meshes
$\mathcal{T}_{-2}$ with 480, 3,840 and 30,720 macro elements, respectively.
The discretization error for \fem, \ax and \ip after 10 multigrid cycles is
displayed in Tab.~\ref{tab:hHstudy} for different numbers of macro elements and
different polynomial degrees, $q=1,2,3$. These results are for $j=2$ in $m(\ell,j)$,
i.e.~for $S_\ell = M_2$. We note that using a computationally more expensive
choice of $S_\ell$ did not yield any significant differences.

As expected, we see in Tab.~\ref{tab:hHstudy} a qua\-dra\-tic convergence for
\fem. For the other variants we observe that for a finer input mesh and for a
higher polynomial degree, the accuracy obtained increases. Combinations for
which $e_{\ax}$ or $e_{\ip}$  differ relatively by more than 10 percent from $e_{\fem}$ are highlighted
by italic table entries.

We find that for \ip and $q=1$ the maximal feasible refinement level $L$
is independent of $H$. 
While for $q=2$ two refinements in $H$ allow one more refinement in $L$. And
for $q=3$ those values form a diagonal, meaning that for each 
finer $H$ also $L$ can be increased by one.
For \ax the results look even better.
As a rule of thumb we observe that for \ax the italic entries are
shifted by at least one level down compared to \ip.

\begin{table}[t!]
\caption{Discretization error for \fem and \ip/\ax (linear, quadratic and cubic) for different combinations
   of $h_\ell$ and $H$. Italic entries have a relative deviation from the reference $e_{\fem}$ of more than 10 percent.
 \label{tab:hHstudy}}
\centering
\begin{tabular}{c|*{4}{c}}
\hline
\tiny{\#mac el}& 60 & 480 & 3840 & 30720 \\
\hline
      & \multicolumn{4}{ c }{$e_{\fem}$}   \\     %\multicolumn{2}{ c }{}
\hline
%$L=1$ & 3.8e-03 & 1.2e-03 & 3.3e-04 & 8.7e-05   \\
$L=2$ & 1.2e-03 & 3.5e-04 & 8.9e-05 & 2.2e-05   \\
$L=3$ & 3.4e-04 & 9.4e-05 & 2.2e-05 & 5.6e-06   \\
$L=4$ & 9.0e-05 & 2.3e-05 & 5.7e-06 & 1.4e-06   \\
$L=5$ & 2.2e-05 & 5.9e-06 & 1.4e-06 & 3.5e-07  \\
$L=6$ & 5.7e-06 & 1.5e-06 & 3.5e-07 & 8.9e-08  \\

\hline
      & \multicolumn{4}{ c }{$e_{\ip}$ ($q=1$)}   \\     %\multicolumn{2}{ c }{}
\hline
%$L=1$ & 3.8e-03 & 1.2e-03 & 3.3e-04 & 8.6e-05 \\
$L=2$ & \textit{1.5e-03} & \textit{4.1e-04} & 8.9e-05 & 2.2e-05 \\
$L=3$ & \textit{2.1e-03} & \textit{5.6e-04} & \textit{1.1e-04} & \textit{2.8e-05} \\
$L=4$ & \textit{2.7e-03} & \textit{7.3e-04} & \textit{1.5e-04} & \textit{3.8e-05} \\
$L=5$ & \textit{3.1e-03} & \textit{8.2e-04} & \textit{1.8e-04} & \textit{4.4e-05} \\
$L=6$ & \textit{3.3e-03} & \textit{8.7e-04} & \textit{1.9e-04} & \textit{4.7e-05}       \\
\hline
      & \multicolumn{4}{ c }{$e_{\ip}$ ($q=2$)}   \\     %\multicolumn{2}{ c }{}
\hline
%$L=1$ & 3.8e-03 & 1.3e-03 & 3.4e-04 & 8.7e-05 \\
$L=2$ & 1.3e-03 & 3.6e-04 & 9.0e-05 & 2.3e-05 \\
$L=3$ & \textit{5.9e-04} & 1.0e-04 & 2.3e-05 & 5.7e-06 \\
$L=4$ & \textit{5.9e-04} & \textit{6.3e-05} & \textit{6.5e-06} & 1.5e-06 \\
$L=5$ & \textit{6.5e-04} & \textit{6.7e-05} & \textit{3.4e-06} & \textit{4.0e-07} \\
$L=6$ & \textit{6.9e-04} & \textit{7.1e-05} & \textit{3.1e-06} & \textit{1.7e-07} \\
\hline
      & \multicolumn{4}{ c }{$e_{\ip}$ ($q=3$)}   \\     %\multicolumn{2}{ c }{}
\hline
%$L=1$ & 3.8e-03 & 1.3e-03 & 3.4e-04 & 8.7e-05 \\
$L=2$ & 1.3e-03 & 3.6e-04 & 9.0e-05 & 2.3e-05 \\
$L=3$ & \textit{4.0e-04} & 9.6e-05 & 2.3e-05 & 5.7e-06 \\
$L=4$ & \textit{2.2e-04} & \textit{2.7e-05} & 5.9e-06 & 1.4e-06 \\
$L=5$ & \textit{2.1e-04} & \textit{1.2e-05} & \textit{1.6e-06} & 3.7e-07 \\
$L=6$ & \textit{2.2e-04} & \textit{9.9e-06} & \textit{6.4e-07} & \textit{1.0e-07} \\
\hline
      & \multicolumn{4}{ c }{$e_{\ax}$ ($q=1$)}   \\     %\multicolumn{2}{ c }{}
\hline
%$L=1$ & \textit{3.8e-03} & \textit{1.3e-03} & \textit{3.4e-04} & \textit{8.7e-05} \\
$L=2$ & 1.3e-03 & 3.7e-04 & 9.0e-05 & 2.3e-05 \\
$L=3$ & \textit{7.7e-04} & \textit{1.4e-04} & \textit{2.7e-05} & \textit{6.4e-06} \\
$L=4$ & \textit{8.0e-04} & \textit{1.1e-04} & \textit{1.5e-05} & \textit{3.1e-06} \\
$L=5$ & \textit{8.5e-04} & \textit{1.2e-04} & \textit{1.4e-05} & \textit{2.8e-06} \\
$L=6$ & \textit{8.8e-04} & \textit{1.2e-04} & \textit{1.5e-05} & \textit{2.9e-06} \\
\hline
      & \multicolumn{4}{ c }{$e_{\ax}$ ($q=2$)}   \\     %\multicolumn{2}{ c }{}
\hline
%$L=1$ & \textit{3.8e-03} & \textit{1.3e-03} & \textit{3.4e-04} & \textit{8.7e-05} \\
$L=2$ & 1.3e-03 & 3.6e-04 & 8.9e-05 & 2.3e-05 \\
$L=3$ & \textit{3.9e-04} & 9.4e-05 & 2.3e-05 & 5.7e-06 \\
$L=4$ & \textit{1.9e-04} & 2.5e-05 & 5.6e-06 & 1.4e-06 \\
$L=5$ & \textit{1.8e-04} & \textit{1.3e-05} & 1.4e-06 & 3.5e-07 \\
$L=6$ & \textit{1.9e-04} & \textit{1.3e-05} & 3.8e-07 & 8.2e-08 \\
\hline
      & \multicolumn{4}{ c }{$e_{\ax}$ ($q=3$)}   \\     %\multicolumn{2}{ c }{}
\hline
%$L=1$ & \textit{3.8e-03} & \textit{1.3e-03} & \textit{3.4e-04} & \textit{8.7e-05} \\
$L=2$ & 1.3e-03 & 3.6e-04 & 8.9e-05 & 2.3e-05 \\
$L=3$ & 3.5e-04 & 9.4e-05 & 2.3e-05 & 5.7e-06 \\
$L=4$ & \textit{1.0e-04} & 2.4e-05 & 5.7e-06 & 1.4e-06 \\
$L=5$ & \textit{6.9e-05} & 5.8e-06 & 1.4e-06 & 3.5e-07 \\
$L=6$ & \textit{7.3e-05} & \textit{1.9e-06} & 3.3e-07 & 8.7e-08 \\
\hline
\end{tabular}
\end{table}

These observations confirm our theory that the discretization error is a
combination of two parts: one fine scale part $h_\ell^2 $ and one depending on
the coarse meshsize $H$, but being of higher order depending on the
selected polynomial degree.

Real geophysical applications will typically be executed on HPC systems with
${\mathcal O} (10^5)$ cores. Each core will then manage at least one, but often also more than
one macro element, imposing a small upper  bound on $H$ and thus a
reasonable large bound on  $L$.
Hence, in these application
we will be automatically in an $(L,H)$-parameter regime that does not exceed the critical
$L$. Consequently our two-scale approach is perfectly suited for this kind
of applications.

% -----------------------------------------------------------------------------
\subsection{Cost considerations}
% -----------------------------------------------------------------------------
%
Besides the quality of the approximation, we also have to consider its cost. We
split our considerations into two parts, the cost for applying and the cost
for setting up the surrogate operator. While our two approaches \ip and \ax have different setup costs,
the cost for the evaluation of the surrogate  stencil operator is the same in both surrogate approaches. Moreover once the stencil weights are evaluated, the application of the actual stencil  is
independent of the chosen approach. 
% -----------------------------------------------------------------------------
\subsubsection{Cost for operator evaluation}\label{subsubsec:operatorEval}
% -----------------------------------------------------------------------------
%
To give a quantitative idea Tab.~\ref{tab:flops} provides the number of
FLOPs for a single Gauss-Seidel relaxation step, which involves the equivalent
of one stencil application, see \eqref{eq:iteration-scheme} below.
We base our considerations, w.r.t.~polynomial evaluation, on the incremental
approach.
Note that the traditional FLOPs measure can only give an impression on the
arithmetic
load of the different variants. The actual execution times depend on various
implementation details, see also 
Sec.~\ref{sec:kernel_opt} where we discuss how to optimize the polynomial
evaluation with respect to run-time. 

For the constant case, we get in Tab.~\ref{tab:flops} the minimal value of
29 FLOPs, as the same pre-computed stencil can be used at all nodes. As
mentioned above this cost for the stencil application remains the same in the
other approaches. However, the extra cost for computing the proper finite
element or surrogate stencil entries occur.
As a small detail note that in the constant case HHG employs a multiplication
with the reciprocal of the central stencil weight in
\eqref{eq:iteration-scheme} which gets pre-computed in the setup phase.
For the other approaches we need to perform a floating point division,
% by a real number,
which is more costly.

Setting up the stencil in the \fem approach with projected coordinates involves
firstly computing for the current node the element matrices of its surrounding
24 elements. For this we use a C++ code that was auto-generated and optimized
using the \mbox{FEniCS} Form Compiler (FFC), see \cite{LogOlgRogWel:2012},
which results in 53 FLOPs per element matrix. 
Stencil weights are obtained by summing
up the associated entries in the element matrices. In our 15-point-stencil, we
have 6 weights belonging to edges attached to 4 neighboring elements and
8 weights belonging to edges attached to 6 neighboring elements, see
Fig.~\ref{fig:uniform}. Hence, computation of the non-central stencil weights
adds another 58 FLOPs. The central weight can either be assembled in the same
fashion (23 FLOPs) or via the row sum property (13 FLOPs). Altogether this
leads to the cost of 1353 resp.~1343 FLOPs given in Tab.~\ref{tab:flops}.
Thus, this approach is nearly a factor of $50$ more costly than the constant
one. Though, to be fair, we remark that our implementation does %currently
not take into
account the fact that each element matrix is shared by four nodes. However,
exploiting this is not straightforward as the point-smoother we are considering
here requires an assembled central stencil weight and re-using element matrices
for different nodes introduces additional overhead to manage the related
dependency information.

\begin{table}
\caption{Number of FLOPs for a single stencil application in a Gauss-Seidel
relaxation step for a scalar operator in HHG. For fair comparison DD cost
is averaged over cheap smoothing and expensive residual
computation.\label{tab:flops}}
\begin{tabular}{l|l}
   \hline
 Constant              & \makebox[13mm][r]{  29} \\
 \fem (direct)         & \makebox[13mm][r]{1353} \\
 \fem (row sum)        & \makebox[13mm][r]{1343} \\
 \ip/\ax (na\"{i}ve, $q=2$)   & \makebox[13mm][r]{ 378} \\
 \ip/\ax (increment)   & \makebox[13mm][r]{29+15q} \\
 %\ip/\ax cubic (incr.) & \makebox[10mm][r]{  74} \\
 CONS DD  &  \makebox[13mm][r]{$\frac{1343 + 29\nu}{\nu+1}$}\\
   \hline 
\end{tabular}
\end{table}

We now turn to our proposed surrogate stencils. Here we need to add on top
of the cost for applying the stencil the evaluation of the polynomials to
obtain its weights. If we employ the na\"{i}ve approach, i.e., a
straightforward evaluation with respect to its monomial basis representation,
we end up with 378 FLOPs for a tri-variate quadratic polynomial.
Here, we assumed the sum rule is employed for the
central weight. This already gains a factor of 3.5 compared to
\fem, but still requires over ten times more FLOPs than in the constant case.
Now, for the incremental approach, asymptotically, only two operations are
needed for the evaluation of each polynomial. Thus,
the total number of FLOPs decreases to 59, which is only about twice  the
optimal value of the constant setting. Note that in this case, evaluation of
the central entry via its polynomial is, of course, cheaper than the sum rule.
Using higher order polynomials of degree $q$, we asymptotically need $q$
operations and an additional cost factor  of 
\begin{equation} \label{eq:costf}
C(q) = 1 + \frac{15 q}{29} \approx 1 + \frac{q}{2}.
\end{equation}
So the cubic \ip / \ax approach, overall,
is a factor of 1.25 more expensive than the quadratic one.

To quantify the cost of a CONST double discreti\-zation approach in comparison with a pure \ip / \ax approach ($q=2$), we take the
average over one residual computation and a total of $\nu$ pre- and
post-smoothing steps. For a typical choice of $\nu$ ranging between 2 and 8 the resulting cost factor  is
roughly  between 8 and 3.
This demonstrates that DD is an option that should seriously be considered, but
overall is still dominated by the expensive residual computation.

In order to reach the best performance it is better to either decrease $H$ or
to increase the polynomial degree. For smaller runs on workstations or on
mid-sized clusters the latter may be the preferred choice in order to keep
the number of macro elements as low as possible.
Whereas for large scale simulation on high performance computers, quadratic
polynomials seem most suitable. Here, $H$ is decreased almost naturally as
the number of cores that are to be used is coupled to the number of macro
elements.

% -----------------------------------------------------------------------------
\subsubsection{Cost for setup}
% -----------------------------------------------------------------------------
%

\begin{table*}[ht]
\caption{Number of sampling points $|S_L|$ for a single macro
element on finest level $(L=5)$, sum of sampling points over all
levels, time $t_{\text{sample}}$ for stencil computation at sampling
points, $t_{\text{linalg}}$ for computation of polynomial coefficients,
$t_{\text{setup}}$ for total setup, $t_{\text{cycle}}$ for a single
V(3,3)-cycle, $p_{\text{cycle}}$ percentage of $t_{\text{setup}}$
w.r.t.~$t_{\text{cycle}}$; for \fem $p_{\text{cycle}}$ relates an \fem to a \ax cycle;
for \ax values depend on strategy
$m(\ell,j)$ for choosing $S_\ell$; times are given in seconds.
\label{tab:setup}}
\begin{tabular}{@{}lc|rrrrrrr}
\hline
\rule[-1.5mm]{0mm}{6mm}%
      & $j$
      & \multicolumn{1}{c}{$|S_L|$}
      & $\sum_{\ell=1}^L |S_\ell|$
      & \multicolumn{1}{c}{$t_{\text{sample}}$}
      & \multicolumn{1}{c}{$t_{\text{linalg}}$}
      & \multicolumn{1}{c}{$t_{\text{setup}}$}
      & \multicolumn{1}{c}{$t_{\text{cycle}}$}
      & \multicolumn{1}{c}{$p_{\text{cycle}}$} \\
\hline
\rule[-1mm]{0mm}{5mm}\ip
         & ---      &      10 &      50 &   0.01 &  0.001 & 0.01 & 16.8 &    0.1\% \\
\hline 
\rule{0mm}{4mm}   
         & $2$      &     455 &   1,855 &   0.52 & 0.02 &   0.56 & 16.8 &     3\%  \\ 
  \ax    & $\ell-1$ &  39,711 &  44,731 &  11.06 & 0.37 &  11.53 & 16.8 &    67\%  \\ 
         & $\ell$   & 333,375 & 378,071 &  96.22 & 4.10 & 102.66 & 16.8 &   611\%  \\ 
\hline
\rule[-1mm]{0mm}{5mm}\fem
      & ---      & 333,375 & 378,071 &   ---  &  ---  &   ---  & 872.8 &    5195\%\\ 
\hline 
\end{tabular}
\end{table*}

We continue by considering the setup phase. This consists in the 
computation of the
sampling values and the solution of the resulting linear systems for \ip or
least-squares problems for \ax.
For the \ip approach with its small number of sampling points and very low
dimensional linear system we expect the cost to be negligible. In the \ax
approach, however, the situation might be different as problem size strongly
varies with the choice of the sampling set $S_\ell$.

To quantify this setup cost we report in Tab.~\ref{tab:setup} on run-times
related to the setup phase and compare them to the execution time
for a single V(3,3)-cycle with our surrogate operator. For this we restrict
ourselves to the case of quadratic polynomials and  deliberately consider the worst case scenario of a monomial basis. 
Measurements were performed for a serial run on a single node on the SuperMUC cluster,
cf. Sec.~\ref{sec:scaling} for technical details.
The code was compiled with the Intel Compiler version 15.0.
Linear systems for \ip were solved using the DGESV routine of LAPACK, while
for the least-squares problems in the \ax case we used DGELS. In both cases 
the LAPACK implementation from the Intel Math Kernel Library (MKL) in version
11.3 was employed. The macro mesh consisted of 60 elements, and we used $L=5$
levels of refinement. Reported values are the minimal ones over a collection
of test runs.

Table \ref{tab:setup} lists the size of the sampling set for the finest
level and the total number of sampling points over levels 1 to $L$ for
\ip and \ax with different choices for $m(\ell,j)$, see
\eqref{eqn:samplingSetFunctions}. The value $t_{\text{sample}}$ represents the
time spent with computing stencil weights at the sampling points from the
finite element formulation with projected coordinates.
This was performed as described in
Sec.~\ref{subsubsec:operatorEval} for the \fem approach. The values
$t_{\text{linalg}}$ and $t_{\text{setup}}$ give the time required for the solution
of the linear systems resp.~least-squares problems and the total time of the
setup phase for computing the polynomial coefficients. The last two columns
provide
information to relate the cost of the setup to those for the evaluation and
application of our surrogate operator. $t_{\text{cycle}}$ re\-pre\-sents the
run-time for a single $V(3,3)$--cycle using, in the case of \ip and \ax, the
surrogate operator and $p_{\text{cycle}}$ gives the percentage of $t_{\text{setup}}$
with respect to $t_{\text{cycle}}$.

For comparison the last row shows for the \fem approach the number of
sampling points, the cost for a $V(3,3)$--cycle using this operator 
and percentage of these cost compared to 
$t_{\text{cycle}}$ of the surrogate operator.
Remember that in this approach there is no setup phase, but instead the
stencils are re-computed on-the-fly for each stencil application. This
must be done for every node which corresponds to \ax using $j=\ell$.

Examining the results we see that, as expected, the cost for \ip is marginal,
only 0.1\% of the time required for a
single V(3,3)-cycle. This changes dramatically, when we employ for \ax all
available points $\mathcal{M}_\ell$ on level $\ell$ as sampling points, i.e., ~we
choose $j=\ell$ for $m(\ell,j)$.
However, already for $j=(\ell-1)$ this reduces to 67\%, i.e., two thirds of one
multigrid cycle. For $m(\ell,j)\equiv 2$ this reduces further to only 3\%
and this choice seems to be the optimal trade-off between accuracy and cost.
We point out that the relative cost is measured with respect to one MG-cycle
based on the surrogate operator evaluation. 

Note that for all choices in the \ax case the majority, between 90 and 95\%,
of the time for the setup phase is spent in the evaluation of the stencil
weights at the
sampling points and only the remaining fraction is required for solving the
least-squares problems.

Our previous accuracy experiments had shown that even using a sampling set
$S_\ell$ as coarse as $M_2$ can give satisfactory results and thus the relative setup
cost for our surrogate approaches is quite small compared to the total cost of the solver.

%% file: 5_kernel-optimization.tex
% \svnid{$Id: 5_kernel-optimization.tex 3209 2016-08-11 12:07:52Z simon.bauer $}

\section{Kernel optimization}
\label{sec:kernel_opt}

In Sec.~\ref{sec:numstudy}, the run-time $t_{\text{cycle}}$ of a single
V-cycle was measured in a serial setting. Comparing the measured 16.8 seconds
to the 5.0 seconds required for a V-cycle with a constant stencil application
shows that our straightforward implementation based on  incremental polynomial evaluation underperforms. We exceed the theoretically expected cost factor given in \eqref{eq:costf} by roughly a factor of 1.68.

Thus we now focus on a node-level performance analysis and a more sophisticated re-implementation of the surrogate polynomial application. 
 As we are mainly interested in extreme
scale simulations, this will be done for the quadratic case.

In the following, we study the optimization of the GS smoother in detail.
Updating the value  of the current iterative solution at  node
$I$ can  be formulated as
\begin{equation}
  \label{eq:iteration-scheme}
  \mathbf{u}_I = \frac{1}{s_{mc}^I} \left( \mathbf{f}_I - \sum_{w \in \mathcal{W} \setminus\{mc\}} s_w^I
  \mathbf{u}_{I_w}\right)
\end{equation}
and, naturally, involves all 15 stencil weights $s_w^I$. Here $I_w$ stands for
the node which is the closest one to $I$ when we move in direction $w$.
We will refer to this as a \emph{stencil-based update (SUP)} below.

Our Gauss-Seidel smoother uses a lexicographic ordering. Within HHG this
implies that the function values $\mathbf{u}_I$ are updated in the following
ordering, as shown in Fig.~\ref{fig:stencilHHG}. We traverse the nodes of a
volume primitive from west to east, going line by line from south to north and
plane by plane from bottom to top.
Hence, when updating the value at node $I$,  the values at the neighboring
nodes with
indices $bc$, $be$, $bnw$, $bn$, $ms$, $mse$, and $mw$ have already been
updated during the current sweep, whereas values at nodes with indices
$me$, $mnw$, $mn$, $ts$, $tse$, $tw$, and $tc$ still need to be updated.
Note that due to the hybrid nature of the smoother this is true for the
majority of, but not for all nodes, as those near the boundary couple to lower
dimensional primitives. Consequently, the only data dependency which must be
considered during the update inside a line is $\mathbf{u}_{I_{mw}}$. This value
is only available for updating $\mathbf{u}_I$ after the update at the previous
node, i.e., $I_{mw}$, was computed. We further remark 
that due to the lexicographic ordering, the $tc$ value is the only one that has not been
touched in a previous update. Hence, it has not been loaded into the cache so far.

Our analysis  revealed that the recursive dependency in
~\eqref{eq:iteration-scheme}
dominates the execution time for the update. Furthermore, it hinders certain
optimizations, like vectorization, and thereby prohibits exploiting the full
floating point performance of the core.

To circumvent this problem, the update formula \eqref{eq:iteration-scheme} for all nodes within one line
of the volume primitive is re-arranged. For this, we split it
into a non-recursive and a recursive part $\mathbf{u}_I = \tau + \rho$ with
\begin{align}
\tau &= \frac{1}{s_{mc}^I} \Big( \mathbf{f}_I -
  \sum_{\substack{w \in \mathcal{W}\\w \notin \{mc,mw\}}} s_w^I \mathbf{u}_{I_w}
  \Big)
\label{eq:iteration-scheme2-tau}\\[1ex]
  \rho &= - \frac{1}{s^I_{mc}}s^I_{mw} \mathbf{u}_{I_{mw}}
\enspace.
\label{eq:iteration-scheme2-rho}
\end{align}
Firstly, the non-recursive part $\tau$ is pre-computed and stored in a
temporary array,
tmp in Alg.~\ref{alg:optGS} below, large enough to hold the results for a
complete line. As an incremental polynomial evaluation does not allow complete
vectorization, we instead represent the stencil weights $s^{d}_w$
for $d\in\mathbb{N}_0$ as
\begin{equation}
\label{eq:linearstencilupdate}
s^{d}_w = \hat{p}_w(0) + d\, \delta \hat{p}_w(0) + \frac{(d-1)d}{2} \delta k_w
\end{equation}
with $\hat p_w(0) :=  {\mathcal P}^\alpha s_w(0,j_0,k_0) $ and
\begin{align*}
\begin{split}
\delta \hat p_w(0) &:=  \left.\frac{d  {\mathcal P}^\alpha s_w(i,j_0,k_0)}{d i} \right|_{i=0} + \frac{\delta k_w}{2}\, ,\\
\delta k_w & := \left.\frac{d^2  {\mathcal P}^\alpha s_w(i,j_0,k_0)}{d i^2} \right|_{i=0} \, .
\end{split}
\end{align*}

Here, the index $d$ denotes the position of a node inside the current line,
starting at 0. For the very first line, the values $\hat p_w(0)$, $\delta \hat p_w(0)$ and
$\delta k_w$ have to be initialized. For all further lines even these
values are obtained by updates analogue to 
\eqref{eq:linearstencilupdate} in $j$ or $k$ direction, respectively.

For performance reasons, \eqref{eq:iteration-scheme2-tau} is further split into
two
loops, where in the first all stencil elements in $\mathcal{W}_1$ are treated
and then those in $\mathcal{W}_2$. The two sets are defined as
\begin{align*}
\mathcal{W}_1 &=\{ me, mnw, mn, ts, tse, tw \}, \\
\mathcal{W}_2 &= \{ tc, bc, be, bnw, bn, ms, mse \} .
\end{align*}
On the tested architecture, see below, this splitting avoids problems with
register spilling, see e.g.~\cite{Comer:2005:PPH}.
The distribution of the sets is based on hardware dependent performance investigations as detailed in
Sec.~\ref{subsec:singelcoreperf} and has no geome\-tric meaning.

Now in a second step, the recursive part $\rho$,
 see \eqref{eq:iteration-scheme2-rho}, is computed and the stencil
update completed by combining it with the pre-computed value $\tau$ stored in the
temporary array.
Unrolling this last loop by a factor of four or eight increases the performance
on the evaluated hardware architecture.

The full pseudo code for the optimized GS smoother is displayed
in Alg.~\ref{alg:optGS}.

\begin{algorithm}
\renewcommand{\algorithmiccomment}[1]{// \emph{#1}}
\begin{algorithmic}[1]
\scriptsize
\FOR{line $(k,j,0:n-1) \in $ TET}   
   \STATE initialize/update $s_w^0$, $\delta\hat p_w^0$
   \STATE \COMMENT{setup stencils along line and store in tmp array tmp[d],
   tmp\_mw[d], and tmp\_mc[d], except recursive part mw.}
   \FOR[loop 1: vectorized]{d = 0, ... n-1}  
      \STATE tmp[d] = $f^d$; tmp\_mc[d] = 0
      \FOR[completely unrolled]{$w_1 \in \mathcal{W}_1$}
         \STATE compute stencil weight $s_{w_1}$ \COMMENT{evaluate \eqref{eq:linearstencilupdate}}
         \STATE tmp[d] -= $s^{d}_{w_1} \times u^{d}_{w_1}$,\enspace tmp\_mc[d] -= $s^{d}_{w_1}$
      \ENDFOR 
    \ENDFOR 

   \FOR[loop 2: vectorized]{d = 0, ... n-1}  
      \FOR[completely unrolled]{$w_2 \in \mathcal{W}_2$}
         \STATE compute stencil weight $s^{d}_{w_2}$ \COMMENT{evaluate \eqref{eq:linearstencilupdate}}
         \STATE tmp[d] -= $s^{d}_{w_2} \times u^{d}_{w_2}$,\enspace tmp\_mc[d] -= $s^{d}_{w_2}$
      \ENDFOR 

      \STATE compute stencil weight $s^{d}_{mw}$ \COMMENT{evaluate \eqref{eq:linearstencilupdate}}
      \STATE tmp\_mc[d] -= $s^{d}_{mw}$,\enspace  s0 = 1 / tmp\_mc[d]

      \STATE tmp\_mw[d] = s0 $\times s^{d}_{mw}$,\enspace tmp[d] = tmp[d] $\times$ s0
   \ENDFOR

   \STATE \COMMENT{compute recursive part to finish application of the stencil(s)}
   \FOR[loop 3: 4- or 8-way unrolled]{d = 0, ... n-1}  
      \STATE $u^{d}_{mc}$ = tmp[d] - tmp\_mw[d] $\times$ $u^{d}_{mw}$
   \ENDFOR
\ENDFOR
\end{algorithmic}
\caption{optimized GS smoother}
\label{alg:optGS}
\end{algorithm}
%
% -----------------------------------------------------------------------------
\subsection{Single core performance}
\label{subsec:singelcoreperf}
% -----------------------------------------------------------------------------
%
In the following, we will conduct a detailed performance investigation. This
will be based on performance modeling and compared to actual measurements.
Note that we do not include the initialization in line 2 of the algorithm and
that the setup phase for computing the polynomials representing the weights of
our surrogate operator is considered neither in the model nor in the
measurements.
The latter is justifiable as for reasonable choices of $m(\ell,j)$ the
cost is negligible.

\paragraph{Benchmarked system}

The execution is modeled for and evaluated on the Haswell based Intel Xeon
E5-2695 processor with $14$ cores, which was configured in cluster on die mode,
i.\,e.\ seven cores per NUMA domain.
This is the same system as used in the SuperMUC Phase~$2$ cluster, on which the
scaling experiments of Sec.~\ref{sec:scaling} were conducted. The only
difference is that the processor here was only clocked at $2.3$\,GHz instead of
$2.6$\,GHz on the SuperMUC.

One core of the Intel Xeon comprises several execution ports, which can
perform: floating point (FP) multiply \& FP fused multiply add (FMA),
FP addition \& FP FMA, FP division, load, and store.
Transferring a cache line between L1/L2 cache or L2/L3 cache takes
$2$~cycles (cy) on Haswell, respectively. As the system sustains $24.2$\,GiB/s
memory bandwidth from one NUMA domain, it takes
$5.7$\,cy at $2.3$\,GHz to transfer one cache line between L3 cache and memory.

\paragraph{Performance modeling}
Our performance ana\-ly\-sis is based on the Execution--Cache--Memory (ECM)
model~\cite{Hager2012}. In contrast
to the simpler Roofline model~\cite{Williams2009}, which only takes into
account the (number of) floating point operations performed for a certain
amount of data transferred between core and memory, the ECM mo\-del considers
the complete path of the memory hierarchy including caches.
Furthermore the complete execution inside the core is examined, which, besides
floating point operations, includes also, e.g., address generation
and register spilling.
As in the actual memory hierarchy, data is transferred in the granularity of
cache lines, the modeling is based on this unit.
On the target architecture a cache line comprises $64$\,bytes, which is for the
stencil considered equivalent to eight updates. In the following, the modeling
of the execution and the data transfers inside the memory hierarchy are
outlined.

\begin{table*}
\caption{Cycle numbers reported by IACA for the three loops normalized to eight
  stencil-based updates. Given are the maximum values for each execution unit
  of a certain category.
  The reported duration of $40$\,cy from IACA for loop 3 is higher than the
  occupancy of the execution ports. This is due to the data dependencies, which
  hinder efficient instruction pipelining.
\label{tab:iaca}
}\centering
\begin{tabular}{ll|crrrrrcc}
   \hline
         && max && FP division & FP mul & FP add && load & store \\
   \hline
  loop 1 && 22 &&  0  & 20 & 20 && 22 & 2 \\
  loop 2 && 56 && 56  & 32 & 34 && 36 & 4 \\
  loop 3 && 40 &&  0  &  4 &  4 && 12 & 8 \\ 
   \hline 
\end{tabular}
\end{table*}

\paragraph{Modeling of execution}

For modelling the execution of the optimized implementation inside the core, the
\emph{Intel Architecture Code Analyzer} (IACA) is used, see \cite{intel-iaca}.
The tool generates a scheduling of the instructions over the available
execution ports on a certain micro-architecture for a given binary.
This was also used to derive the optimal distribution of the two sets $\mathcal{W}_1$
and $\mathcal{W}_2$.
Here it is assumed that all data can be fetched from L1 data cache.
IACA can operate in two modes: throughput and latency.
In throughput mode it is assumed that the considered loop iterations are
independent and can overlap, while in latency mode, it is assumed that
operations between iterations cannot overlap. Analyzing the three loops in our
implementation with throughput mode resulted in the closest agreement of
prediction and measurement.

The values reported by IACA for eight sten\-cil-based updates, our base unit
from above, on the Haswell
micro-architecture are listed in Tab.~\ref{tab:iaca} for the two vectorized
loops (loop~1 and 2) as well as the loop containing the recursive component
(loop~3).
To ease the modeling process the initialization and setup of the stencil
weights in line~$2$ of Alg.~\ref{alg:optGS} are not considered.
Whereas loop~1 and~2 are limited by the occupancy of the execution ports,
IACA reports for loop~3 as bottleneck the \textit{loop latency} with 40\,cy.
This indicates that due to data dependencies resulting from the recursion in
loop~3 
instructions cannot be fully pipelined and stall until their corresponding 
operands become available.

A more detailed overview is presented in Fig.~\ref{fig:ecm} which displays the
occupancy of the execution units.

\paragraph{Modeling of cache and memory transfers}

The amount and duration of cache line transfers inside the memory hierarchy for
eight stencil-based updates beyond L1 data cache are modeled manually.
Those numbers depend on the refinement level $\ell$ of the macro tetrahedron, as it
determines how many data items can be re-used from cache.
If two planes of the tetrahedron fit concurrently into a cache level, i.e.,
the so called \textit{layer condition}~\cite{stengel-2015} is fulfilled,
then only the leading
stencil neighbor $u^I_{tc}$ must be loaded from a higher cache level, the last
stencil neighbor $u^I_{bc}$ gets evicted, and the remaining neighbors can be
fetched from this cache level.
In our case, we focus on $\ell = 5$ and for simplicity consider
the case where the layer condition is only fulfilled in L3 cache.
For L1 and L2 cache we assume, that in each cache level only six rows of the
tetrahedron, i.e., two rows from top, middle, and bottom plane each, can be
kept concurrently.

With this assumptions between L1/L2 cache and L2/L3 cache five cache lines are
transferred, respectively: one cache line from each plane containing $u^I_{tc}$,
$u^I_{mn}$, and $u^I_{bn}$ is loaded (three loads), one cache line containing
the corres\-ponding element from the $f$ array (one load), and the modified cache
line containing $u^I_{mc}$ gets evic\-ted (one store).
As in L3 cache the layer condition is fulfilled, instead of three cache lines
only the one for the top plane containing $u^I_{tc}$ must be loaded (one load). 
Everything else stays the same, which results in three cache line transfers.
Concerning the arrays used for vectorization to store the temporary results,
they are small enough to fit with the other data into L1 cache.
Further the \emph{least recently used} cache policy, commonly used on current
architectures, ensures that they get not evicted.
\paragraph{Results}
The timings resulting from the IACA analysis combined with the manual modeling
of the data transfers in the memory hierarchy are summarized in
Fig.~\ref{fig:ecm}.
Note that this does not show an exact scheduling of the different execution
phases and only sums up the duration of the different parts, which are involved
during execution.
The predicted duration of $132$\,cy for eight stencil-based updates by the ECM model 
is $\approx 15$\% off from the measurement of the executed kernel.

\begin{figure*}\centering
\includegraphics[width=0.98\textwidth,clip=true]{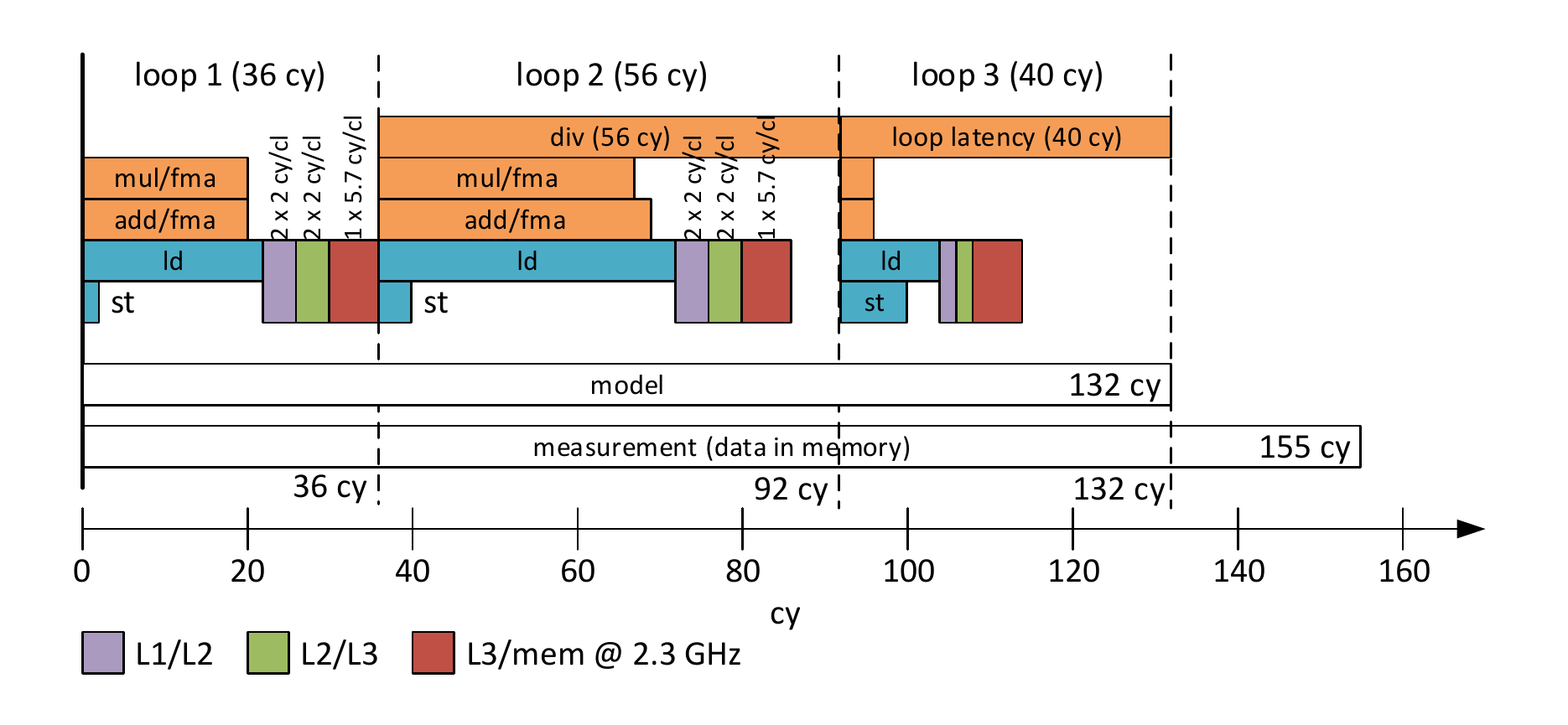}
\caption{Duration of the different phases involved during eight stencil-based
updates as obtained by the ECM model.
Shown are the durations of executing loop 1, 2, and 3, as well as the time
spend with cache line transfers inside the memory hierarchy.
For the loops the occupancy of the different execution ports are shown:
multiplication/fused multiply add (mul/fma), addition/fused multiply add
(add/fma), division (div), loads (ld), stores (st). Considered in neither
model nor measurement is the setup of the initial stencil weights.
\label{fig:ecm}}
\end{figure*}

\subsection{Socket performance}

When increasing the core count the ECM model assumes that only shared resources
can become a bottleneck. In the case of the benchmarked Haswell system this
concerns the segmented L3 cache and the memory interface. However, the
bandwidth of the L3 cache scales linearly with the number of cores. Thus, only
the memory interface could limit performance. Consequently the performance
should increase linearly with the number of cores up to the point where
the memory interface becomes continuously occupied.
The model prediction and the corresponding measurement are both shown in
Fig.~\ref{fig:ecm-scaling}.
The code does not saturate the memory bandwidth (green line with triangles)
completely.
This upper limit is given as ratio of the achievable memory
bandwidth of $24.2$\,GiB/s with the bytes transferred between L3 and memory required
for one stencil-based update.
If the layer condition is fulfilled in L3 cache one stencil-based update
requires 24 bytes: 
load of $u^I_{tc}$, load of one element from $f$ array,
and eviction of updated $u^I_{mc}$.
Utilizing also the SMT thread of each core (single filled red square) increases
performance slightly. 
This has two reasons.
Firstly the memory interface is not continuously utilized and secondly still
execution resources on each core are available.
The resources are available because of pipeline bubbles, i.e., time slots with
no operation in the execution pipeline, due to the division in loop~$2$ and the
recursion in loop~$3$.

\begin{figure}\centering
\includegraphics[width=0.45\textwidth,clip=true]{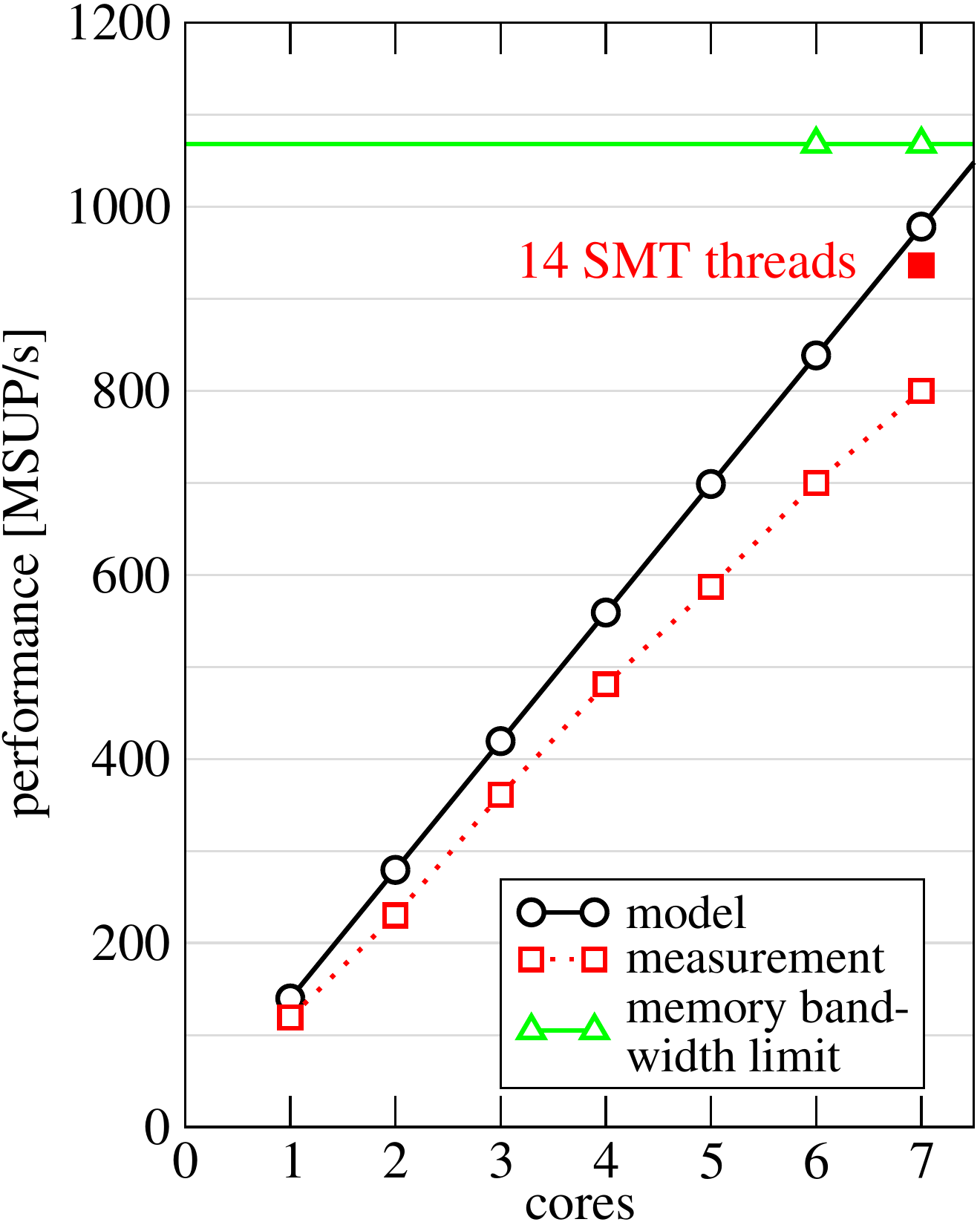}
\caption{Performance of the optimized code (in MSUP/s, mega stencil-based
updates per second) compared with the ECM model prediction.
\label{fig:ecm-scaling}}
\end{figure}

%% file: 6_scaling-results.tex
% \svnid{$Id: 6_scaling-results.tex 3205 2016-08-10 06:00:55Z barbara.wohlmuth $}
\section{Weak and strong scaling}
\label{sec:scaling}

\begin{figure*}\centering
\begin{tabular}{ccc}
\includegraphics[width=0.45\textwidth]{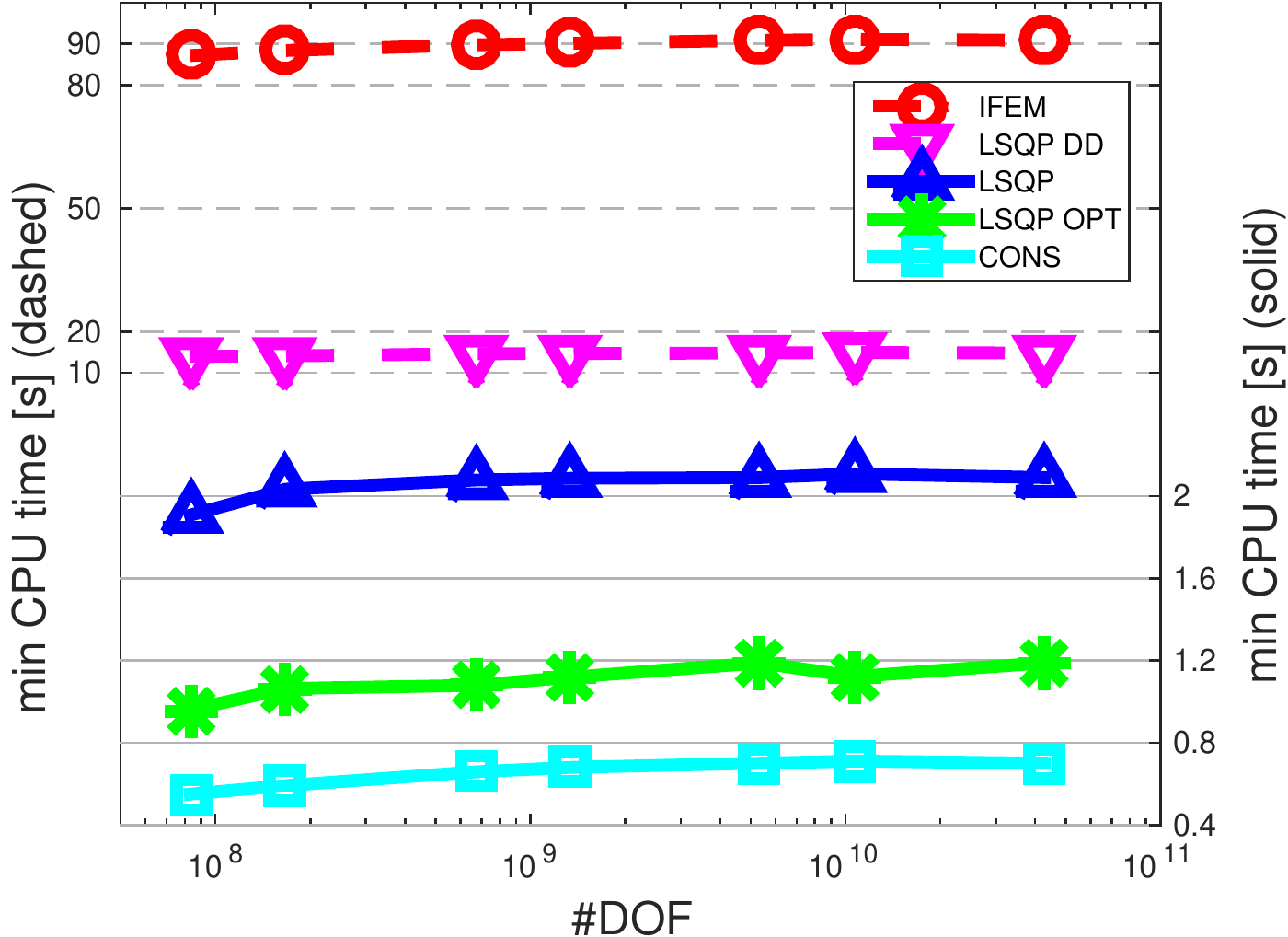} &
\includegraphics[width=0.45\textwidth]{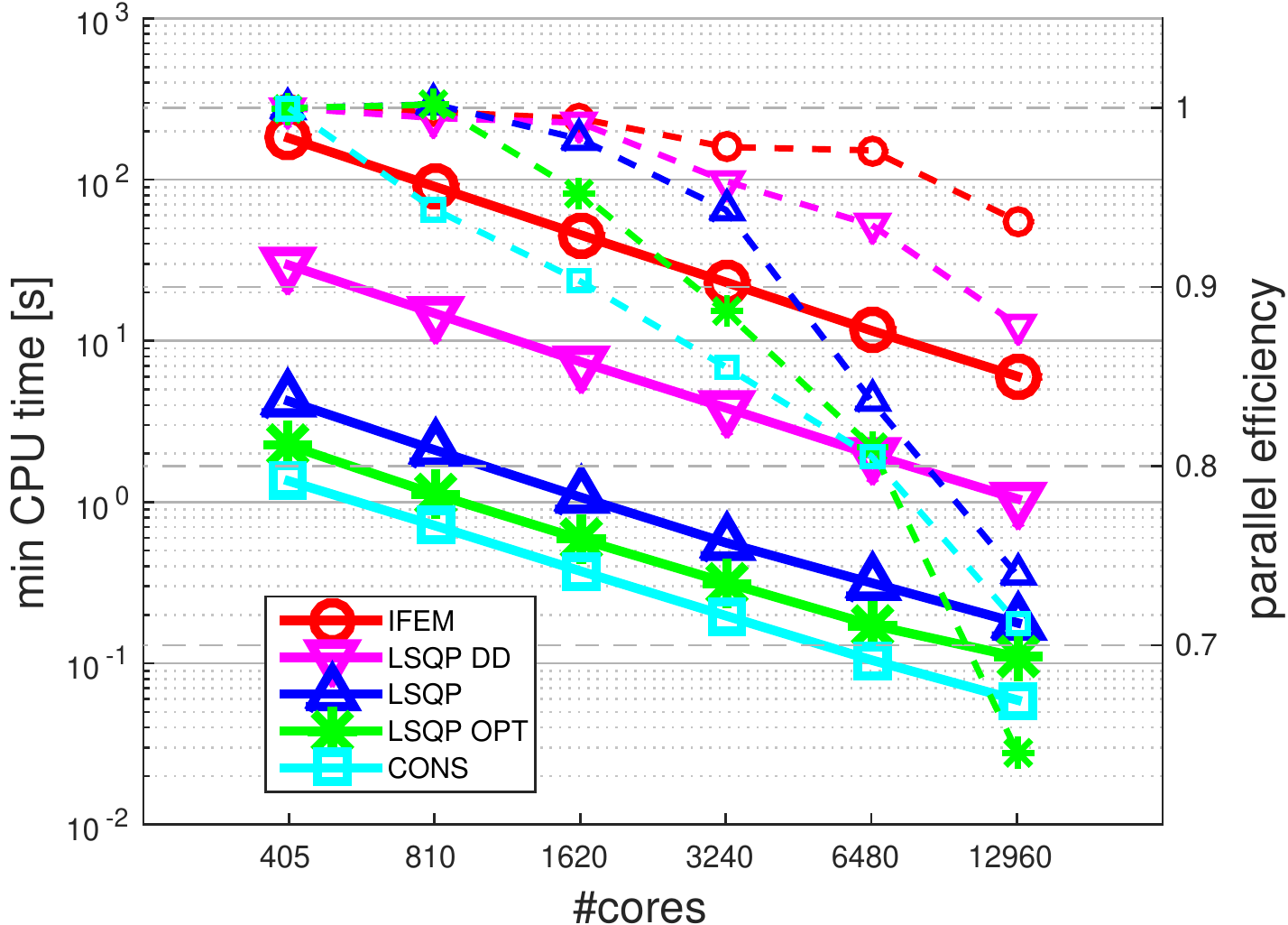} \\
\end{tabular}
\caption{Run-times of one V(3,3)-cycle. (Left) Weak scaling on SuperMUC. 16 macro elements are assigned to one core. The finest grid is fixed to $L=5$,
i.e. there are $3.3\cdot10^5$ DOFs per volume primitive. 
(Right) Strong scaling (solid) and parallel efficiency (dashed). The problem size is fixed to 12,960 macro elements and $L=5$.
In total there are $4.5\cdot10^9$ DOFs.
\label{fig:scaling}}
\end{figure*}

To demonstrate the performance and scala\-bility of our approach, we perform
weak and strong scaling runs on the supercomputer SuperMUC (Phase 2) that is
ranked number 28 in the June 2016 TOP500 list. It is a Lenovo NeXtScale system
with a theoretical peak performance of 3.58 PFLOPs/s.
Each compute node consists of two Haswell Xeon E5-2697 v3 processors where each
processor is equipped with 14 cores running at 2.6\,GHz. Per core 2.3\,GiB of
memory are provided, but typically only 2.1 GiB are available for applications. 
The nodes are connected via an Infiniband (FDR14) non-blocking tree.

For our runs we utilize the Intel C/C++ Compiler 15.0 with flags -O3 -CORE-AVX2 -fma and the Intel MPI library 5.0

As mentioned before degree two polynomials for the \ip/\ax approach provide the most efficient compromise between accuracy and computational cost 
in the HPC context.
In Fig.~\ref{fig:scaling} we present the run-times of one V(3,3)-cycle for the different implementations.
For reproducibility we take the minimum value out of ten cycles reflecting the best exploitation of the machine's capabilities.

The \cons setting serves as minimal cost reference. \ip and \ax 
differ only in the setup phase, so only \ax is considered here. For this approach,
we compare the straightforward implementation of the method with the optimized version
from Sec.~\ref{sec:kernel_opt} termed \axop. 
To demonstrate the cost of a conventional assembly of the stencils also the \fem case is shown.
Furthermore the run-times for double discretization are given.
Here the non-optimized \ax implementation is chosen for the smoother and the conventional stencil assembly
for residual. As this is dominated by the expensive residual, employing the optimized smoother does not lead 
to a significant improvement.

For the weak scaling, Fig.~\ref{fig:scaling} (left), we assign 16 macro
elements to one core and fix the refinement level to $L=5$ giving us
\mbox{$3.3\cdot10^5$} DOFs per volume primitive. The number of macro
elements and cores varies from 240 to 122,880 and from 15 to 7,680, respectively.

We observe that the \fem method is even more expensive than the theoretical
prediction from Tab.~\ref{tab:flops} indicated. This is due to additional memory
accesses and function calls which are not considered in the flop count.
The double discretization is about a factor of 7 slower than the non-optimized \ax method. 
This is in line with the flop count for a V(3,3)-cycle taking into
account the aforementioned additional cost for the residual.
The values for the straightforward implementation of the \ax method reflect the
serial run-time factor of
$(1.68\cdot C(2))$ mentioned in Sec.~\ref{sec:kernel_opt}.
Our optimized version outperforms the cost factor $C(2) \approx 2$, and the run-time increases by only a factor of $1.6$ compared to the \cons case.
Also it shows excellent weak scaling
behavior.

For the strong scaling case, Fig.~\ref{fig:scaling} (right), we fix the problem size to 12,960 macro elements and set $L=5$. In total this 
results in $4.5\cdot10^9$ DOFs. This also gives an upper bound for the number of cores that can be used.
The lower bound is determined by the available memory of 2.1\,GiB per core where we need about
100\,bytes per DOF in our implementation.

Not surprisingly, the \fem implementation shows the best parallel efficiency of above 90\% even for the largest number of cores. 
Here, most of the time is spent with the cost intense on-the-fly stencil computation.
Thus, increasing the amount of communication by increasing the number of cores does hardly affect the performance that scales almost
linearly with the number of cores. This is also the case for the DD approach.
Even though the computationally expensive residual occurs a factor of 7 less
often then in the \fem case, it is still the dominating factor. Whereas for the
\cons and for the \ax implementations
where the stencil is either already known or computed by an inexpensive evaluation of the surrogate polynomials,
we clearly see the influence of communication for the runs with larger number of cores. But nevertheless the run-time even for the largest case of the optimized \ax method 
is a factor of 10 or 55 faster than \ax DD or \fem, respectively, and only 1.8 times slower than the \cons reference. We point out that for the DD approaches, the expensive operator need only be used in the final
few iterations and thus  its cost within the solver can be further reduced.

%% file: 7_outlook.tex
% \svnid{$Id: 7_outlook.tex 3205 2016-08-10 06:00:55Z barbara.wohlmuth $}
\section{Conclusion and outlook}
\label{sec:outlook}
We introduced a novel two-scale approach for handling non-polyhedral domains in large scale simulations with matrix-free finite elements.
It is based on replacing the cost intensive 
on-the-fly computations of the node stencil
by the evaluation of a surrogate approximation polynomial of low order.
Typically for large system size, second order polynomials already guarantee
high enough accuracy and thus reduce the flop counts significantly.
This technique can be combined for a further improvement of the numerical accuracy with the
classical double discretization approach. Here the smoother
in the multigrid algorithm employs cheap lower order stencils
while the expensive high accuracy operators are only used to 
compute the residual for the coarse grid correction.

The accuracy of our approach was first examined theoretically by an a priori
estimate and secondly verified by numerical experiments. The implementation was
further optimized on the node-level and systematically analyzed by the ECM
model. Excellent performance and scalability was demonstrated on the
supercomputer SuperMUC. Although all our examples referred to the spherical
shell, this approach is completely local with respect to the macro elements
and open to general blending functions.

Future work will include different partial differential equations such as the
Stokes system, variable coefficients in the PDE and alternative stencil
approximations preserving symmetry for $\Lt$.

%% file: acknowledgements.tex
% \svnid{$Id: acknowledgements.tex 3198 2016-08-09 14:06:42Z simon.bauer $}
\subsubsection*{Acknowledgements}

This work was partly supported  by the German Research Foundation through
the Priority Programme 1648 "Software for Exascale Computing" (SPP\-EXA) and by WO671/11-1.
The authors gratefully acknowledge the Gauss Centre for Supercomputing (GCS) 
for providing computing time on the supercomputer SuperMUC at Leibniz-Rechenzentrum (LRZ).

%% file: interpolation.bbl
\begin{thebibliography}{10}

\bibitem{CHB09}
J.A. Cottrell, T.J.R. Hughes, and Y.~Bazilevs.
\newblock {\em Isogeometric {A}nalysis: {T}owards {I}ntegration of {CAD} and
  {FEA}}.
\newblock John Wiley $\&$ Sons, Ltd., 2009.

\bibitem{BTT13}
Y.~Bazilevs, K.~Takizawa, and T.E. Tezduyar.
\newblock {\em Computational Fluid-Structure Interaction: Methods and
  Applications}.
\newblock John Wiley $\&$ Sons, Ltd., 2013.

\bibitem{Kennett:2008:CUP}
Brian L.~N. Kennett and Hans-Peter Bunge.
\newblock {\em {Geophysical Continua}}.
\newblock Cambridge University Press, 2008.

\bibitem{Igel:2016:CUP}
Heiner Igel.
\newblock {\em {Computational Seismology: A Practical Introduction}}.
\newblock Oxford University Press, 2016.

\bibitem{Barrett:1994:SIAM}
R.~Barrett, M.~Berry, T.~F. Chan, J.~Demmel, J.~Donato, J.~J. Dongarra,
  V.~Eijkhout, R.~Pozo, C.~Romine, and H.~van~der Vorst.
\newblock {\em Templates for the {S}olution of {L}inear {S}ystems: {B}uilding
  {B}locks for {I}terative {M}ethods}.
\newblock SIAM, 1994.

\bibitem{Kreutzer:2014:SISC}
M.~Kreutzer, G.~Hager, G.~Wellein, H.~Fehske, and A.~Bishop.
\newblock {A Unified Sparse Matrix Data Format for Efficient General Sparse
  Matrix-Vector Multiplication on Modern Processors with Wide SIMD Units}.
\newblock {\em SIAM J. Sci. Comput.}, 36(5):C401--C423, 2014.

\bibitem{Guo01022016}
Dahai Guo, William Gropp, and Luke~N Olson.
\newblock {A hybrid format for better performance of sparse matrix-vector
  multiplication on a GPU}.
\newblock {\em International Journal of High Performance Computing
  Applications}, 30(1):103--120, 2016.

\bibitem{Dalton:2015:OSM:2835205.2699470}
Steven Dalton, Luke Olson, and Nathan Bell.
\newblock {Optimizing Sparse Matrix-Matrix Multiplication for the GPU}.
\newblock {\em ACM Trans. Math. Softw.}, 41(4):25:1--25:20, 2015.

\bibitem{2015_BiFaGrOlSc}
Amanda Bienz, Robert Falgout, William Gropp, Luke Olson, and Jacob Schroder.
\newblock Reducing parallel communication in algebraic multigrid.
\newblock {\em SIAM Journal on Scientific Computing (2015),}, 2015.
\newblock in review.

\bibitem{douglas2000cache}
Craig~C Douglas, Jonathan Hu, Markus Kowarschik, Ulrich R{\"u}de, and Christian
  Wei{\ss}.
\newblock Cache optimization for structured and unstructured grid multigrid.
\newblock {\em Electronic Transactions on Numerical Analysis}, 10:21--40, 2000.

\bibitem{kowarschik2002data}
Markus Kowarschik, Ulrich R{\"u}de, and Christian Wei{\ss}.
\newblock Data layout optimizations for variable coefficient multigrid.
\newblock In {\em International Conference on Computational Science}, volume
  2331 of {\em Lecture Notes in Computer Science}, pages 642--651. Springer,
  2002.

\bibitem{Rietbergen1996computational}
Bert van Rietbergen, Harrie Weinans, Rik Huiskes, and Ben Polman.
\newblock Computational strategies for iterative solutions of large fem
  applications employing voxel data.
\newblock {\em International Journal for Numerical Methods in Engineering},
  pages 2743--2767, 1996.

\bibitem{Bergen2004}
B.~Bergen and F.~H{\"u}lsemann.
\newblock {Hierarchical hybrid grids: data structures and core algorithms for
  multigrid}.
\newblock {\em Numer. Linear Algebra Appl.}, 11:279--291, 2004.

\bibitem{arbenz2008scalable}
Peter Arbenz, G~Harry van Lenthe, Uche Mennel, Ralph M{\"u}ller, and Marzio
  Sala.
\newblock A scalable multi-level preconditioner for matrix-free $\mu$-finite
  element analysis of human bone structures.
\newblock {\em International Journal for Numerical Methods in Engineering},
  73(7):927--947, 2008.

\bibitem{Kronbichler:2012:CAF}
M.~Kronbichler and K.~Kormann.
\newblock A generic interface for parallel cell-based finite element operator
  application.
\newblock {\em Computers and Fluids}, 63:135--147, 2012.

\bibitem{May:2015:CMAME}
D.~A. May, J.~Brown, and L.~Le Pourhiet.
\newblock A scalable, matrix-free multigrid preconditioner for finite element
  discretizations of heterogeneous {S}tokes flow.
\newblock {\em Computer Methods in Applied Mechanics and Engineering},
  290:496--–523, 2015.

\bibitem{baker2012scaling}
A.~Baker, R.~Falgout, T.~Kolev, and U.~Meier Yang.
\newblock {\em {High-Performance Scientific Computing -- Algorithms and
  Applications}}, chapter {Scaling Hypre's Multigrid Solvers to 100,000 Cores},
  pages 261--279.
\newblock Springer, 2012.

\bibitem{baker-yang-2016}
Allison~H. Baker, Axel Klawonn, Tzanio Kolev, Martin Lanser, Oliver Rheinbach,
  and Ulrike~Meier Yang.
\newblock Scalability of classical algebraic multigrid for elasticity to half a
  million parallel tasks.
\newblock In {\em Software for Exascale Computing -- SPPEXA 2013--2015}.
  Springer, 2016.
\newblock to appear.

\bibitem{bastian2008generic}
P.~Bastian, M.~Blatt, A.~Dedner, C.~Engwer, R.~Kl{\"o}fkorn, R.~Kornhuber,
  M.~Ohlberger, and O.~Sander.
\newblock A generic grid interface for parallel and adaptive scientific
  computing. {Part} ii: Implementation and tests in {DUNE}.
\newblock {\em Computing}, 82(2-3):121--138, 2008.

\bibitem{Bastian2014}
Peter Bastian, Christian Engwer, Dominik G{\"o}ddeke, Oleg Iliev, Olaf Ippisch,
  Mario Ohlberger, Stefan Turek, Jorrit Fahlke, Sven Kaulmann, Steffen
  M{\"u}thing, and Dirk Ribbrock.
\newblock {\em EXA-DUNE: Flexible PDE Solvers, Numerical Methods and
  Applications}, pages 530--541.
\newblock Springer International Publishing, Cham, 2014.

\bibitem{falgout2002hypre}
R.~Falgout and U.~Meier-Yang.
\newblock hypre: A library of high performance preconditioners.
\newblock {\em Computational Science-ICCS 2002}, pages 632--641, 2002.

\bibitem{notay2015massively}
Yvan Notay and Artem Napov.
\newblock A massively parallel solver for discrete {Poisson}-like problems.
\newblock {\em J. Comput. Phys.}, 281(C):237--250, January 2015.

\bibitem{Bielak:2005:CMES}
J.~Bielak, O.~Ghattas, and E.-J. Kim.
\newblock {Parallel Octree-Based Finite Element Method for Large-Scale
  Earthquake Ground Motion Simulation}.
\newblock {\em Computer Modeling in Engineering \& Sciences}, 10:99--112, 2005.

\bibitem{burstedde-stadler-alisic-wilcox-tan-gurnis-ghattas_2013}
C.~Burstedde, G.~Stadler, L.~Alisic, L.~C. Wilcox, E.~Tan, M.~Gurnis, and
  O.~Ghattas.
\newblock Large-scale adaptive mantle convection simulation.
\newblock {\em Geophys. J. Internat.}, 192(3):889--906, 2013.

\bibitem{Rudi2015}
Johann Rudi, A.~Cristiano~I. Malossi, Tobin Isaac, Georg Stadler, Michael
  Gurnis, Peter W.~J. Staar, Yves Ineichen, Costas Bekas, Alessandro Curioni,
  and Omar Ghattas.
\newblock An extreme-scale implicit solver for complex pdes: Highly
  heterogeneous flow in earth's mantle.
\newblock In {\em Proceedings of the International Conference for High
  Performance Computing, Networking, Storage and Analysis}, SC '15, pages
  5:1--5:12, New York, NY, USA, 2015. ACM.

\bibitem{Brandt2011}
A.~Brandt and O.E. Livne.
\newblock {\em Multigrid Techniques: 1984 Guide with Applications to Fluid
  Dynamics, Revised Edition}.
\newblock Classics in Applied Mathematics. Society for Industrial and Applied
  Mathematics, 2011.

\bibitem{Bey95}
J.~Bey.
\newblock Tetrahedral grid refinement.
\newblock {\em Computing}, 55(4):355--378, 1995.

\bibitem{FMW05}
B.~Flemisch, J.~M. Melenk, and B.~I. Wohlmuth.
\newblock Mortar methods with curved interfaces.
\newblock {\em Appl. Numer. Math.}, 54(3-4):339--361, 2005.

\bibitem{Bergen:2005:PhD}
Benjamin Bergen.
\newblock {\em {Hierarchical Hybrid Grids: Data Structures and Core Algorithms
  for Efficient Finite Element Simulations on Supercomputers}}.
\newblock PhD thesis, Technische Fakult{\"a}t der
  Friedrich-Alexander-Universit{\"a}t Erlangen-N{\"u}rnberg, 2005.

\bibitem{Gmeiner2015}
B.~Gmeiner, U.~R{\"u}de, H.~Stengel, C.~Waluga, and B.~Wohlmuth.
\newblock Performance and {S}calability of {H}ierarchical {H}ybrid {M}ultigrid
  {S}olvers for {S}tokes {S}ystems.
\newblock {\em SIAM J. Sci. Comput.}, 37(2):C143--C168, 2015.

\bibitem{Rockwood1987}
Alyn~P. Rockwood.
\newblock Generalized scanning technique for display of parametrically defined
  surfaces.
\newblock {\em IEEE Comput. Graph. Appl.}, 7(9):15--26, August 1987.

\bibitem{Rappoport1991}
Ari Rappoport.
\newblock Rendering curves and surfaces with hybrid subdivision and forward
  differencing.
\newblock {\em ACM Trans. Graph.}, 10(4):323--341, October 1991.

\bibitem{LogOlgRogWel:2012}
A.~Logg, K.~B. {\symbol{31}}lgaard, M.~E. Rognes, and G.~N. Wells.
\newblock {FFC}: the {FEniCS} {F}orm {C}ompiler.
\newblock In A.~Logg, K.-A. Mardal, and G.~N. Wells, editors, {\em Automated
  Solution of Differential Equations by the Finite Element Method}, volume~84
  of {\em Lecture Notes in Computational Science and Engineering}, chapter~11,
  pages 227--238. Springer, 2012.

\bibitem{Comer:2005:PPH}
D.~E. Comer.
\newblock {\em {Essentials of Computer Architecture}}.
\newblock Pearson Prentice Hall, New Jersey, 2005.

\bibitem{Hager2012}
Georg Hager, Jan Treibig, Johannes Habich, and Gerhard Wellein.
\newblock Exploring performance and power properties of modern multicore chips
  via simple machine models.
\newblock {\em Concurrency and Computation: Practice and Experience}, 2014.

\bibitem{Williams2009}
Samuel Williams, Andrew Waterman, and David Patterson.
\newblock Roofline: an insightful visual performance model for multicore
  architectures.
\newblock {\em Commun. ACM}, 52(4):65--76, Apr 2009.

\bibitem{intel-iaca}
{Intel Corp.}
\newblock {I}ntel {A}rchitecture {C}ode {A}nalyzer.
\newblock
  \url{http://software.intel.com/en-us/articles/intel-architecture-code-analyzer},
  2012.
\newblock Version: 2.1.

\bibitem{stengel-2015}
Holger Stengel, Jan Treibig, Georg Hager, and Gerhard Wellein.
\newblock Quantifying performance bottlenecks of stencil computations using the
  execution-cache-memory model.
\newblock In {\em Proceedings of the 29th International Conference on
  Supercomputing}, ICS '15, pages 207--216, New York, NY, USA, 2015. ACM.

\end{thebibliography}
